\documentclass[11pt]{article}
\usepackage{fullpage}
\RequirePackage[OT1]{fontenc}
\RequirePackage[colorlinks,citecolor=blue,urlcolor=blue,linkcolor=blue,linktocpage=true]{hyperref}
\RequirePackage{amsthm,amsmath}
\RequirePackage{natbib}

\usepackage{amssymb,bbm,tikz}
\usepackage{mathtools}
\mathtoolsset{showonlyrefs}
\usepackage{cases}
\usepackage{graphicx}
\usepackage{tabularx}
\usepackage{microtype}
\usepackage{enumitem}
\usepackage{url}
\usepackage{amsfonts, mathabx}

\usepackage{appendix}

\let\hat\widehat
\let\tilde\widetilde

\newtheorem{theorem}{Theorem}
\newtheorem{lemma}[theorem]{Lemma}
\newtheorem{corollary}[theorem]{Corollary}
\newtheorem{remark}[theorem]{Remark}

\newtheorem{proposition}[theorem]{Proposition}

\newtheorem*{definition*}{Definition}
\newtheorem*{remark*}{Remark}

\DeclareMathOperator*{\argmin}{argmin}

\makeatletter
\def\namedlabel#1#2{\begingroup
    #2%
    \def\@currentlabel{#2}%
    \phantomsection\label{#1}\endgroup
}
\makeatother

\usepackage[textsize=scriptsize]{todonotes}

\title{\bf Berry-Esseen Bounds for Projection Parameters and Partial Correlations With Increasing Dimension}
\author{
    Arun Kumar Kuchibhotla\thanks{Email: {\tt
            arunku@cmu.edu}.}
    \and
    Alessandro Rinaldo\thanks{Email: {\tt arinaldo@cmu.edu}}
    \and
    Larry Wasserman\thanks{Email: {\tt larry@stat.cmu.edu}.}
}
\begin{document}
\maketitle

{\centering
\vspace*{-0.5cm}
\textit{Carnegie Mellon University}\\
\par\bigskip
July 15, 2020
\par
}

\begin{abstract}
We provide finite sample bounds
on the Normal approximation to the law of
the least squares estimator of the projection parameters
normalized by the sandwich-based standard errors. Our results hold in the increasing dimension setting and under minimal assumptions on the data generating distribution. In particular, we do not assume a linear regression function and only require the existence of finitely many moments for the response and  the covariates.  
Furthermore, we construct  confidence sets 
for the projection parameters in the form of hyper-rectangles and establish finite sample bounds on their coverage and accuracy.
We provide analogous results for partial correlations among the entries of sub-Gaussian vectors. 
\end{abstract}



\section{Introduction}
Linear regression is a ubiquitous technique in applied statistics. In much of the classic and recent literature in regression, the theoretical study of ordinary least squares (OLS) estimation has focused primarily on the well-specified case, where the observations are obtained from a model postulating a linear regression function. Although widely studied in the statistical and econometric literature, the properties of the OLS estimator for inference in non-standard settings has gained attraction only relatively recently~\citep{Buja14,Buja16,Uniform:Kuch18}.

In this paper, we study the finite-sample theoretical properties of the OLS estimator, such as estimation error and approximation error to a normal distribution, under minimal assumptions on the data generating distribution and in high-dimensional settings in which the dimension $d$ of the covariates may grow with the sample size $n$ in such a way that $d = o(n)$. In particular, we focus on misspecified regression models in which the regression function is not linear.
Specifically, we adopt the standard regression setting,  in which we observe an i.i.d sample $(X_1,Y_1),\ldots, (X_n,Y_n)$ from an unknown distribution $P$ on $\mathbb{R}^d \times \mathbb{R}$, where $X_i$ is the $d$-dimensional vector of covariates and $Y_i$ the response variable for the $i$th sample point. 
 We are interested in providing inferential guarantees for the best linear approximation to the  regression function $x \in \mathbb{R}^d \mapsto \mathbb{E}[Y|X=x]$, which may take any form.
When the underlying  distribution $P$ admits a second moment, it is well known \citep[see, e.g.][]{Buja14} that, even in misspecified models and regardless of the true underlying relationship between the covariate and the response variables, the best (in $L_2$ sense) linear approximation to the regression function is well-defined. It is equal to the linear function  $x \mapsto x^\top \beta$, where $\beta \in \mathbb{R}^d$ is any solution to the optimization problem
\[
\beta = \argmin_{b \in \mathbb{R}^d}\, \mathbb{E}[(Y-b^T X)^2],
\]
with $(X,Y) \sim P$.
When the Gram matrix $\Sigma  = \mathbb{E}[X X^\top]$ of the covariate vector is invertible, the solution is unique and is given by the vector of {\it projection parameters}
\[
\beta = \Sigma^{-1}\Gamma, \quad \text{where} \quad  \Gamma = \mathbb{E}[YX]\in\mathbb{R}^d.
\]
Making inferential statements on $\beta$ in these {\it assumption-lean} settings \citep{Buja14}, i.e. with random covariates and without a true underlying linear model, is an exceedingly common task in applied regression. However, as elucidated in \cite{Buja14}, in this case it is necessary to apply appropriate modifications to standard theory and methods in order to obtain valid asymptotic conclusions, even in fixed-dimensional settings. In particular, it is essential to deploy the sandwich estimator \citep{White1980,Buja14} of the variance of the ordinary least squares estimators. 

In this paper we follow the {\it assumption-lean} framework put forward by \cite{Buja14}, \cite{kuchibhotla2018valid} and \cite{boot} and derive novel non-asymptotic inferential guarantees for the projection parameters that hold under minimal assumptions on the data generating distribution and in the high-dimensional regime of $d$ increasing in (but of smaller order than) $n$.
The main goals of this paper are to present (1) precise high-dimensional Berry-Esseen bounds 
for the Normal approximation to the law of
the OLS estimator
$\hat\beta$ of $\beta$ (normalized by the sandwich standard error) under weak assumptions,
(2) finite sample bounds on the
accuracy of the coverage of a simple and practicable class of confidence intervals for the entries of
 $\beta$,
and
(3) similar results for the 
partial correlations of the entries of sub-Gaussian vectors.
The confidence sets we consider are
hyper-rectangles, which
immediately yield
simultaneous
confidence intervals
for all the components of $\beta$.

To the best of our knowledge, our results provide the sharpest known rates for the projection parameters under arguably the weakest possible settings considered in the  literature.

\subsubsection*{Related Work}
Berry-Esseen bounds for M-estimators such as ordinary least squares is a seasoned topic in statistics. \cite{pfanzagl1973accuracy} and~\cite{paulauskas1996rates}, among others, have considered Berry-Esseen bounds for multivariate estimators albeit without explicit focus on the dependence on the dimension. 
Statistical inference for linear regression models based on central limit theorems in increasing dimensions is also a well-established topic in the statistical literature. 
In a series of paper, \cite{Portnoy84,Portnoy85,Portnoy86,Portnoy88} established 
various types of asymptotic normal approximations in increasing dimensions in a variety of settings. When applied to our problem, those results imply a scaling for the dimension of order $d = o(\sqrt{n})$, assuming a correctly specified model and arguably strong assumptions.
\cite{bickel1983bootstrapping}
showed  consistency of the bootstrap
when $d=o(\sqrt{n})$ assuming again a linear regression model, i.i.d. errors and deterministic covariates \cite[see also][]{mammen1989}. Under more general settings, but still postulating a correctly specified linear model, \cite{mammen1993} proposed the wild (aka multiplier) bootstrap strategy \citep[see also][]{liu1988} for linear contrasts and proved its consistency. \cite{He2000} \citep[see also][]{welsh1989} established component-wise asymptotic normality of regression parameters and of more general estimators in parametric models in increasing dimensions.

Recently, in a groundbreaking series of papers \cite{Cher13,chernozhukov2017detailed, chernozhukov2019improved} have obtained high-dimensional Berry-Esseen rates over hyper-rectangles and certain types of sparsely-convex sets exhibiting only a poly-logarithmic dependence on the dimension (see also~\citet[Section 2]{MR1115160}). These results, which also hold for the ordinary and multiplier bootstrap,   have been further extended by  
\cite{lopes2020central},~\cite{das2020central},~\cite{kuchibhotla2020high}, \cite{chernozhukov2020nearly}, and~\cite{deng2020slightly} (only for bootstrap).
They have seen applications in numerous statistical problems and especially in high-dimensional regression settings: see, e.g., \cite{zhang2014confidence}, \cite{wasserman2014berry}, \cite{10.1093/biomet/asu056},  \cite{doi:10.1146/annurev-economics-012315-015826}, \cite{zhang2017simultaneous}, \cite{test}, and \cite{boot}.  
The recent statistical literature has produced a variety of methods for constructing confidence sets for the individual regression parameters (or fixed contrast thereof) in high-dimensional settings, some based on the bootstrap. See, e.g., 
\cite{javanmard2014confidence}, \cite{javanmard2018}, \cite{ning2017},  \cite{zhu2018,doi:10.1080/01621459.2017.1356319}, \cite{cai2017confidence}, \cite{ren2015}, \cite{rajen.peter.2018} and \cite{peter.sarah.2015}. 
What sets the present paper apart from much of the existing  literature on the topic is the lack of the linearity and of the sparsity assumptions and, more generally, reliance on very weak conditions on the underlying data distribution, consistent with the {\it assumption-lean} approach. 

\cite{boot} tackled the same misspecified settings considered in this article and formulated general Berry-Esseen bounds for non-linear statistics 
in increasing dimensions.
When applied to the projection parameters $\beta$,
the resulting rates are sub-optimal, as they require
$d = o(n^{1/5})$.
For partial correlations,
\cite{wasserman2014berry}
obtain Berry-Esseen bounds in the increasing dimension case.
The current paper sharpens these bounds considerably and requires much weaker assumptions.

Under the {\it assumption-lean} settings, \cite{kuchibhotla2018valid} proposed the UPoSI methodology for constructing simultaneous confidence sets for the projection parameters of all possible submodels, which in turn is equivalent to post-selection inference control. For the special case of the saturated model, UPoSI implies a confidence set for $\beta$ with coverage guarantees under weaker scaling for the dimension than the one required by the present paper, though at the cost of larger volumes. We comment on the differences between our results and those of  \cite{kuchibhotla2018valid} below in Section \ref{section::conclusion}.



\subsubsection*{Summary of our Contributions}
The main
contributions of the paper can be summarized as follows:
\begin{itemize}
\item Theorem \ref{thm:Berry-Esseen-OLS} 
provides a Berry-Esseen bound
on the difference between
the law of $\hat\beta - \beta$ ---
normalized by the standard errors from the sandwich estimator ---
and an appropriate Gaussian distribution.
The result relies on the deterministic inequality of Theorem~\ref{thm:Basic-deter-ineq} for the least squares estimator, that hold without distributional assumptions 
on the data generating process and in particular remains true for non i.i.d. data,
\item
Theorems \ref{thm:asymptotic-scaling-CLT-explicit} and \ref{thm::berry-esseen} 
are our main results.
Assuming independence and additional mild conditions,
we bound the error terms in
Theorem \ref{thm:Berry-Esseen-OLS} 
to derive explicit Berry-Esseen rates where the dimension $d$, as well as other parameters of the underlying distribution (including the condition number of the Gram matrix and the number of finite moments of the response variable) are accounted for.
To that effect, we apply recent high-dimensional central limit results of \cite{kuchibhotla2020high} 
Considering  the case where only the dimension $d$ is allowed to change, and ignoring logarithmic terms for convenience, the rates we derive are vanishing provided that
$$
d = o\left(\min\left\{n^{1/2},\, n^{2(1-2/q_x)/3}\right\}\right)\quad\mbox{and}\quad \frac{qq_x}{q + q_x} \ge 4,
$$
where $q_x$ is the number of
finite moments of $X$ and $q$ is the number of finite moments of
$Y-X^T\beta$. 
If $q_x \ge 8$, this requirement reduces to the familiar scaling of $d = o(\sqrt{n})$, which has been found, among others, by \cite{Portnoy84,Portnoy85,Portnoy86,portnoy1987central,Portnoy88},~\cite{bickel1983bootstrapping}, ~\cite{mammen1993}~\cite{He2000} and~\cite{spokoiny2012parametric}, though with different settings, techniques  and assumptions. 
To the best of our knowledge,
ours is the sharpest result
in the mis-specified case and one that relies on the weakest assumptions on the data-generating distribution. Our finite sample bounds immediately yield practicable simultaneous confidence intervals for the projection parameters, constructed either through a simple Bonferroni or {\v{S}}id{\'a}k correction, or using the bootstrap (see Theorem~\ref{thm:multiplier-bootstrap-consistency}), a more laborious but sharper method. Furthermore, in all these cases, the length of the (simultaneous) confidence intervals for the entries of the $\beta$ can be of order $1/\sqrt{n}$, independently of the dimension.
It is noteworthy that the most favorable scaling requirement of $d = o(\sqrt{n})$ is also needed to ensure consistency (in the operator norm) of the sandwich estimator; see Lemma~\ref{lem:sandwich-variance-error-control}. We conjecture that such scaling cannot be weakened while retaining the parametric rate of $n^{-1/2}$ for accuracy.  

\item 
Leveraging these results and the mathematical relationship between partial correlations and projection parameters, in Theorem \ref{thm:Berry-Esseen-bound-partial-corr}
we derive a Berry-Esseen bound
for the $d\times d$ matrix of partial correlations
corresponding to a sub-Gaussian random vector $X\in\mathbb{R}^d$, which in turns yield simultaneous confidence intervals for the partial correlation parameters.

\end{itemize}

\subsection*{Problem Formulation and Notation}
Let $(X_1,Y_1), \ldots, (X_n,Y_n)$ be a sample of $n$ observations in $\mathbb{R}^{d} \times \mathbb{R}$, not necessarily independent nor identically distributed. If an intercept term is included in the regression fit, as it is customary, the first coordinate of each covariate vector $X_i$ is set to $1$. 
We seek to draw inference on the {\it projection parameter}
\[
\beta = \beta_n := \argmin_{\theta\in\mathbb{R}^d}\,\frac{1}{n}\sum_{i=1}^n \mathbb{E}[(Y_i - X_i^{\top}\theta)^2].
\]
If the matrix
\[
\Sigma_n : = n^{-1}\sum_{i=1}^n \mathbb{E}[X_iX_i^{\top}]
\]
is positive definite, the projection parameter  is well-defined and equal to $\Sigma_n^{-1} \Gamma_n$, where $\Gamma_n : = n^{-1}\mathbb{E} \left[ \sum_{i=1}^n X_i Y_i\right]$.
When the sample points satisfy the linear model
$Y_i = X_i^\top \beta^* + \xi_i$, $1\le i\le n$,
where $\mathbb{E}[\xi_i | X_i ] = 0$ for all $i$, then the projection parameter corresponds to the vector of linear coefficients, i.e. $\beta = \beta^*$. In this paper, we will \emph{not} posit any relationship between the vectors of covariates $X_i$ and the responses. In this case, the projection parameter lacks a direct interpretation. In the i.i.d. setting, if the response variable has finite second moment, then the projection parameter collects the coefficient of the $L_2$ projection of $Y_1$ into the linear space spanned by the coordinates of $X_1$, i.e. $X_1^\top \beta$. See~\cite{Buja14,Buja16} for a discussion on interpretation of $\beta$ in a mis-specified case.

The projection parameter is traditionally estimated using the ordinary least squares (OLS) estimator 
\[
\widehat{\beta} = \widehat{\beta}_n := \argmin_{\theta\in\mathbb{R}^d}\,\frac{1}{n}\sum_{i=1}^n (Y_i - X_i^{\top}\theta)^2,
\]
which corresponds to the plug-in estimator of $\beta$.
Letting 
\[
\widehat{\Sigma}_n := \frac{1}{n}\sum_{i=1}^n X_iX_i^{\top}\quad\mbox{and}\quad \widehat{\Gamma}_n := \frac{1}{n}\sum_{i=1}^n X_iY_i,
\]
and provided that $\widehat{\Sigma}_n$ is positive definite, the ordinary least squares estimator is well-defined and can be expressed as  
\[
\widehat{\beta}_n := \widehat{\Sigma}_n^{-1} \widehat{\Gamma}_n.
\]
\paragraph{Notation} For any $x\in\mathbb{R}^d$ and a positive-definite matrix $A\in\mathbb{R}^{d\times d}$, $\|x\|_A = \sqrt{x^{\top}A x}$ represents the scaled Euclidean norm. If $A$ is a squared matrix, $\mbox{diag}(A)$ is the diagonal matrix with diagonal elements matching those of $A$ and if, in addition, $A$ a positive definite (thus, a covariance matrix), we set $\mbox{Corr}(A) = (\mathrm{diag}(A))^{-1/2}A(\mathrm{diag}(A))^{-1/2}$ to be the corresponding correlation matrix.  We denote with $S^{d-1} := \{\theta\in\mathbb{R}^d:\,\theta^{\top}\theta = 1\}$ the unit sphere in $\mathbb{R}^d$.

\paragraph{Outline}
Section \ref{section::determiniswtic}
provides the deterministic CLT for 
the OLS estimator normalized by an estimated
standard deviation.
Section \ref{section::explicit}
treats the case of indpendent observations
and provides explicit rate constraints on
the growth of dimension $d$ with respect to
the sample size $n$.
Section~\ref{sec::confidence-sets-OLS} provides
explicit confidence sets for the projection 
parameter $\beta_n$; we describe three methods
based on Bonferroni, {\v{S}}id{\'a}k inequality
and wild/multiplier bootstrap.
Section \ref{section::partial}
derives similar results for partial correlations.
Concluding remarks and future directions are in
Section \ref{section::conclusion}.

\section{Central Limit Theorems using a Deterministic Inequality}
\label{section::determiniswtic}
In this section
we establish a Berry Esseen bound
for the joint law
of the entries of $\hat\beta-\beta$ divided elementwise by the estimated
standard errors.
The result is ``deterministic,'' as it does not hinge upon any distributional  assumptions on the sample.
The strategy is to first obtain a deterministic finite sample bound
for the magnitude of the difference between $\hat\beta-\beta$
and a sample average of natural quantities akin to evaluations of an influence function
(Theorem \ref{thm:Basic-deter-ineq}).
Then the effect of the randomness
due to the use of the standard errors
is bounded
by comparing to the average of the values of the influence function
divided by its true standard deviation
(Corollary \ref{cor:Max-Statistic-Correct-Scaling}).
This leads to the main result,
Theorem \ref{thm:Berry-Esseen-OLS}.
In the subsequent section we will describe minimal distributional assumptions
that will allow us to explicitly bound
the error terms in
Theorem \ref{thm:Berry-Esseen-OLS} and derive rates of consistency for the normal approximation.
The proofs of these results are in Appendix \ref{appendix:deterministic}.

To introduce our deterministic bound, we first define the matrix
\[
V_n := \mbox{Var}\left(\frac{1}{n}\sum_{i=1}^n X_i(Y_i - X_i^{\top}\beta)\right),
\]
which corresponds to the ``meat'' of the sandwich variance of the OLS estimator in a mis-specified linear models; see~\cite{Buja14}.
Notice that in the above expression the variance cannot be pushed inside the summation since the observations are not necessarily assumed to be independent. 

%


\begin{theorem}\label{thm:Basic-deter-ineq}
Assume $\Sigma_n$ and $V_n$ to be invertible and set 
\begin{equation}\label{eq:Dn}
\mathcal{D}_n^{\Sigma} := \|\Sigma_n^{-1/2}\widehat{\Sigma}_n\Sigma_n^{-1/2} - I_d\|_{\mathrm{op}}.
\end{equation}
Then
\[
\left\|\widehat{\beta} - \beta - \frac{1}{n}\sum_{i=1}^n \psi_i\right\|_{\Sigma_n V_n^{-1}\Sigma_n} ~\le~ 
\|V_n^{-1/2}\Sigma_n^{1/2}\|_{\mathrm{op}}\mathcal{D}_n^{\Sigma}\|\widehat{\beta} - \beta\|_{\Sigma_n}.
\]
where 
\[
\psi_i : = \Sigma^{-1}_nX_i(Y_i - X_i^{\top}\beta) \in \mathbb{R}^d, \quad i=1,\ldots,n.
\]
\end{theorem}


The previous  result immediately implies the following normalized point-wise bound, which also holds deterministically.

\begin{corollary}\label{cor:Max-Statistic-Correct-Scaling}
Under the same conditions of Theorem \ref{thm:Basic-deter-ineq}, we have that 
\[
\max_{1\le j\le d}\left|\frac{\widehat{\beta}_j - \beta_j - n^{-1}\sum_{i=1}^n \psi_{ij}}{\sqrt{(\Sigma_n^{-1}V_n\Sigma_n^{-1})_{jj}}}\right| ~\le~ \|V_n^{-1/2}\Sigma_n^{1/2}\|_{\mathrm{op}}\mathcal{D}_n^{\Sigma}\|\widehat{\beta} - \beta\|_{\Sigma_n}.
\]
where $\psi_{ij}$ represents the $j$-th coordinate of $\psi_i\in\mathbb{R}^d$ and $(\Sigma_n^{-1}V_n\Sigma_n^{-1})_{jj}$ is the $j$-th diagonal element of $\Sigma_n^{-1}V_n\Sigma_n^{-1}$.
\end{corollary}

Note that $\Sigma_n^{-1}V_n\Sigma_n^{-1}$ is the variance of $n^{-1}\sum_{i=1}^n \psi_i$ and hence the statistics defined in Corollary~\ref{cor:Max-Statistic-Correct-Scaling} are normalized by their standard deviation. Corollary~\ref{cor:Max-Statistic-Correct-Scaling} is a deterministic inequality and can be used to derive bounds on the Gaussian approximation for the maximum of ``$z$-statistics''. The argument is standard~\citep[see, e..g,  Corollary 10 of][for details]{paulauskas1996rates}: if $U$, $W$ and $R > 0$ are real-valued random variables such that $|U - W| \le R$ almost surely, then for \emph{any} $\varepsilon > 0$ and \emph{any} function $\Psi(\cdot)$,
\begin{equation}\label{eq:basic-identity-BE}
\begin{split}
\sup_{t\in\mathbb{R}}|\mathbb{P}(U \le t) - \Psi(t)| ~&\le~ \sup_{t\in\mathbb{R}}|\mathbb{P}(W \le t) - \Psi(t)|\\ 
&\qquad+ \mathbb{P}(R > \varepsilon) + \sup_{t\in\mathbb{R}}\,[\Psi(t + \varepsilon) -
\Psi(t-\varepsilon)].
\end{split}
\end{equation}

Each quantity on the right hand side has a natural interpretation. The first term quantifies how well the c.d.f. of $W$ is approximated by $\Psi$ in a uniform sense. The second term shows the magnitude of difference between $U$ and $W$, while the third term measures the distortion in $\Psi$ because of this difference and, whenever $\Psi$ is a c.d.f. captures  anti-concentration properties of the corresponding distribution.  
Then, a direct application of the inequality~\eqref{eq:basic-identity-BE} along with the bound in Corollary~\ref{cor:Max-Statistic-Correct-Scaling} yields a Gaussian approximation for $U = \max_{1\le j\le d}|\widehat{\beta}_j - \beta_j|/(\Sigma_n^{-1}V_n\Sigma_n^{-1})_{jj}^{1/2}$ by $W = \max_{1\le j\le d}|G_j|$, where   $G = (G_1,\ldots,G_d)^{\top}$ is a mean zero Gaussian vector such that
\begin{equation}\label{eq:cov.G}
\mbox{Var}(G) = \mbox{Corr}\left( \frac{1}{n} \sum_{i=1}^n \psi_i \right) = \mbox{Corr}\left( \Sigma_n^{-1}V_n\Sigma_n^{-1}  \right)
\end{equation}
and, naturally, $\Psi(t) = \mathbb{P}(\max_{1\le j\le d}|G_j| \le t)$, for $t \geq 0$.
In details, define
\begin{equation}\label{eq:mean-Gaussian-approximation}
\Delta_n := \sup_{t\ge0}\left|\mathbb{P}\left(\max_{1\le j\le d}\left|\frac{n^{-1}\sum_{i=1}^n\psi_{ij}}{\sqrt{(\Sigma_n^{-1}V_n\Sigma_n^{-1})_{jj}}}\right| \le t\right) - \mathbb{P}\left(\max_{1\le j\le d}|G_j| \le t\right)\right|,
\end{equation}
and further set
\[
\Phi_{AC} := \sup_{t\ge0,\varepsilon>0}\,\frac{1}{\varepsilon}\mathbb{P}\left(t - \varepsilon \le \max_{1\le j\le d}|G_j| \le t+\varepsilon\right),
\]
as the anti-concentration constant.  Using bounds derived by ~\cite{Naz03} as elucidated in \cite{chernozhukov2017detailed}, it follows that $\Phi_{AC} \le C\sqrt{\log(ed)}$ for a universal constant $C$ \citep[the dependence on the dimension can be improved using the results of][]{kuchibhotla2021high}). Finally, define the event 
\begin{equation}\label{eq:beta-error-event}
\mathcal{E}_{\beta,\eta_n} := \left\{\|V_n^{-1/2}\Sigma_n^{1/2}\|_{\mathrm{op}}\mathcal{D}_n^{\Sigma}\|\widehat{\beta} - \beta\|_{\Sigma_n} \le \eta_n\right\},
\end{equation}
where $\eta_n>0$ is a positive number, possibly depending on $n$.
Using Corollary~\ref{cor:Max-Statistic-Correct-Scaling} and~\eqref{eq:basic-identity-BE}, we obtain the following result.
\begin{proposition}\label{prop:asymptotic-scaling-CLT}
For any $\eta_n > 0$, we have
\begin{equation}\label{eq:asymptotic-scaling-CLT}
\begin{split}
&\sup_{t\ge0}\left|\mathbb{P}\left(\max_{1\le j\le d}\left|\frac{\widehat{\beta}_j - \beta_j}{\sqrt{(\Sigma_n^{-1}V_n\Sigma_n^{-1})_{jj}}}\right| \le t\right) - \mathbb{P}\left(\max_{1\le j\le d}|G_j| \le t\right)\right|\\
&\quad\le \Delta_n + \mathbb{P}(\mathcal{E}_{\beta,\eta_n}^c) + \eta_n\Phi_{AC}.
\end{split}
\end{equation}
\end{proposition}
Proposition~\ref{prop:asymptotic-scaling-CLT} is, of course, only of theoretical interest;  for practical inferential purposes, we need a stronger distributional approximation result with the true standard errors replaced by their estimators. Towards that end, let $\widehat{\Sigma_n^{-1}V_n\Sigma_n^{-1}}$ be an estimator of $\Sigma_n^{-1}V_n\Sigma_n^{-1}$ and consider the event
\begin{equation}\label{eq:key.event}
\begin{split}
\mathcal{E}_{\eta_n} &:= 
\mathcal{E}_{\beta,\eta_n}
~\cap~\left\{\max_{1\le j\le   d}
\left|\sqrt{\frac{(\Sigma_n^{-1}V_n\Sigma_n^{-1})_{jj}}{(\widehat{\Sigma_n^{-1}V_n\Sigma_n^{-1}})_{jj}}}- 1\right| \le \eta_n\right\}.
\end{split}
\end{equation}
The first event in the intersection leads to an upper bound on the right hand side of the inequality in Corollary~\ref{cor:Max-Statistic-Correct-Scaling}, while the second event enables one to replace $(\Sigma_n^{-1}V_n\Sigma_n^{-1})_{jj}$ by $(\widehat{\Sigma_n^{-1}V_n\Sigma_n^{-1}})_{jj}$.

Combining all the pieces, we have now arrived to  the following general Berry-Esseen bound for the normalized entries of the OLS estimator $\widehat{\beta}$. This is the main result of this section.

\begin{theorem}\label{thm:Berry-Esseen-OLS}
For any $\eta_n > 0$,
\begin{equation}\label{eq:deterministic-BE-bound}
\begin{split}
&\sup_{t\ge0}\left|\mathbb{P}\left(\max_{1\le j\le d}\left|\frac{\widehat{\beta}_j - \beta_j}{\sqrt{(\widehat{\Sigma_n^{-1}V_n\Sigma_n^{-1}})_{jj}}}\right| \le t\right) - \mathbb{P}\left(\max_{1\le j\le d}|G_j| \le t\right)\right|\\ 
&\quad\qquad\le 2\Delta_n + \mathbb{P}(\mathcal{E}_{\eta_n}^c) + C_n(\eta_n)\eta_n\Phi_{AC} + \frac{d}{n},
\end{split}
\end{equation}
where $\mathcal{E}_{\eta_n}$ is the event given in \eqref{eq:key.event} and $C_n(\eta_n) := 1 +  \eta_n + \sqrt{2\log(2n)}
$.
\end{theorem}

The above result builds upon the deterministic bound of Theorem~\ref{thm:Basic-deter-ineq}  and holds (possibly trivially) for any  data generating distributions. In particular, it does not require independence or identically distributed observations. Thus a Berry--Esseen type result for $\widehat{\beta}$ can now be proved under various assumptions of dependence among the observations, such as $m$-dependence or more general temporal dependence. Theorem~\ref{thm:Berry-Esseen-OLS} does not even require the sample size $n$ to be a fixed number; in case $n$ is a random variable independent of the data and $(X_i, Y_i)$ are identically distributed, then the $d/n$ term in~\eqref{eq:deterministic-BE-bound} is replaced by $d/\mathbb{E}[n]$ and $C_n(\eta_n)$ is replaced by 
$1 + \eta_n + \sqrt{2\log(2\mathbb{E}[n])}$. In principle, the bound of Theorem~\ref{thm:Basic-deter-ineq} allows for $n$ to be a stopping time with respect to the filtration generated by $(X_i, Y_i), i\ge1$ provided that appropriate bounds on the terms in equation  \eqref{eq:deterministic-BE-bound} that hold for stopping ties are available.

As mentioned before, we took $\Psi(t) = \mathbb{P}(\max_{1\le j\le d}|G_j| \le t)$ in identity~\eqref{eq:basic-identity-BE}, which is arguably a natural choice because $n^{-1}\sum_{i=1}^n \psi_{ij}/(\Sigma_n^{-1}V_n\Sigma_n^{-1})_{jj}^{1/2}$ converges in distribution to $G_j$ for each $1\le j\le d$. There are other choices for $\Psi(\cdot)$ that would lead to faster rates of convergence for $\Delta_n$, such as Edgeworth expansions and moment-matching distributions. Edgeworth expansions can be found in~\citet[Theorem 20.1]{bhattacharya2010normal}, but the dependence on the dimension here is not explicit. With moment-matching distributions one replaces the Gaussian vector $G$ by a different one which matches more than the first two moments of $n^{-1}\sum_{i=1}^n \psi_{ij}/(\Sigma_n^{-1}V_n\Sigma_n^{-1})_{jj}^{1/2}$; these can be found in~\cite{boutsikas2015penultimate} and~\cite{zhilova2016non}. We leave these refined Berry-Esseen bounds for future work.

The four terms appearing in the Berry-Esseen bound~\eqref{eq:deterministic-BE-bound} of Theorem~\ref{thm:Berry-Esseen-OLS} capture different types of approximations, both of deterministic and stochastic nature. Specifically, the quantity  $C_n(\eta_n)\eta_n$ is a bound on the linearization  error, appropriately measured in the $\| \cdot \|_{ \Sigma_n V^{-1}_n\Sigma_n }$ norm, stemming from replacing $\widehat{\beta} - \beta$ with its  linear approximation $n^{-1}\sum_{i=1}^n \psi_i$, and  $\Delta_n$  is its corresponding Berry-Esseen bound. The term $\mathbb{P}(\mathcal{E}_{\eta_n}^c)$ collects multiple types of estimation errors: for $\widehat{\Sigma_n^{-1}V_n\Sigma_n^{-1}}$, $\Sigma_n^{-1/2}\widehat{\Sigma}_n\Sigma_n^{-1/2}$ and $\widehat{\beta} - \beta$, which we will study in the next section assuming an i.i.d. setting. The presence of the anti-concentration constant $\Phi_{AC}$ is standard in high-dimensional Berry--Esseen type results \citep[see, e.g.][]{chernozhukov2017detailed}, and allows to separate the effect of the estimation errors from the choice of the value of the threshold $t$ to produce an approximate $1-\alpha$ nominal coverage. Finally, the term $d/n$ stems from the use of a Gaussian approximation for $\max_{j}|n^{-1}\sum_{i=1}^n \psi_{ij}|/{(\Sigma_n^{-1}V_n\Sigma^{-1})_{jj}}^{1/2}$. This is also the reason for the coefficient $2$ of $\Delta_n$ in~\eqref{eq:deterministic-BE-bound}; see the proof of Theorem~\ref{thm:Berry-Esseen-OLS} for details.

Given appropriate distributional assumptions on the random vectors $(X_i, Y_i), 1\le i\le n$, the right hand side of~\eqref{eq:deterministic-BE-bound} can be bounded as follows. For any $\eta_n > 0$, the quantity $\mathbb{P}(\mathcal{E}_{\eta_n}^c)$ is controlled using concentration inequalities for mean zero random vectors and random matrices. In the next section, we will derive such inequalities assuming i.i.d. observations and imposing rather weak assumptions on the distribution of the data, namely the existence of finitely many moments.
For data obeying certain types of dependence, useful concentration inequalities are given in~\cite{Liu13} and~\citet[Section 5 and Appendix B]{Uniform:Kuch18}. Then, choosing $\eta_n$ suitably so that $\mathbb{P}(\mathcal{E}_{\eta_n}^c)$ tends to zero as $n$ and $d$ increase yields that $C_n(\eta_n)\eta_n\Phi_{AC} = o(1)$. Finally, the quantity $\Delta_n$ is controlled using Berry-Esseen bounds for averages of mean zero random vectors with explicit dependence on the dimension. This can be accomplished in more than one way. For independent random vectors, optimal Gaussian approximation bounds holding uniformly over all convex sets as given in ~\cite{bentkus2003dependence} and~\cite{raivc2019multivariate} would imply the requirement that $d = o(n^{2/7})$. 
However, in our analysis we only require convergence to the standard Gaussian distribution over the smaller sub-class of all symmetric hyyper-rectangles. In this case, we apply the recent, sharp high-dimensional Berry--Esseen bounds of~\cite{kuchibhotla2020high}, which are valid under the minimal condition of existence of the third moments, allow for a poly-logarithmic dependence in the dimension $d$ and exhibit the optimal scaling of $n^{-1/2}$ in the sample size. Alternatively, the results of~\cite{chernozhukov2020nearly} could also be used; such bounds exhibit a slightly more favorable dependence in $d$ \citep[see also][]{das2020central} but requires the existence of fourth  moments. 
Finally, for dependent observations, the Berry-Eseeen bounds of \cite{ZhangWu17} would be relevant. 


\section{Explicit Rates in case of Independent Observations}
\label{section::explicit}

Theorem~\ref{thm:Berry-Esseen-OLS} in the previous section provides a bound on the difference between the distribution of OLS estimator to that of the Gaussian distribution without assuming any specific dependence structure on the observations $(X_i, Y_i), 1 \le i\le n$. In order to derive concrete rates from this result, it remains to construct an estimator $\widehat{\Sigma_n^{-1}V_n\Sigma_n^{-1}}$, bound the Gaussian approximation error $\Delta_n$  and control the term $\mathbb{P}(\mathcal{E}_{\eta_n}^c)$ for a suitable chosen $\eta_n > 0$. 
Below, we carry out this program assuming independent and identically distributed observations and in a high-dimensional framework in which the parameters of the data generating distribution, including its dimension, are allowed to vary with the sample size. In this case, letting $(X, Y)$ be identically distributed as the observations $(X_i, Y_i), 1\le i\le n$, we have that 
\[
\Sigma_n = \Sigma := n^{-1}\sum_{i=1}^n \mathbb{E}[X_iX_i^{\top}] = \mathbb{E}[XX^{\top}] 
\]
and
\begin{align*}
V_n &= \mbox{Var}\left(\frac{1}{n}\sum_{i=1}^n X_i(Y_i - X_i^{\top}\beta)\right)\\ 
~&=~ \frac{1}{n^2}\sum_{i=1}^n \mbox{Var}(X_i(Y_i - X_i^{\top}\beta)) ~\overset{(a)}{=}~ \frac{1}{n}\mathbb{E}[XX^{\top}(Y - X^{\top}\beta)].
\end{align*}
The first equality follows because $(X_i, Y_i), 1\le i\le n$ are independent and $\beta$ satisfies $\sum_{i=1}^n \mathbb{E}[X_i(Y_i - X_i^{\top}\beta)] = 0$. It is interesting to note that equality (a) holds only because $\mathbb{E}[X_i(Y_i - X_i^{\top}\beta)] = 0$ for all $i$, which does not follow if the observations are non-identically distributed.
Furthermore,  the matrix $\Sigma$ can be estimated by $\widehat{\Sigma}_n$ and the matrix $V_n$ by
\[
\widehat{V}_n := \frac{1}{n}\times\frac{1}{n}\sum_{i=1}^n X_iX_i^{\top}(Y_i - X_i^{\top}\hat{\beta})^2. 
\]
The final plug-in estimator of the
asymptotic variance $\Sigma_n^{-1}V_n\Sigma_n^{-1}$ is the classical
{\em sandwich estimator}~\citep{White1980,Buja14}:
\[
\widehat{\Sigma}_n^{-1}\widehat{V}_n\widehat{\Sigma}_n^{-1}.
\]
For notational convenience, set
\[
V := \mathbb{E}[XX^{\top}(Y - X^{\top}\beta)^2]\quad\mbox{and}\quad \widehat{V} := \frac{1}{n}\sum_{i=1}^n X_iX_i^{\top}(Y_i - X_i^{\top}\widehat{\beta})^2,
\]
so that $V_n = n^{-1}V$ and $\widehat{V}_n = n^{-1}\widehat{V}$.

We will derive our bounds using the following assumptions on the data generating distribution. In particular, our result will also hold assuming only moment conditions  on the response and covariates, which can therefore be heavy-tailed. This is a considerable weakening of the assumptions commonly used in the literature, where the response and covariates are often assumed to have moments of all order (i.e. to be sub-Gaussian). In contrast we only assume the response and covariates to have $q \geq 2 $ and $q_x \ge 4$ moments, respectively, whose values may in principle depend on the dimension $d$.  
\begin{description}
	\item[\namedlabel{eq:DGP}{(\textbf{DGP})}] The observations $(X_i, Y_i)\in\mathbb{R}^d\times\mathbb{R}, 1\le i\le n$ are independent and identically distributed (i.i.d.).
	\item[\namedlabel{eq:moments-errors}{(\textbf{E})($q$)}] There exists some $q \ge 2$ and a constant $K_y\in(0, \infty)$ such that
	\[
	\Big(\mathbb{E}[|Y_i - X_i^{\top}\beta|^q]\Big)^{1/q} ~\le~ K_y < \infty,\quad\mbox{for all}\quad 1\le i\le n. 
	\]
	\item[\namedlabel{eq:covariate-subGaussian}{(\textbf{X-SG})}] There exists a constant $K_x\in(0, \infty)$ such that
	\[
	\mathbb{E}\left[\exp\left(\frac{|u^{\top}\Sigma^{-1/2}X_i|^2}{K_x^2}\right)\right] \le 2,\quad\mbox{for all}\quad 1\le i\le n\mbox{ and }u\in S^{d-1}.
	\]
	\item[\namedlabel{assump:finite-moment-covariates}{(\textbf{X})($q_x$)}] There exists some $q_x \ge 4$ and a constant $K_x \ge 1$ such that
    \[
    \left(\mathbb{E}[|u^{\top}\Sigma^{-1/2}X_i|^{q_x}]\right)^{1/q_x} \le K_x,\quad\mbox{for all}\quad 1\le i\le n\mbox{ and }u\in S^{d-1}.
    \]
    \item[\namedlabel{assump:finite-moment-independence-covariates}{(\textbf{X-IND})($q_x$)}] There exists $q_x \ge 4$ and a constant $K_x \ge 1$ such that
    \[
    \left(\mathbb{E}[|u^{\top}\Sigma^{-1/2}X_i|^{q_x}]\right)^{1/q_x} \le K_x,\quad\mbox{for all}\quad 1\le i\le n\mbox{ and }u\in S^{d-1},
    \]
    and $Z_i := \Sigma^{-1/2}X_i, 1\le i\le n$ are vectors with independent random variables as their coordinates, i.e., if $Z_i = (Z_i(1), \ldots, Z_i(d))^{\top}$, then $Z_i(1), \ldots, Z_i(d)$ are independent random variables.
	\item[\namedlabel{eq:bounded-asymptotic-variance}{(\textbf{$\Sigma$-$V$})}] There exist constants $0 < \underline{\lambda} \le \overline{\lambda} < \infty$ such that
	\[
	\underline{\lambda} ~\le~ \lambda_{\texttt{min}}(\Sigma^{1/2} V^{-1}\Sigma^{1/2}) ~\le~ \lambda_{\texttt{max}}(\Sigma^{1/2}V^{-1}\Sigma^{1/2}) ~\le~ \overline{\lambda}.
	\]
\end{description}
We now provide some comments on these assumptions. Our results below can be modified to hold true even if the condition~\ref{eq:DGP}
is weakened by requiring the the observations to the independent and not necessarily identically distributed.
However, in this case, the parameter $\beta$ will depend on the data generating distributions in complicated ways and will only satisfy
$\sum_{i=1}^n \mathbb{E}[X_i(Y_i - X_i^{\top}\beta)] = 0,$
with no control on the expectation of individual summands. This leads to an impossibility in estimating the variance of $n^{-1}\sum_{i=1}^n X_i(Y_i - X_i^{\top}\beta)$, without further assumptions; see~\cite{Liu95} and~\citet[Proposition 3.5]{Bac16} for details. Condition~\ref{eq:moments-errors} requires the existence of $q$-th order moment of the ``errors'' $Y_i - X_i^{\top}\beta$ and may be verified by assuming the response variables $Y_i$'s' to only have a finite $q$-th order moment. Indeed, observe that
\[
\Big(\mathbb{E}[|Y - X^{\top}\beta|^q]\Big)^{1/q} \le \left(\mathbb{E}[|Y|^q]\right)^{1/q} + \left(\mathbb{E}[ |(\Sigma^{-1/2}X)^{\top}\Sigma^{1/2}\beta |^q]\right)^{1/q}.
\]
Now, because $0 \leq \mathbb{E}[(Y - X^{\top}\beta)^2] = \mathbb{E}[Y^2] - \beta^\top \Sigma \beta$, we have $\|\Sigma^{1/2}\beta\| \le (\mathbb{E}[Y^2])^{1/2}$ and hence
\[
\Big(\mathbb{E}[|Y - X^{\top}\beta|^q]\Big)^{\frac{1}{q}} \le \left(\mathbb{E}[|Y|^q]\right)^{\frac{1}{q}} + (\mathbb{E}[Y^2])^{1/2}\sup_{a\in S^{d-1}}\,\left(\mathbb{E}[|(\Sigma^{-1/2}X)^{\top}a|^q]\right)^{1/q}.
\]
Therefore, assuming that $(\mathbb{E}[|Y|^q])^{1/q} \le \overline{K}_x < \infty$ for some $\overline{K}_x$ along with condition~\ref{assump:finite-moment-covariates} (with $q_x \ge q$) implies condition~\ref{eq:moments-errors}. For the sake of readability, we do not make the dependence on the parameter $K_y$ in Condition~\ref{eq:moments-errors} explicit in our bounds, though it could be tracked through the constants in our proofs, allowing in principle for a dependence of $K_y$ on $d$.
Assumption~\ref{assump:finite-moment-covariates} is a finite moment assumption on the covariates and is a significant weakening of the sub-Gaussianity assumption~\ref{eq:covariate-subGaussian} commonly used in the literature. Assumption~\ref{assump:finite-moment-independence-covariates} is more restrictive than~\ref{assump:finite-moment-covariates} and is commonly used in the random matrix theory literature. In~\ref{assump:finite-moment-covariates} and~\ref{assump:finite-moment-independence-covariates}, we require $q_x \ge 4$ leading to $L_4$--$L_2$ equivalence in moments. This is used for an application of the results of~\cite{oliveira2013lower} in Proposition~\ref{prop:estimation-error-beta} to obtain a rate for $\| \widehat{\beta} - \beta \|_{\Sigma}$. It can be relaxed by instead making use of the results of~\cite{yaskov2015sharp}. All our results in this section are derived under~\ref{assump:finite-moment-covariates} or~\ref{assump:finite-moment-independence-covariates}. The condition number assumption~\ref{eq:bounded-asymptotic-variance} requires $\Sigma$ and $V$ to be of the ``same order'' and appears to be unavoidable.  Noting that
\[
V = \mathbb{E}[XX^{\top}(Y - X^{\top}\beta)^2] = \mathbb{E}\left[XX^{\top}\mathbb{E}[(Y - X^{\top}\beta)^2|X]\right],
\]
we see that condition~\ref{eq:bounded-asymptotic-variance} is satisfied if
\[
\overline{\lambda}^{-1} ~\le~ \inf_{x}\,\mathbb{E}[(Y - X^{\top}\beta)^2|X = x] ~\le~ \sup_{x}\mathbb{E}[(Y - X^{\top}\beta)^2|X = x] ~\le~ \underline{\lambda}^{-1},
\]
where $\inf_x$ and $\sup_x$ should be taken as the essential infimum and supremum with respect to the distribution of $X$. 
In particular, Condition~\ref{eq:bounded-asymptotic-variance} does not rule out the possibility of vanishing eigenvalues (in $n$ and/or $d$). Again, 
we do not make the dependence on $\underline{\lambda}$ and $\overline{\lambda}$ explicit in our rates. 

In the following, we will make use of the general inequalities derived in Section~\ref{section::determiniswtic} and the assumptions above to derive explicit rates in the Berry--Esseen bounds. We will first prove a bound on the error of the Gaussian approximation $\Delta_n$ defined in~\eqref{eq:mean-Gaussian-approximation}, which is a key component of the bounds in Proposition~\ref{prop:asymptotic-scaling-CLT} and Theorem~\ref{thm:Berry-Esseen-OLS}. Define
\[
\lambda_{\circ} := \lambda_{\min}\left(\mbox{Corr}(\Sigma^{-1}V\Sigma^{-1})\right),
\]
a quantity that will play a role in controlling the rate of convergence of $\widehat{\beta}$ to a Gaussian distribution \citep[see, e.g.][]{lopes2020central, kuchibhotla2020high}.
\begin{proposition}\label{prop:mean-Gaussian-approximation-explicit}
Suppose assumptions~\ref{eq:DGP} and~\ref{eq:bounded-asymptotic-variance} hold. Further, suppose~\ref{eq:moments-errors} and~\ref{assump:finite-moment-covariates} hold true with $1/q + 1/q_x \le 1/3$. Then there exists a universal constant $C$ such that
\begin{equation}\label{eq:mean-Gaussian-approx-explicit}
    \Delta_n ~\le~ 
    CK_x^3K_y^3\frac{d^{3/q_x}(\log n)^{5/2}\log(\lambda_{\circ}^3\sqrt{n})}{n^{1/2}\underline{\lambda}^{3/2}\lambda_{\circ}}.
\end{equation}
\end{proposition}
\begin{remark}
The condition $1/q_x + 1/q \le 1/3$ implies that $q_x \ge 3$ and hence the right hand side of~\eqref{eq:mean-Gaussian-approx-explicit} converges to zero whenever $d = o(\sqrt{n})$, up to poly-logarithmic factors in $n$. It is important to mention that a third moment condition of sort cannot be dispensed of. Indeed, the condition that $\mathbb{E}[|e_j^{\top}\psi_i|^3] < \infty$ for all $j$ is required  to achieve an $n^{-1/2}$ dependence on the sample size; this is well-known even in the univariate case. Because $\psi_i$ is the product of $X_i$ and $Y_i - X_i^{\top}\beta$, by H{\"o}lder's inequality, a finite third moment bound on $\psi_i$ is equivalent to  $1/q_x + 1/q \le 1/3$. This requirement can be further weakened to condition that $\min\{q_x, q\} \ge 3$, as long as assumption~\ref{eq:moments-errors} is modified to the conditional moment assumption: $\mathbb{E}[|Y_i - X_i^{\top}\beta|^q|X] \le K_y$, almost everywhere with respect to the distribution of  $X$. 
\end{remark}
Next, Proposition~\ref{prop:mean-Gaussian-approximation-explicit} can be combined with Proposition~\ref{prop:asymptotic-scaling-CLT} to finally produce a  Berry--Esseen bound for $\widehat{\beta} - \beta$ with the asymptotic variance scaling, as shown next.
\begin{theorem}\label{thm:asymptotic-scaling-CLT-explicit}
Suppose assumptions~\ref{eq:DGP} and~\ref{eq:bounded-asymptotic-variance} hold. Further, suppose~\ref{eq:moments-errors} and~\ref{assump:finite-moment-covariates} hold true with $1/q + 1/q_x \le 1/3$ and $q_x \ge 4$. If $d + 2\log(4n) \le n/(14K_x)^2$, then there exists a constant $C = C(q, q_x, K_x, K_y, \widebar{\lambda})$ such that
\begin{equation}\label{eq:asymptotic-scaling-CLT-explicit}
\begin{split}
&\sup_{t\ge0}\left|\mathbb{P}\left(\max_{1\le j\le d}\left|\frac{\widehat{\beta}_j - \beta_j}{\sqrt{(\Sigma_n^{-1}V_n\Sigma_n^{-1})_{jj}}}\right| \le t\right) - \mathbb{P}\left(\max_{1\le j\le d}|G_j| \le t\right)\right|\\
&\quad\le 
CK_x^3K_y^3\frac{d^{3/q_x}(\log n)^{5/2}\log(\lambda_{\circ}^3\sqrt{n})}{n^{1/2}\underline{\lambda}^{3/2}\lambda_{\circ}}
\\
&\qquad
+ C\log(en)\left\{\frac{d\log^4(n/d)}{\sqrt{n}} + \left(\frac{d^{3/2}}{n^{1-2/q_x}}\right)^{\max\{q_x, 6\}/8}\right\}.
\end{split}
\end{equation}
If, in addition,~\ref{assump:finite-moment-independence-covariates} holds true, then $(d^{3/2}/n^{1-2/q_x})^{\max\{q_x, 6\}/8}$ on the right hand side can be replaced by $(d^{1/2 + 2/q_x}/n^{1-2/q_x})^{\max\{q_x, 6\}/8}$.
\end{theorem}

Note that the right hand side of the Berry--Esseen bound of Theorem~\ref{thm:asymptotic-scaling-CLT-explicit} converges to zero as long as $d = o(\min\{n^{2(1-2/q_x)/3}, n^{1/2}\})$ (ignoring the $log(n)$ factors). This condition becomes $d = o(n^{1/2})$ for $q_x \ge 8$. 
If the covariates are sub-Gaussian/sub-exponential or just have $\log(n)$ number of moments, then Theorem~\ref{thm:asymptotic-scaling-CLT-explicit} only requires $q \ge 3\log(n)/(\log(n) - 3)$ many moments on the response. 
Under~\ref{assump:finite-moment-independence-covariates}, the dimension requirement reduces to $d = o(n^{1/2})$ for all $q_x \ge 4$. Note that the condition $1/q + 1/q_x \le 1/3$ offers a trade-off between the number of moments on covariates and response.

As mentioned before, a result of the type given in Theorem~\ref{thm:asymptotic-scaling-CLT-explicit} is mostly of theoretical importance: in practice, in order to build confidence intervals  it is also necessary to incorporate the estimated variances in the Berry--Esseen bound. To that effect, we deploy the general bounds given in  Theorem~\ref{thm:Berry-Esseen-OLS}. We remark that the assumptions of Theorem~\ref{thm:asymptotic-scaling-CLT-explicit} do not appear to be strong enough because we now need to control the error the in sandwich variance estimator, as stated in the event~\eqref{eq:key.event}. This is accomplished in the next result, which provides consistency rates for the sandwich variance estimator in high-dimensions and under mild moment conditions, and may be of independent interest. The proof is somewhat long and can be found in Appendix~\ref{appsubsec:sandwich-variance}.
Throughout, we will assume that conditions~\ref{assump:finite-moment-covariates} (or~\ref{assump:finite-moment-independence-covariates}) and~\ref{eq:moments-errors} are in effect with $q_x$ and $q$ such that
\[
q_{xy} ~:=~ \frac{qq_x}{q + q_x} = \left(\frac{1}{q} + \frac{1}{q_x} \right)^{-1} \ge 4,
\]
which in particular implies that $\min\{ q, q_x\}  \ge 4$.
In fact, the results in Appendix~\ref{appsubsec:sandwich-variance} are derived only assuming $q_{xy} \ge 2$.
\begin{lemma}\label{lem:sandwich-variance-error-control}
Suppose assumptions~\ref{eq:DGP} and~\ref{eq:bounded-asymptotic-variance} hold. Further, suppose~\ref{eq:moments-errors} and~\ref{assump:finite-moment-covariates} hold true with $q_{xy} \ge 4$. If $d + 2\log(4n) \le n/(14K_x)^2$, then there exists a constant $C = C(q, q_x, K_x, K_y, \widebar{\lambda}, \underline{\lambda})$ such that
\begin{equation}\label{eq:sandwich-error}
    \begin{split}
    &\sup_{\theta\in\mathbb{R}^d}\left|\frac{\theta^{\top}\widehat{\Sigma}_n^{-1}\widehat{V}_n\widehat{\Sigma}_n^{-1}\theta}{\theta^{\top}\Sigma^{-1}V\Sigma^{-1}\theta} - 1\right|\\ 
    &\le C\left[\frac{d^{1/2}}{n^{1/2 - 1/q_{xy}}} + \frac{d^{3/4}}{n^{3/4-3/(2q_x)}} + \frac{d^{1/3}}{n^{1/3}} + \frac{d(\log n)^{1/3}}{n^{2/3 - 4/(3q_x)}} + \frac{d^{2/3}(\log n)^{1/6}}{n^{1/2 - 2/(3q_x)}}\right],
    \end{split}
\end{equation}
with probability at least 
$$
1 - \frac{d^{q_{xy}/4}}{n^{q_{xy}/4-1/2}} - \frac{d^{q_x/8}}{n^{q_x/8-1/4}}- C\frac{d^{1/3}\log n}{n^{1/3}} - C\frac{d(\log n)^{4/3}}{n^{2/3 - 4/(3q_x)}} - C\frac{d^{2/3}(\log n)^{7/6}}{n^{1/2 - 2/(3q_x)}}.
$$
Here we require that the right hand side of~\eqref{eq:sandwich-error} is less than $1/21$.
\end{lemma}

\begin{remark}
The proof of Lemma~\ref{lem:sandwich-variance-error-control} establishes a concentration inequality for $\sup_{\theta\in\mathbb{R}^d}\left|\frac{\theta^{\top}\widehat{\Sigma}_n^{-1}\widehat{V}_n\widehat{\Sigma}_n^{-1}\theta}{\theta^{\top}\Sigma^{-1}V\Sigma^{-1}\theta} - 1\right|$.
The consistency rate in \eqref{eq:sandwich-error} was obtained by setting the largest possible scaling of $d$ with respect to $n$ that  ensures consistency while keeping all the other quantities fixed. This choice also affects the value for the lower bound of the probability of this event. Other choices are possible.
\end{remark}

    Lemma~\ref{lem:sandwich-variance-error-control} shows that the sandwich estimator is consistent for the asymptotic variance of $\widehat{\beta}$ provided that 
\[
d = o\left(\min\left\{n^{1 - 2/q_{xy}},\,n^{1 - 2/q_x},\,n,\,n^{2(1-2/q_x)/3},\,n^{3/4 - 1/q_x}\right\}\right),
\]
ignoring  $\log(n)$ factors.
Using the assumption that $q_{xy} \ge 4$, this condition reduces to $d = o(\min\{n^{1/2}, n^{2(1-2/q_x)/3}\})$ which becomes $d = o(n^{1/2})$ for $q_x \ge 8$. This requirement is the same as the one obtained in Theorem~\ref{thm:asymptotic-scaling-CLT-explicit}, but it should be emphasized that Theorem~\ref{thm:asymptotic-scaling-CLT-explicit} only requires $q_{xy} \ge 3$. If the assumption~\ref{assump:finite-moment-covariates} is replaced with~\ref{assump:finite-moment-independence-covariates}, then the requirement on the dimension becomes $d = o(n^{1/2})$ even when $q_{xy} > 2, q_{x} \ge 4$ 
This can be derived following the proof of Lemma~\ref{lem:sandwich-variance-error-control} and using the second result of Lemma~\ref{lem:M-4-bound}.

Using the above error bounds for the sandwich variance estimator from Lemma~\ref{lem:sandwich-variance-error-control} along with the Berry-Eseeen estimate of Theorem~\ref{thm:asymptotic-scaling-CLT-explicit}  immediately leads to the result of this section, a uniform Berry-Esseen bound for the studentized OSL estimator.

\begin{theorem}[Berry--Esseen bound under Independence]\label{thm::berry-esseen}
Assume conditions~\ref{eq:DGP},~\ref{eq:bounded-asymptotic-variance}. Suppose conditions~\ref{assump:finite-moment-covariates} and~\ref{eq:moments-errors} hold true with $q_{xy} \ge 4$. Further suppose the dimension requirements of Lemma~\ref{lem:sandwich-variance-error-control} also hold. Then there exists a constant $C = C(q, q_x, K_x, K_y, \widebar{\lambda}, \underline{\lambda})$ such that
\begin{align*}
&\sup_{t\ge0}\left|\mathbb{P}\left(\max_{1\le j\le d}\left|\frac{n^{1/2}(\widehat{\beta}_j - \beta_j)}{\sqrt{(\widehat{\Sigma}_n^{-1}\widehat{V}\widehat{\Sigma}_n^{-1})_{jj}}}\right| \le t\right) - \mathbb{P}\left(\max_{1\le j\le d}|G_j| \le t\right)\right|\\ 
&\le 
CK_x^3K_y^3\frac{d^{3/q_x}(\log n)^{5/2}\log(\lambda_{\circ}^3\sqrt{n})}{n^{1/2}\underline{\lambda}^{3/2}\lambda_{\circ}}
\\
&\qquad
+ C\log(en)\left[\frac{d\log^4(n/d)}{\sqrt{n}}+ \frac{d^{1/2}}{n^{1/2 - 1/q_{xy}}} + \frac{d^{3/4}}{n^{3/4-3/(2q_x)}}\right]
\\
&\qquad
+ C\log(en)\left[\frac{d^{1/3}}{n^{1/3}} + \frac{d(\log(en))^{1/3}}{n^{2/3 - 4/(3q_x)}} + \frac{d^{2/3}(\log(en))^{1/6}}{n^{1/2 - 2/(3q_x)}}\right].
\end{align*}
\end{theorem}

Ignoring the logarithmic terms in $d$ and $n$, the bound converges to zero if $q_{xy} \ge 4$ and 
\begin{equation}\label{eq:requirement-d-BE-OLS}
d = o\left(\min\left\{n^{1/2},\, n^{2(1-2/q_x)/3}\right\}\right),
\end{equation}
More generally, if $q_x \ge 8$ and $q_{xy} \ge 4$, then the requirement~\eqref{eq:requirement-d-BE-OLS} reduces to $d = o(\sqrt{n})$. This is a high-dimensional scaling that has been obtained, in different settings and based on more stringent assumptions, by several authors in the past, including~\cite{Portnoy84,Portnoy85,Portnoy86,portnoy1987central,Portnoy88},~\cite{He2000} and~\cite{spokoiny2012parametric}. We point out, however, the important difference that in these and related papers, the authors prove central limit theorems under a well-specified model (e.g., assuming a linear regression function) and for estimators normalized by their true but unknown variance. In contrast, we prove a finite-sample Berry-Esseen bound with an estimated variance. The requirement~\eqref{eq:requirement-d-BE-OLS} is under the mild  assumption~\ref{assump:finite-moment-covariates} on the covariates, which only calls for finite moments. If instead  the stronger Assumption~\ref{assump:finite-moment-independence-covariates} is in effect, then the requirement becomes $d = o(\sqrt{n})$ whenever $q_x \ge 4$ and $q_{xy} \ge 2$. Finally, it is worth emphasizing that we impose higher moment restrictions for the Berry--Esseen bound with estimated variance (Theorem~\ref{thm:Berry-Esseen-OLS}) compared to that with the asymptotic variance (Theorem~\ref{thm:asymptotic-scaling-CLT-explicit}).

Importantly,  Theorem~\ref{thm::berry-esseen} further yields that the length of the individual confidence intervals for the entries of $\beta$ can be of order $1/\sqrt{n}$, which amounts to a parametric accuracy rate, independent of the dimension. Indeed, for any fixed $t \geq 0$, the length of the interval for the $j$th coordinate is ${2tn^{-1/2}}(\widehat{\Sigma}_n^{-1}\widehat{V}\widehat{\Sigma}_n^{-1})_{jj}^{1/2}$, which is
\begin{align*}
&\frac{2t}{\sqrt{n}} \sqrt{({\Sigma}^{-1}V{\Sigma}^{-1})_{jj}}
+ \frac{2t}{\sqrt{n}} \left( \sqrt{(\widehat{\Sigma}_n^{-1}\widehat{V}\widehat{\Sigma}_n^{-1})_{jj}} - \sqrt{({\Sigma}^{-1}V{\Sigma}^{-1})_{jj}}  \right)\\ 
&= O \left( \sqrt{\frac{({\Sigma}^{-1}V{\Sigma}^{-1})_{jj}}{n}} \right),
\end{align*}
where the bound stems from Lemma~\ref{lem:sandwich-variance-error-control}, which  implies,  when $d = o(\sqrt{n})$ and since $\underline{\lambda}$, $\overline{\lambda}$, $q$ and $K_x$ are of constant order, that the difference inside the parenthesis in the above equation is vanishing in $n$.

\section{Confidence Sets for the Regression Parameters}\label{sec::confidence-sets-OLS}
In the previous section, we have proved a Gaussian approximation for the normalized least squares estimator of the projection parameter. To obtain confidence intervals for the coordinates of the projection parameter we further need to know or estimate the quantiles of $\max_{1\le j\le d}\,|G_j|$. In the current section, we describe three approaches. The first two methods do not come with any additional computational cost but are conservative, as they are based on Bonferroni and {\v{S}}id{\'a}k inequalities. The final way is asymptotically exact and makes use of the multiplier bootstrap.  

\subsection{Bonferroni and {\v{S}}id{\'a}k Method}\label{subsec:bonferroni.sidak}

We have proved that
\[
\sup_{t\ge 0}\left|\mathbb{P}\left(\max_{1\le j\le d}\left|\frac{n^{1/2}(\widehat{\beta}_j - \beta_j)}{\sqrt{(\widehat{\Sigma}_n^{-1}\widehat{V}\widehat{\Sigma}_n^{-1})_{jj}}}\right| \le t\right) - \mathbb{P}\left(\max_{1\le j\le d}|G_j| \le t\right)\right| \le r_n,
\]
for some rate $r_n$ and mean zero Gaussian random vector $G\in\mathbb{R}^d$ with unit variance on each coordinate. Taking $t = z_{\alpha/(2d)}$, where $z_{\gamma}$ represents the $(1-\gamma)$-th quantile of the standard Gaussian distribution, by symmetry and the union bound we get that
\begin{align*}
\mathbb{P}\left(\max_{1\le j\le d}\left|\frac{n^{1/2}(\widehat{\beta}_j - \beta_j)}{\sqrt{(\widehat{\Sigma}_n^{-1}\widehat{V}\widehat{\Sigma}_n^{-1})_{jj}}}\right| \le z_{\alpha/(2d)}\right) &\ge \mathbb{P}\left(\max_{1\le j\le d}|G_j| \le z_{\alpha/(2d)}\right) - r_n\\
&\ge (1 - \alpha) - r_n.
\end{align*}
Alternatively, we can sharpen the Bonferroni confidence regions by using instead  {\v{S}}id{\'a}k's inequality~\citep[Corollary 1]{vsidak1967rectangular}, which implies that, for all $t > 0$,
\[
\mathbb{P}\left(\max_{1\le j\le d}|G_j| \le t\right) \ge \prod_{j=1}^d \mathbb{P}(|G_j| \le t).
\]
Thus,
\begin{align*}
\mathbb{P}\left(\max_{1\le j\le d}\left|\frac{n^{1/2}(\widehat{\beta}_j - \beta_j)}{\sqrt{(\widehat{\Sigma}_n^{-1}\widehat{V}\widehat{\Sigma}_n^{-1})_{jj}}}\right| \le z_{(1 - (1-\alpha)^{1/d})/2}\right) &\ge (1 - \alpha) - r_n.
\end{align*}
For any $d \ge 1$, $\alpha\in[0, 1]$, we have $1 - (1 - \alpha)^{1/d} \ge \alpha/d$
and hence, $z_{(1- (1-\alpha)^{1/d})/2} \le z_{\alpha/(2d)}$. Thus, the confidence sets based on  {\v{S}}id{\'a}k's method are always smaller than the ones based on Bonferroni's adjustment and should be preferred. 
The preference of {\v{S}}id{\'a}k's method over Bonferroni's and its possible use have been discussed in~\cite{westfall1993resampling} and~\cite{drton2004model}.

As a final remark, we note that a naive application of both Bonferroni's or Sidak's correction combined with the one-dimensional Berry-Esseen bound will result in a worse scaling of $d$ with respect to $n$. 

\subsection{Bootstrap}
The confidence sets described in the previous section can be conservative because they do not take into account the correlation structure of $G = (G_1, \ldots, G_d)^{\top}$.  
Recall that $(G_1, \ldots, G_d)^{\top}$ has a normal distribution on $\mathbb{R}^d$ with mean zero and unknown covariance matrix given by $\mbox{corr}(\Sigma_n^{-1}V_n\Sigma_n^{-1})$. Hence one way to find the quantiles of $\max_{1\le j\le d}|G_j|$ is to generate Guassian random vectors from the distribution $$N_d(0, \mbox{corr}(\widehat{\Sigma}_n^{-1}\widehat{V}_n\widehat{\Sigma}_n^{-1}))$$ and use the sample quantiles of the maximum norm of these random vectors. This procedure is equivalent to the multiplier bootstrap methodology, detailed  below:
\begin{enumerate}
	\item Define the estimated ``score'' vectors 
	\[
	\widehat{\psi}_i := \widehat{\Sigma}^{-1}X_i(Y_i - X_i^{\top}\widehat{\beta})~\in~\mathbb{R}^d, \quad 1 \leq i \leq n.
	\]
	From the definition of $\widehat{\beta}$, we have $\sum_{i=1}^n \widehat{\psi}_i = 0.$
	\item Fix the number of bootstrap samples $B \ge 1$. For each $1 \leq b \leq B$, generate random vectors $e_i^{(b)}\overset{iid}{\sim} N(0, 1)$, $1\le i\le n$ and compute the bootstrap statistics
	\[
	T_b ~:=~ \max_{1\le j\le d}\frac{|n^{-1}\sum_{i=1}^n e_i^{(b)}\widehat{\psi}_{i,j}|}{\sqrt{(\widehat{\Sigma}_n^{-1}\widehat{V}_n\widehat{\Sigma}_n^{-1}})_{jj}}.
	\]
\end{enumerate}
Conditionally on the data $\mathcal{D}_n := \{(X_i, Y_i):\,1\le i\le n\}$, for each $b$ the vector
\[
\left(\frac{n^{-1}\sum_{i=1}^n e_i^{(b)}\widehat{\psi}_{i,j}}{\sqrt{(\widehat{\Sigma}_n^{-1}\widehat{V}_n\widehat{\Sigma}_n^{-1})_{jj}}}~:\, 1\le j\le d\right)^{\top}
\]
has a normal distribution with mean zero and variance given by $\mbox{corr}(\widehat{\Sigma}_n^{-1}\widehat{V}_n\widehat{\Sigma}_n^{-1})$. This follows from the fact that $$\widehat{\Sigma}_n^{-1}\widehat{V}_n\widehat{\Sigma}_n^{-1} ~=~ \frac{1}{n^2}\sum_{i=1}^n \widehat{\psi}_i\widehat{\psi}_i^{\top}.$$
The next result proves that the empirical distribution of $T_b, 1\le b\le B$ approximates the distribution of $\max_{1\le j\le d}|G_j|$ and hence provides an approximation to the distribution of $\max_{1\le j\le d}|n^{1/2}\widehat{\beta}_j - \beta_j|/(\widehat{\Sigma}_n^{-1}\widehat{V}\widehat{\Sigma}_n)_{jj}^{1/2}$. The proof uses a novel Gaussian comparison bound due to \cite{lopes2020central} and is given in Appendix~\ref{appendix:main.ols}.
\begin{theorem}[Consistency of Multiplier Bootstrap]\label{thm:multiplier-bootstrap-consistency}
Under the assumptions of Theorem~\ref{thm::berry-esseen}, for every $B \ge 1$, there exists a constant $C = C(q, q_x, K_x, K_y, \widebar{\lambda}, \underline{\lambda})$ such that
\begin{align*}
&\sup_{t\ge0}\,\left|\frac{1}{B}\sum_{b=1}^B \mathbbm{1}\{T_b \le t\} - \mathbb{P}\left(\max_{1\le j\le d}|G_j| \le t\right)\right|\\ 
&\;\le\sqrt{\frac{\log(2n)}{2B}}\\ &\;+ C\log(n)\left[\frac{d^{1/2}}{n^{1/2 - 1/q_{xy}}} + \frac{d^{3/4}}{n^{3/4-3/(2q_x)}} + \frac{d^{1/3}}{n^{1/3}} + \frac{d(\log n)^{1/3}}{n^{2/3 - 4/(3q_x)}} + \frac{d^{2/3}(\log n)^{1/6}}{n^{1/2 - 2/(3q_x)}}\right],
\end{align*}
with probability at least 
$$
1 - \frac{d^{q_{xy}/4}}{n^{q_{xy}/4-1/2}} - \frac{d^{q_x/8}}{n^{q_x/8-1/4}}- C\frac{d^{1/3}\log n}{n^{1/3}} - C\frac{d(\log n)^{4/3}}{n^{2/3 - 4/(3q_x)}} - C\frac{d^{2/3}(\log n)^{7/6}}{n^{1/2 - 2/(3q_x)}}.
$$
\end{theorem}


Comparing the above bootstrap bound with the Berry-Esseeen bound from Theorem~\ref{thm::berry-esseen}, we see that deploying the bootstrap procedure does not add additional requirements on the allowable scaling between $d$ and $n$.
We further remark that the usual consistency results for the bootstrap imply closeness of $\mathbb{P}(T_b \le t\big|\mathcal{D}_n)$ to $\mathbb{P}(\max_{1\le j\le d}|G_j| \le t)$, uniformly in $t$. In contrast Theorem~\ref{thm:multiplier-bootstrap-consistency} proves (uniform) closeness of the empirical distribution of multiplier bootstrap $t\mapsto B^{-1}\sum_{b=1}^B \mathbbm{1}\{T_b \le t\}$ to $\mathbb{P}(\max_{1\le j\le d}|G_j| \le t)$, with a rate depending on the number $B$ of the bootstrap repetitions.


\section{Berry--Esseen bounds for Partial Correlations}
\label{section::partial}

It is well known that the vector of projection parameters is related to the
partial correlations between $Y$ and each component of $X$ given
the other components.
This suggests that our results should generalize to
partial correlations.
In this section we confirm this intuition by deriving simultaneous confidence intervals for the partial correlation coefficients of a high-dimensional sub-Gaussian vector.
This is of interest since partial correlations
play an important role in graphical models: see, e.g., \cite{Lau96} and~\cite{drton2004model}. The results in this section sharpen complementary results in \cite{wasserman2014berry}.

Suppose $X_1, \ldots, X_n\in\mathbb{R}^d$ are identically distributed random vectors (but not required to be independent). Let $\Sigma$ denote the $d\times d$ covariance matrix of $X_i$ and let $\Omega = \Sigma^{-1}.$ Then the $d \times d$ matrix of partial correlations is the symmetric matrix given by $\Theta = (\theta_{jk})_{j,k=1,\ldots,d}$ where
\[
\theta_{jk} := -\frac{e_j^{\top}\Sigma^{-1}e_k}{\sqrt{(e_j^{\top}\Sigma^{-1}e_j)(e_k^{\top}\Sigma^{-1}e_k)}},
\]
and $e_1, \ldots, e_d$ denote the canonical basis of $\mathbb{R}^d$. A natural estimator of $\theta_{jk}$ is given by $\widehat{\theta}_{jk}$, defined as
\[
\widehat{\theta}_{jk} := -\frac{e_j^{\top}\widehat{\Sigma}^{-1}e_k}{\sqrt{(e_j^{\top}\widehat{\Sigma}^{-1}e_j)(e_k^{\top}\widehat{\Sigma}^{-1}e_k)}}\quad\mbox{where}\quad \widehat{\Sigma}_n := \frac{1}{n}\sum_{i=1}^n (X_i - \overline{X}_n)(X_i - \overline{X}_n)^{\top},
\]
with $\overline{X}_n$ representing the (sample) average of $X_1, \ldots, X_n$. Notice that in this section, $\Sigma$ and $\widehat{\Sigma}_n$ denote the true and the sample covariance matrices, respectively,  rather than the corresponding Gram matrices as in the previous sections.

Berry--Esseen bound for the partial correlation coefficients can be derived from arguments similar, albeit more involved,  to those used to prove the results in previous sections. We begin by establishing a basic linear representation result for partial correlations. To that effect, we define the intermediate covariance ``estimator'' as
\[
\widetilde{\Sigma} := \frac{1}{n}\sum_{i=1}^n (X_i - \mu_X)(X_i - \mu_X)^{\top},\quad\mbox{where}\quad \mu_X := \mathbb{E}[X_i].
\]
In fact, this is not an estimator because of the unknown vector $\mu_X$ in the definition. For notational convenience, and with a slight abuse of notation, we set
\begin{equation}\label{eq:D-sigma-notation}
\mathcal{D}_n^{\Sigma} ~:=~ \|\Sigma^{-1/2}\widehat{\Sigma}_n\Sigma^{-1/2} - I_d\|_{\mathrm{op}}. 
\end{equation}
It is important to realize that this quantity is different from the corresponding one defined in previous sections because $\widehat{\Sigma}_n$ is the sample covariance matrix. 
The following result mirrors Theorem~\ref{thm:Basic-deter-ineq} as it provides a linear approximation to the difference between the the true and estimated partial correlation coefficients in terms of influence-like functions. 
\begin{theorem}\label{thm:linear-expansion-partial-corr}
Under the assumption that $\mathcal{D}_n^{\Sigma} \le 1/2$, there exists a universal constant $C\in(0, \infty)$ such that
\[
\max_{1\le j \le k\le d}\left|\widehat{\theta}_{jk} - \theta_{jk} + \frac{1}{n}\sum_{i=1}^n \psi_{jk}(X_i)\right| \le C(\mathcal{D}_n^{\Sigma})^2 + \|\overline{X}_n - \mu_X\|_{\Sigma^{-1}}^2,
\]
 where 
\begin{equation}
\begin{split}
\frac{1}{n}\sum_{i=1}^n \psi_{jk}(X_i) ~&:=~ \frac{e_j^{\top}\Sigma^{-1}(\widetilde{\Sigma} - \Sigma)\Sigma^{-1}e_k}{\sqrt{(e_j^{\top}\Sigma^{-1}e_j)(e_k^{\top}\Sigma^{-1}e_k)}}\\ 
&\qquad- \frac{\theta_{jk}}{2}\left[\frac{e_j^{\top}\Sigma^{-1}(\widetilde{\Sigma} - \Sigma)\Sigma^{-1}e_j}{e_j^{\top}\Sigma^{-1}e_j} + \frac{e_k^{\top}\Sigma^{-1}(\widetilde{\Sigma} - \Sigma)\Sigma^{-1}e_k}{e_k^{\top}\Sigma^{-1}e_k}\right].
\end{split}
\end{equation}
\end{theorem}

\begin{remark}
The functions $\psi_{jk} \colon \mathbb{R}^d \rightarrow \mathbb{R}$, $1 \leq j \leq k \leq d$, are linear, since the right hand side of the last expression is an average, because $\widetilde{\Sigma}$ is itself an average.
Furthermore, each  $\psi_{jk}(X_i)$ has zero expectation.
\end{remark}

Using Theorem~\ref{thm:linear-expansion-partial-corr}, we derive a high-dimensional central limit approximation  for the properly normalized partial correlation coefficients~\eqref{eq:proper-normalized-partial-correlation}. Towards that end, we note that, for each $j \leq k$, the function  $\psi_{jk}(\cdot)$ from Theorem~\ref{thm:linear-expansion-partial-corr} can be written as
\begin{equation}\label{eq:def-psi-jk}
\begin{split}
\psi_{jk}(x) &= -\bigg\{a_j(x)a_k(x) - \mathbb{E}[a_j(X_1)a_k(X_1)]\bigg\}\\ 
&\qquad- \frac{\theta_{jk}}{2}\bigg\{a_j^2(x) + a_k^2(x) - \mathbb{E}[a_j^2(X_1) + a_k^2(X_1)]\bigg\},
\end{split}
\end{equation}
where
\begin{equation}\label{eq:ajx-function}
x \in \mathbb{R}^d \mapsto a_j(x) ~:=~ \frac{(x - \mu_X)^{\top}\Sigma^{-1}e_j}{(e_j^{\top}\Sigma^{-1}e_j)^{1/2}}.
\end{equation}
Now, for a fixed $x \in \mathbb{R}^d$, define the plug-in estimator of $a_j(x)$ as
\begin{equation}\label{eq:ajhatx-function}
\widehat{a}_j(x) ~:=~ \frac{(x - \widebar{X}_n)^{\top}\widehat{\Sigma}^{-1}e_j}{(e_j^{\top}\widehat{\Sigma}^{-1}e_j)^{1/2}}.
\end{equation}
In turn, this estimator leads to an estimator $\widehat{\psi}_{jk}(x)$ of $\psi_{jk}(x)$ by replacing $a_j(x)$ and $a_k(x)$ by $\widehat{a}_j(x)$ and $\widehat{a}_k(x)$, respectively.  Formally, for any $x \in \mathbb{R}$, we let
\begin{equation}\label{eq:def-widehat-psi-jk}
\begin{split}
\widehat{\psi}_{jk}(x) ~&:=~ -\bigg\{\widehat{a}_j(x)\widehat{a}_k(x) - \mathbb{E}_n[\widehat{a}_j(X)\widehat{a}_k(X)]\bigg\}\\ 
&\qquad- \frac{\widehat{\theta}_{jk}}{2}\bigg\{\widehat{a}_j^2(x) + \widehat{a}_k^2(x) - \mathbb{E}_n[\widehat{a}_j^2(X) + \widehat{a}_k^2(X)]\bigg\},
\end{split}
\end{equation}
where, for any arbitrary, possibly random, function $f \colon \mathbb{R}^d \rightarrow \mathbb{R}$,  we set $\mathbb{E}_n[f(X)] := n^{-1}\sum_{i=1}^n f(X_i)$. Because the asymptotic variance of $\sqrt{n}(\widehat{\theta}_{jk} - \theta_{jk})$ is $\mathbb{E}[\psi_{jk}^2(X_1)]$ and its plug-in estimator is $\sum_{i=1}^n \widehat{\psi}_{jk}^2(X_i)/n$, a proper normalization of the partial correlation coefficient is given by 
\[
\frac{\sqrt{n}(\widehat{\theta}_{jk} - \theta_{jk})}{\widehat{\zeta}_{jk}}, \quad \text{where} \quad 
\widehat{\zeta}_{jk}^2 ~:=~ {\frac{1}{n}\sum_{i=1}^n \widehat{\psi}_{jk}^2(X_i)}.
\]

We are now ready to state the main result of this section, a high-dimensional  Berry-Esseen bound for the partial correlations of sub-Gaussian vectors.In deriving this bound, we have taken extra care in exhibiting an explicit dependence on the minimal variance of the $\psi_{jk}(X_i)$'s.  We remark here that the sub-Gaussianity assumption can be relaxed to appropriate moment assumptions, and the calculations used to establish the results of Section~\ref{section::explicit} can be directly adapted. For simplicity, we refrain from pursuing these laborious generalizations. 

Define
\[
\zeta_{\min} := \min_{1 \leq j < k \leq d} \left(\mathbb{E}[\psi^2_{jk}(X_i)]\right)^{1/2}.
\]
\begin{theorem}\label{thm:Berry-Esseen-bound-partial-corr}
Suppose $X_1, \ldots, X_n$ are independent and identically distributed random vectors such that
\begin{equation}\label{eq:centered-subGaussian-x}
\mathbb{E}\left[\exp\left(\frac{|u^{\top}\Sigma^{-1/2}(X_i - \mu_X)|^2}{2K_x^2}\right)\right] \le 2\quad\mbox{for all}\quad u\in S^{d-1},
\end{equation}
for some constant $K_x\in(0, \infty).$ Further assume that $d \leq \sqrt{n}$. Then, there exists a constant $C = C(K_x, \zeta_{\min}) \in(0,\infty)$ depending only on $K_x$ and $\zeta_{\min}$  such that 
\begin{align*}
&\sup_{t\ge0}\left|\mathbb{P}\left(\max_{1\le j < k\le d}\left|\frac{\sqrt{n}(\widehat{\theta}_{jk} - \theta_{jk})}{\widehat{\zeta}_{jk}}\right| \le t\right) - \mathbb{P}\left(\max_{1\le j < k\le d}|G_{jk}| \le t\right)\right|\\ 
&\qquad\le C \left( \frac{(\log(d\vee n))^{5/6}}{n^{1/6}} + \eta_n\sqrt{\log(ed)} \right),
\end{align*}
for a mean zero Gaussian vector $(G_{jk})_{1\le j < k\le d}$ with covariance satisfying $\mathrm{cov}(G_{jk}, G_{j'k'}) = \mathrm{corr}(\psi_{jk}(X_1), \psi_{j'k'}(X_1))$ and where
\[
\eta_n := \frac{K_x^4}{\zeta_{\min}}\frac{d + \log n}{\sqrt{n}} ~+~ \left(\frac{K_x^8}{\zeta_{\min}^2} + \frac{K_x^6}{\zeta^3_{\min}}\right)\sqrt{\frac{(d + \log n)\log(dn)}{n}}.
\]

\end{theorem}

\begin{remark}
Unlike in Theorem~\ref{thm:linear-expansion-partial-corr}, in Theorem~\ref{thm:Berry-Esseen-bound-partial-corr}, we only consider the maximum over $1\le j < k \le d$ because, by construction, $\widehat{\theta}_{jj} = \theta_{jj} = 1$ for all $1\le j\le d$. 

Ignoring log terms, the upper bound on the distributional approximation holds true when $d = o(\sqrt{n})$. Should the  assumption of sub-Gaussianity be relaxed to moment bounds only, the requirement on $d$ will depend on the number of moments of the $X_i$'s.
\end{remark}

\begin{remark}\label{rem:fast-rate-partial-correlation}
Theorem~\ref{thm:Berry-Esseen-bound-partial-corr} provides an $n^{-1/6}$ rate of convergence in the central limit theorem. This can be improved to an $n^{-1/2}$ rate if we assume that a certain correlation matrix has a minimum eigenvalue bounded away from zero. Let $\lambda_{\star}^2$ be the smallest eigenvalue of the correlation matrix of the random vector $(\psi_{jk}(X_i)/\zeta_{jk}: 1\le j < k \le d)$. Then under the assumptions of Theorem~\ref{thm:Berry-Esseen-bound-partial-corr} can be improved to
\begin{align*}
&\sup_{t\ge0}\left|\mathbb{P}\left(\max_{1\le j < k\le d}\left|\frac{\sqrt{n}(\widehat{\theta}_{jk} - \theta_{jk})}{\widehat{\zeta}_{jk}}\right| \le t\right) - \mathbb{P}\left(\max_{1\le j < k\le d}|G_{jk}| \le t\right)\right|\\ 
&\qquad\le C \left( \frac{(\log(d))^{4}\log(n)}{n^{1/2}\lambda_{\star}^2} + \eta_n\sqrt{\log(ed)} \right),
\end{align*}
for $\eta_n$ defined in Theorem~\ref{thm:Berry-Esseen-bound-partial-corr}. 
See the proof of Theorem~\ref{thm:Berry-Esseen-bound-partial-corr} in Appendix~\ref{appendix:main.partial} for the proof of this improvement.
\end{remark}

\begin{remark}
Although Theorem~\ref{thm:Berry-Esseen-bound-partial-corr} is proved for independent random vectors, it is easy to obtain a result similar to Theorem~\ref{thm:Berry-Esseen-OLS} for arbitrary random vectors. 
\end{remark}

For inference, one needs to estimate the quantiles of $\max_{1\le j < k\le d}|G_{jk}|$ in order to produce simultaneous confidence intervals for the partial correlation coefficients. An easy solution is to simply apply Bonferroni's or {\v{S}}id{\'a}k's correction for multiple parameters -- in this case ${d \choose 2} $--  as described above in Section~\ref{subsec:bonferroni.sidak}.

Alternatively, (asymptotically) sharper results may be obtained with the multiplier bootstrap, described as follows:
\begin{enumerate}
\item Define the estimated ``score'' vectors $\widehat{\psi}_{jk}(X_i)$ as above. From the definition, it follows that $\sum_{i=1}^n \widehat{\psi}_{jk}(X_i) = 0$ for all $1\le j < k\le d$.
\item Fix the number of bootstrap samples $B \ge 1$. Generate random vectors $e_i^{(b)}\overset{iid}{\sim}N(0, 1)$, $1\le i\le n, 1\le b\le B$ and compute the bootstrap statistics
\[
T_b ~:=~ \max_{1\le j < k\le d}\frac{|n^{-1/2}\sum_{i=1}^n e_i^{b}\widehat{\psi}_{jk}(X_i)|}{\sqrt{n^{-1}\sum_{i=1}^n \widehat{\psi}_{jk}^2(X_i)}}.
\]
\end{enumerate}
The following result proves that the empirical distribution of $T_b, 1\le b\le B$ approximates the distribution of $\max_{1\le j < k\le d}|G_{jk}|$.
\begin{theorem}[Consistency of Multiplier Bootstrap]\label{eq:multplier-bootstrap-consistency-partial-corr}
Under the assumptions of Theorem~\ref{thm:Berry-Esseen-bound-partial-corr}, for every $B\ge1$, we have with probability at least $1 - C/n$,
\begin{align*}
&\sup_{t\ge0}\left|\frac{1}{B}\sum_{b=1}^B \mathbbm{1}\{T_b \le t\} - \mathbb{P}\left(\max_{1\le j < k\le d}|G_{jk}| \le t\right)\right|\\ 
&\quad\le \sqrt{\frac{\log(2n)}{2B}} + CK_x^{2}(\log d)^{{2}/{3}}\left(\frac{d + \log n}{n}\right)^{{1}/{6}},
\end{align*}
whenever the right hand side is less than 1.
\end{theorem}
In Theorem~\ref{eq:multplier-bootstrap-consistency-partial-corr}, we make all the assumptions used in Theorem~\ref{thm:Berry-Esseen-bound-partial-corr}, in particular, we also assume $d \le \sqrt{n}$. Once again the rate of convergence can be improved to $n^{-1/2}$ if $\lambda_{\star}$ (defined in Remark~\ref{rem:fast-rate-partial-correlation}) is bounded away from zero; see the proof of Theorem~\ref{eq:multplier-bootstrap-consistency-partial-corr} for details.

\section{Conclusion}
\label{section::conclusion}

We have provided explicit Berry-Esseen bounds
for mis-specified linear models. The bounds are
derived based on deterministic inequalities that do not
require any specific independence or dependence assumptions
on the observations. 
Explicit requirements on the growth of dimension $d$ 
in terms of the sample size $n$ are given for an asymptotic
normal approximation when the observations are independent
and identically distributed. 
The Berry--Esseen bounds as well as the bootstrap consistency guarantees here allow for construction of valid confidence sets
for the projection parameter provided that $d = o(\sqrt{n})$. 
Using the methods in
\cite{kuchibhotla2018valid}, confidence sets for the projection parameters can be constructed even for larger $d$ .
However, these sets are not rectangles and the projected
confidence intervals for individual projection parameters 
are much wider than those obtained here. The confidence 
intervals we obtained here have width of order $O(1/\sqrt{n})$,
whenever $d = o(\sqrt{n})$.

All the results are derived without any structural or sparsity assumptions on the projection parameter $\beta$ (and hence an \emph{assumption-lean} setting), unlike much of the recent literature on high-dimensional linear regression. Because we consider the ordinary least squares as our estimator, imposing any sparsity assumption on $\beta$ will not impact the final results; in particular, the estimator is still asympotically normal when centered at its target. If one uses different estimators designed to produce sparse estimates, then these estimators cannot be uniformly $\sqrt{n}$-consistent; see ~\cite{potscher2009confidence}. Further, if one applies debiasing~\citep{javanmard2014confidence,vandegeer2014asymptotically,zhang2014confidence} on the sparse estimator, then such an estimator has asymptotically the same behavior as the OLS estimator because the OLS estimator is semiparametrically efficient for the projection parameter $\beta$~\cite[Example 5]{Levit76}. 

In the following, we describe two interesting future directions.
Our confidence intervals have width of order $n^{-1/2}$ whenever $d = o(\sqrt{n})$. We believe this requirement on the dimension is optimal in order to obtain $n^{-1/2}$ width intervals. This conjecture is obtained from the results of~\cite{cai2017confidence}; the authors prove that the minimax width of a confidence intervals for individual coordinates in a $k$-sparse linear regression is $n^{-1/2} + k(\log d)/n$ whenever $k = O(\min\{d^{\gamma}, n/\log d\})$ (for $\gamma < 1/2$). We believe the correct formulation of the minimax width is $n^{-1/2} + k\log(ed/k)/n$ for all $1\le k\le d$, in which case taking $k = d$ yields the minimax width $n^{-1/2} + k/n$. This rate becomes $n^{-1/2}$ only when $d = O(\sqrt{n})$. This raises several interesting questions: ``What is the analogue of our results when $d \gg \sqrt{n}$? What kind of asymptotic distribution can one expect? What is a confidence set that works simultaneously for all $d \le n$? Does such a set still center at $\widehat{\beta}$?'' 
Secondly, it would be interesting to develop similar bounds for
other mis-specified parametric models such as a generalized linear
models (GLMs). The deterministic inequalities of~\cite{2018arXiv180905172K}
imply results similar to Theorem~\ref{thm:Basic-deter-ineq} and hence
a parallel set of results for GLMs could be obtained. Of course, it involves non-trivial calculations to derive sandwich consistency which is
relatively easy for linear models because the objective function is quadratic.  

\newpage

\begin{appendices}

\begin{center}
{\Large {\bf Appendix}}
\end{center}

\section{Proofs of the Results from Section~\ref{section::determiniswtic}}
\label{appendix:deterministic}

\begin{proof}[Proof of Theorem \ref{thm:Basic-deter-ineq}]
By optimality,  $\widehat{\beta}$, solves the normal equations $\widehat{\Sigma}_n\widehat{\beta} = \widehat{\Gamma}_n.$
Subtracting $\widehat{\Sigma}_n\beta\in\mathbb{R}^d$ from both sides, we get
$\widehat{\Sigma}_n(\widehat{\beta} - \beta) = \widehat{\Gamma}_n - \widehat{\Sigma}_n\beta,$
which is equivalent to
\[\textstyle
(\Sigma^{-1/2}_n\widehat{\Sigma}_n\Sigma_n^{-1/2})\Sigma^{1/2}_n(\widehat{\beta} - \beta) ~=~ \Sigma^{-1/2}_n(\widehat{\Gamma}_n - \widehat{\Sigma}_n\beta),
\]
since $\Sigma_n$ is invertible.
Adding and subtracting $\Sigma^{-1/2}_n (\widehat{\beta}  - \beta)$ on both sides further yields the identity
\[
\Sigma ^{1/2}_n\left[\widehat{\beta}  - \beta  - \Sigma ^{-1}_n(\widehat{\Gamma}_n  - \widehat{\Sigma}_n \beta )\right] ~=~ (I_d - \Sigma_n ^{-1/2}\widehat{\Sigma}_n \Sigma_n ^{-1/2})\Sigma_n ^{1/2}(\widehat{\beta}  - \beta ).
\]
Taking the Euclidean norm, we have that
\begin{equation}\label{eq:influence-error-bound}
\begin{split}
\|\widehat{\beta} - \beta - \Sigma_n^{-1}(\widehat{\Gamma}_n - \widehat{\Sigma}_n\beta)\|_{\Sigma_n} ~&=~ \|(I_d - \Sigma_n ^{-1/2}\widehat{\Sigma}_n \Sigma_n ^{-1/2})\Sigma_n ^{1/2}(\widehat{\beta}  - \beta )\|\\
~&\le~ \|I_d - \Sigma_n ^{-1/2}\widehat{\Sigma}_n \Sigma_n ^{-1/2}\|_{\mathrm{op}}\|\widehat{\beta}  - \beta\|_{\Sigma_n}\\
~&=~ \mathcal{D}_n^{\Sigma}\|\widehat{\beta} - \beta\|_{\Sigma_n}.
\end{split}
\end{equation}
To obtain a bound in the $\| \cdot \|_{\Sigma_n V_n^{-1}\Sigma_n}$ norm instead of the $\| \cdot \|_{\Sigma_n}$ norm, note that, for any $\theta\in\mathbb{R}^d$,
\[
\|\theta\|_{\Sigma_nV_n^{-1}\Sigma_n} = \|V_n^{-1/2}\Sigma_n\theta\| ~\le~ \|V_n^{-1/2}\Sigma_n^{1/2}\|_{\mathrm{op}}\|\theta\|_{\Sigma_n}.
\]
After substituting these inequalities in~\eqref{eq:influence-error-bound}
we arrive at the bound
\[
\|\widehat{\beta} - \beta - \Sigma_n^{-1}(\widehat{\Gamma}_n - \widehat{\Sigma}_n\beta)\|_{\Sigma_n V_n^{-1}\Sigma_n} ~\le~ 
\|V_n^{-1/2}\Sigma_n^{1/2}\|_{\mathrm{op}}\mathcal{D}_n^{\Sigma}\|\widehat{\beta} - \beta\|_{\Sigma_n}.
\]
This completes the proof.
\end{proof}

\begin{proof}[Proof of the Corollary \ref{cor:Max-Statistic-Correct-Scaling}]
Observe that, for any $x\in\mathbb{R}^{d}$ and any invertible matrix $A$,
\begin{align}\label{eq:Scaled-Euclidean-Maximum-Comparison}
\begin{split}
\|x\|_A = \|A^{1/2}x\| &= \max_{\theta\in\mathbb{R}^{d}}\frac{\theta^{\top}x}{\sqrt{\theta^{\top}A^{-1}\theta}}\\ &\geq \max_{\theta \in \{ \pm e_1,\ldots,\pm e_d \}} \frac{\theta^{\top}x}{\sqrt{\theta^{\top}A^{-1}\theta}} = \max_{1\le j\le d}\,\frac{|x_j|}{\sqrt{(A^{-1})_{jj}}}.
\end{split}
\end{align}
The result follows from Theorem~\ref{thm:Basic-deter-ineq}.
\end{proof}

\begin{proof}[Proof of Theorem \ref{thm:Berry-Esseen-OLS}]
On the event $\mathcal{E}_{\eta_n}$, we have that
\[
\|V_n^{-1/2}\Sigma_n^{1/2}\|_{\mathrm{op}}\mathcal{D}_n^{\Sigma}\|\widehat{\beta} - \beta\|_{\Sigma_n} \le \eta_n \quad\mbox{and}\quad \max_{1\le j\le d}\sqrt{\frac{(\Sigma_n^{-1}V_n\Sigma_n^{-1})_{jj}}{(\widehat{\Sigma_n^{-1}V_n\Sigma_n^{-1}})_{jj}}} \le 1 + \eta_n.
\]
Hence Corollary~\ref{cor:Max-Statistic-Correct-Scaling} yields
\[
\max_{1\le j\le d}\left|\frac{\widehat{\beta}_j - \beta_j - n^{-1}\sum_{i=1}^n \psi_{ij}}{\sqrt{(\widehat{\Sigma_n^{-1}V_n\Sigma_n^{-1}})_{jj}}}\right| \le \eta_n{(1 + \eta_n)}.
\]
Notice that the denominator involves an estimator of the ``asymptotic'' standard deviation. Therefore, on the event $\mathcal{E}_{\eta_n}$,
\begin{align*}
&\max_{1\le j\le d}\left|\frac{\widehat{\beta}_j - \beta_j}{\sqrt{(\widehat{\Sigma_n^{-1}V_n\Sigma_n^{-1}})_{jj}}} - \frac{n^{-1}\sum_{i=1}^n \psi_{ij}}{\sqrt{(\Sigma_n^{-1}V_n\Sigma_n^{-1})_{jj}}}\right|\\ 
&\qquad\le \eta_n(1 + \eta_n) + \max_{1\le j\le d}\left|\frac{n^{-1}\sum_{i=1}^n \psi_{ij}}{\sqrt{(\Sigma_n^{-1}V_n\Sigma_n^{-1})_{jj}}}\right|\times\max_{1\le j\le d}\left|\sqrt{\frac{(\Sigma_n^{-1}V_n\Sigma_n^{-1})_{jj}}{(\widehat{\Sigma_n^{-1}V_n\Sigma_n^{-1}})_{jj}}} - 1\right|\\
&\qquad\le \eta_n(1 + \eta_n) + \eta_n\max_{1\le j\le d}\left|\frac{n^{-1}\sum_{i=1}^n \psi_{ij}}{\sqrt{(\Sigma_n^{-1}V_n\Sigma_n^{-1})_{jj}}}\right|
\end{align*}
By standard Gaussian concentration and a union bound,
\[
\mathbb{P}\left(\max_{1\le j\le d}|G_j| \ge \sqrt{2\log(2 n  )}\right) \le \frac{d}{n}.
\]
Also,  the definition of $\Delta_n$ implies that
\begin{align*}
&\mathbb{P}\left(\max_{1\le j\le d}\left|\frac{n^{-1}\sum_{i=1}^n \psi_{ij}}{\sqrt{(\Sigma_n^{-1}V_n\Sigma_n^{-1})_{jj}}}\right| \ge \sqrt{2\log(2n)}\right)\\ 
&\quad\le \mathbb{P}\left(\max_{1\le j\le d}|G_j| \ge \sqrt{2\log(2n)}\right) + \Delta_n \le \frac{d}{n} + \Delta_n.
\end{align*}
Combining  these inequalities we obtain  that that there exists an event of probability at least $1 - \Delta_n - d/n - \mathbb{P}(\mathcal{E}_{\eta_n}^c)$ such that, on that event, it holds that 
\begin{align*}
&\max_{1\le j\le d}\left|\frac{\widehat{\beta}_j - \beta_j}{\sqrt{(\widehat{\Sigma_n^{-1}V_n\Sigma_n^{-1}})_{jj}}} - \frac{n^{-1}\sum_{i=1}^n \psi_{ij}}{\sqrt{(\Sigma_n^{-1}V_n\Sigma_n^{-1})_{jj}}}\right|\\ 
&\qquad \leq  \eta_n\times \left[ 1 +  \eta_n + \sqrt{2\log(2n)} \right] = \eta_n C_n(\eta_n),
\end{align*}
by the definition of $C(\eta_n)$.
 Hence for any $t \ge 0$ and $\eta_n > 0$, we have that
\begin{align*}
&\mathbb{P}\left(\max_{1\le j\le d}\left|\frac{\widehat{\beta}_j - \beta_j}{\sqrt{(\widehat{\Sigma_n^{-1}V_n\Sigma_n^{-1}})_{jj}}}\right| \le t\right)\\ 
&\quad\le \mathbb{P}\left(\max_{1\le j\le d}\left|\frac{n^{-1}\sum_{i=1}^n \psi_{ij}}{\sqrt{(\Sigma_n^{-1}V_n\Sigma_n^{-1})_{jj}}}\right| \le t + C_n(\eta_n)\eta_n\right)
+ \mathbb{P}(\mathcal{E}_{\eta_n}^c) + d/n + \Delta_n\\
&\mathbb{P}\left(\max_{1\le j\le d}\left|\frac{\widehat{\beta}_j - \beta_j}{\sqrt{(\widehat{\Sigma_n^{-1}V_n\Sigma_n^{-1}})_{jj}}}\right| \le t\right)\\ &\quad\ge \mathbb{P}\left(\max_{1\le j\le d}\left|\frac{n^{-1}\sum_{i=1}^n \psi_{ij}}{\sqrt{(\Sigma_n^{-1}V_n\Sigma_n^{-1})_{jj}}}\right| \le t - C_n(\eta_n)\eta_n\right)
- \mathbb{P}(\mathcal{E}_{\eta_n}^c) - d/n - \Delta_n.
\end{align*}
Next, from the definition of $\Delta_n$ and $\Phi_{AC}$, we finally obtain that
\begin{align*}
&\mathbb{P}\left(\max_{1\le j\le d}\left|\frac{\widehat{\beta}_j - \beta_j}{\sqrt{(\widehat{\Sigma_n^{-1}V_n\Sigma_n^{-1}})_{jj}}}\right| \le t\right)\\ &\quad\le \mathbb{P}\left(\max_{1\le j\le d}|G_j| \le t + C_n(\eta_n)\eta_n\right) + 2\Delta_n + \mathbb{P}(\mathcal{E}_{\eta_n}^c) + \frac{d}{n}\\
&\quad\le \mathbb{P}\left(\max_{1\le j\le d}|G_j| \le t\right) + 2\Delta_n + \mathbb{P}(\mathcal{E}_{\eta_n}^c) + C_n(\eta_n)\eta_n \Phi_{AC} + \frac{d}{n}\\
&\mathbb{P}\left(\max_{1\le j\le d}\left|\frac{\widehat{\beta}_j - \beta_j}{\sqrt{(\widehat{\Sigma_n^{-1}V_n\Sigma_n^{-1}})_{jj}}}\right| \le t\right)\\ &\quad\ge \mathbb{P}\left(\max_{1\le j\le d}|G_j| \le t - C_n(\eta_n)\eta_n\right) - 2\Delta_n - \mathbb{P}(\mathcal{E}_{\eta_n}^c) - \frac{d}{n}\\
&\quad\ge \mathbb{P}\left(\max_{1\le j\le d}|G_j| \le t\right) - 2\Delta_n - \mathbb{P}(\mathcal{E}_{\eta_n}^c) - C_n(\eta_n)\eta_n \Phi_{AC} - \frac{d}{n},
\end{align*}
as claimed.
\end{proof}

\section{Proof of the Main Results from Sections~\ref{section::explicit} and~\ref{sec::confidence-sets-OLS} (Projection Parameters)}
\label{appendix:main.ols}

\begin{proof}[Proof of Proposition~\ref{prop:mean-Gaussian-approximation-explicit}]
Theorem 1 of~\cite{kuchibhotla2020high} implies that
\begin{equation}\label{eq:Delta-n-first-bound}
\Delta_n \le C\frac{(\log n)^{5/2}}{\sqrt{n}\lambda_{\circ}}\log(\lambda_{\circ}^3\sqrt{n})\mathbb{E}\left[\max_{1\le j\le d}\,\frac{|\psi_{ij}|^3}{(\Sigma^{-1}V\Sigma^{-1})^{3/2}}\right],
\end{equation}
where $C$ is a universal constant.
We now bound the expectation in~\eqref{eq:Delta-n-first-bound}. Recall that $\psi_i = \Sigma^{-1}X_i(Y_i - X_i^{\top}\beta).$ 
\begin{align*}
\mathbb{E}\left[\max_{1\le j\le d}\,\frac{|\psi_{ij}|^3}{(\Sigma^{-1}V\Sigma^{-1})_{jj}^{3/2}}\right] &= \mathbb{E}\left[\max_{1\le j\le d}\left(\frac{|e_j^{\top}\Sigma^{-1}X_i|}{\sqrt{e_j^{\top}\Sigma^{-1}V\Sigma^{-1}e_j}}\right)^{3}|Y_i - X_i^{\top}\beta|^3\right]\\
&\le \frac{1}{\underline{\lambda}^{3/2}}\mathbb{E}\left[\max_{1\le  j\le d}\left|a_j^{\top}\Sigma^{-1/2}X_i\right|^3|Y_i - X_i^{\top}\beta|^3\right],
\end{align*}
where $a_j = \Sigma^{-1/2}e_j/(e_j^{\top}\Sigma^{-1}e_j)^{1/2}\in S^{d-1}$; the last inequality follows from assumption~\ref{eq:bounded-asymptotic-variance}. H{\"o}lder's inequality implies that whenever $1/q_x + 1/q \le 1/3$, we obtain
\begin{align*}
&\mathbb{E}\left[\max_{1\le  j\le d}\left|a_j^{\top}\Sigma^{-1/2}X_i\right|^3|Y_i - X_i^{\top}\beta|^3\right]\\ 
&\le \left(\mathbb{E}\left[\max_{1\le j\le d}\left|a_j^{\top}\Sigma^{-1/2}X_i\right|^{q_x}\right]\right)^{3/q_x}\left(\mathbb{E}[|Y_i - X_i^{\top}\beta|^{q}]\right)^{3/q}\\
&\le \left(\sum_{j=1}^d\mathbb{E}[|a_j^{\top}\Sigma^{-1/2}X_i|^{q_x}]\right)^{3/q_x}\left(\mathbb{E}[|Y_i - X_i^{\top}\beta|^{q}]\right)^{3/q}
\le d^{3/q_{x}}K_x^3K_y^3.
\end{align*}
Substituting this in~\eqref{eq:Delta-n-first-bound} yields
\[
\Delta_n \le C\frac{(\log n)^{5/2}}{\sqrt{n}\lambda_{\circ}}\log(\lambda_{\circ}^3\sqrt{n})\frac{d^{3/q_x}K_x^3K_y^3}{\underline{\lambda}^{3/2}}.
\]
This completes the proof.
\end{proof}

\begin{proof}[Proof of Theorem~\ref{thm:asymptotic-scaling-CLT-explicit}]
We use Proposition~\ref{prop:asymptotic-scaling-CLT}. The term $\Delta_n$ can be bounded as in Proposition~\ref{prop:mean-Gaussian-approximation-explicit}. So, we need to choose a suitable $\eta_n$ and bound the probability $\mathbb{P}(\mathcal{E}_{\beta,\eta_n}^{c})$, where the event $\mathcal{E}_{\beta,\eta_n}^{c}$ is given in \eqref{eq:beta-error-event}. Firstly, note that
\[
\|V_n^{-1/2}\Sigma_n^{1/2}\|_{op} = n^{1/2}\|V^{-1/2}\Sigma^{1/2}\|_{op} \le \overline{\lambda}^{1/2}{n^{1/2}},
\]
with the last inequality obtained from Assumption~\ref{eq:bounded-asymptotic-variance}. Hence, to bound $\mathbb{P}(\mathcal{E}_{\beta,\eta_n}^c)$ it suffices to find $\eta_n$ that upper bounds $n^{1/2}\mathcal{D}_n^{\Sigma}\|\widehat{\beta} - \beta\|_{\Sigma}$ with ``high'' probability.
By Propositions~\ref{prop:D_Sigma-concentration} and~\ref{prop:estimation-error-beta} (in Appendix~\ref{appendix:auxiliary.ols}), we get that with probability at least $1 - \delta - n^{-1/2}$,
\begin{equation}\label{eq:implication-proposition-D_sigma-OLS}
\begin{split}
    \mathcal{D}_n^{\Sigma} ~&\le~ C_{q_x}K_x^2\left\{\frac{d\delta^{-2/q_x}}{n^{1-2/q_x}} + \sqrt{\frac{d}{n}}\log^4\left(\frac{n}{d}\right)\right\},\\
    \|\widehat{\beta} - \beta\|_{\Sigma} ~&\le~ C_{q,q_x}\sqrt{\frac{d}{n}}\left[\sqrt{\frac{\log(en)}{\underline{\lambda}}} + {K_xK_y}\right].
\end{split}
\end{equation}
Above, the first inequality follows from  Propositions~\ref{prop:D_Sigma-concentration} by setting $q_x \ge 4$, which implies that $(d/n)^{1-2/q_x} \le \sqrt{d/n}$.
For the second inequality in~\eqref{eq:implication-proposition-D_sigma-OLS}, we use Proposition~\ref{prop:estimation-error-beta} with the choice of $\delta = n^{-qq_x/(2q+2q_x) + 1} \le n^{-1/2}$, because of the assumption that $1/q + 1/q_x \le 1/3$. 
Take $\delta = (d^{3/2}/n^{1-2/q_x})^{q_x/8}$ in~\eqref{eq:implication-proposition-D_sigma-OLS} and set
\[
\eta_n := C_{q,q_x}\widebar{\lambda}^{1/2}K_x^2\left[\sqrt{\frac{\log(en)}{\underline{\lambda}}} + {K_xK_y}\right]\left\{\frac{d}{\sqrt{n}}\log^4\left(\frac{n}{d}\right) + \left(\frac{d^{3/2}}{n^{1-2/q_x}}\right)^{3/4}\right\}.
\]
to ensure that $\mathbb{P}(\mathcal{E}_{\beta,\eta_n}^c) \le (d^{3/2}/n^{1-2/q_x})^{q_x/8} + n^{-1/2}$. Thus the result now follows from Proposition~\ref{prop:asymptotic-scaling-CLT}. If the stronger condition ~\ref{assump:finite-moment-independence-covariates}  holds true, then in the bound on $\mathcal{D}_n^{\Sigma}$, $d\delta^{-2/q_x}/n^{1-2/q_x}$ can be replaced by $d^{2/q_x}\delta^{-2/q_x}/n^{1-2/q_x}$. This follows from Proposition~\ref{prop:D_Sigma-concentration}. Then the same calculation detailed  above yields the result under~\ref{assump:finite-moment-independence-covariates}.
\end{proof}

\begin{proof}[Proof of Lemma~\ref{lem:sandwich-variance-error-control}]
We will use the deterministic inequality presented in Lemma~\ref{lem:std-err-bound-deterministic}, which requires controlling the quantities  $\mathcal{M}_4^{1/2}\|\widehat{\beta}-\beta\|_{\Sigma}$, $\mathcal{D}_n^{\Sigma}$, and $\mathbf{I}_1$ defined there. First, we control the term $\mathcal{M}_4^{1/2}\|\widehat{\beta}-\beta\|_{\Sigma}$.
As in the proof of Theorem~\ref{thm:asymptotic-scaling-CLT-explicit}, we have with probability at least $1 - 1/n^{q_{xy}/2 - 1}$ (which is at least $1 - 1/n$ if $q_{xy} \ge 4$),
\[
\|\widehat{\beta} - \beta\|_{\Sigma} \le C_{q,q_x}\sqrt{\frac{d}{n}}\left[\sqrt{\frac{\log(en)}{\underline{\lambda}}} + K_xK_y\right].
\]
Observe that
\begin{align*}
&\mathbb{P}\left(\mathcal{M}_4^{1/2}\|\widehat{\beta} - \beta\|_{\Sigma} \ge t\right)\\ 
&\le \mathbb{P}\left(\mathcal{M}_4^{1/2}\|\widehat{\beta} - \beta\|_{\Sigma} \ge t, \|\widehat{\beta} - \beta\|_{\Sigma} \le C_{q,q_x}\sqrt{\frac{d}{n}}\left[\sqrt{\frac{\log(en)}{\underline{\lambda}}} + K_xK_y\right]\right) \\
&\quad+ \mathbb{P}\left(\|\widehat{\beta} - \beta\|_{\Sigma} > C_{q,q_x}\sqrt{\frac{d}{n}}\left[\sqrt{\frac{\log(en)}{\underline{\lambda}}} + K_xK_y\right]\right)\\
&\le \mathbb{P}\left(C_{q,q_x}\sqrt{\frac{d}{n}}\left[\sqrt{\frac{\log(en)}{\underline{\lambda}}} + K_xK_y\right]\mathcal{M}_4^{1/2} \ge t\right) + \frac{1}{n^{q_{xy}/2 - 1}}\\
&\le \mathbb{E}[\mathcal{M}_4]\frac{C_{q,q_x}^2}{t^2}\frac{d}{n}\left[\frac{\log(en)}{\underline{\lambda}} + K_x^2K_y^2\right] + \frac{1}{n^{q_{xy}/2 - 1}}.
\end{align*}
The last inequality follows from Markov's inequality.
Taking $t = (d\mathbb{E}[\mathcal{M}_4]/n)^{1/3}$ and using Lemma~\ref{lem:M-4-bound} to upper bound   $\mathbb{E}[\mathcal{M}_4]$, we conclude that
\begin{align}
&\mathbb{P}\left(\mathcal{M}_4^{1/2}\|\widehat{\beta} - \beta\|_{\Sigma} \ge C_{q_x}\frac{K_x^{4/3}}{\underline{\lambda}^{1/3}}\left[\frac{d^{1/3}}{n^{1/3}} +\frac{d(\log n)^{1/3}}{n^{2/3 - 4/(3q_x)}} +\frac{d^{2/3}(\log n)^{1/6}}{n^{1/2 - 2/(3q_x)}}\right]\right)   \nonumber \\
&\le C_{q,q_x}\left[\frac{d^{1/3}\log n}{n^{1/3}} +\frac{d(\log n)^{4/3}}{n^{2/3 - 4/(3q_x)}} +\frac{d^{2/3}(\log n)^{7/6}}{n^{1/2 - 2/(3q_x)}}\right]\left(\frac{K_x^{4/3}}{\underline{\lambda}^{4/3}} + \frac{K_x^{10/3}K_y^2}{\underline{\lambda}^{1/3}}\right) \nonumber \\ 
&\quad+ \frac{1}{n^{q_{xy}/2 - 1}}.\label{eq:M-4-beta-error}
\end{align}
The remaining two terms $\mathcal{D}_n^{\Sigma}$ and $\mathbf{I}_1$ from Lemma~\ref{lem:std-err-bound-deterministic} are bounded  in Proposition~\ref{prop:D_Sigma-concentration} and Lemma~\ref{lem:I-1-ideal-Sandwich-error}, respectively. Those bounds, combined with the above inequality~\eqref{eq:M-4-beta-error} for $\mathcal{M}_4^{1/2}\|\widehat{\beta}-\beta\|_{\Sigma}$, yield the result.
\end{proof}

\begin{proof}[Proof of Theorem~\ref{thm::berry-esseen}]
Because $\widehat{V}_n = n^{-1}\widehat{V}$,
\[
\frac{n^{1/2}(\widehat{\beta}_j - \beta_j)}{\sqrt{(\widehat{\Sigma}_n^{-1}\widehat{V}\widehat{\Sigma}_n^{-1})_{jj}}} = \frac{\widehat{\beta}_j - \beta_j}{\sqrt{(\widehat{\Sigma}_n^{-1}\widehat{V}_n\widehat{\Sigma}_n^{-1})_{jj}}}.
\]
Thus, Theorem~\ref{thm:Berry-Esseen-OLS} can be applied with  $\eta_n$ taken to be the maximum of $\eta_n$ from Theorem~\ref{thm:asymptotic-scaling-CLT-explicit} and the right hand side of the bound in Lemma~\ref{lem:sandwich-variance-error-control}.
\end{proof}

\begin{proof}[Proof of Theorem \ref{thm:multiplier-bootstrap-consistency}]
We use Corollary 1 of~\cite{massart1990tight} and Theorem 2.2 of~\cite{lopes2020central} to prove the result. 
Firstly, because $T_b, 1\le b\le B$, are independent and identically distributed random variables in $\mathbb{R}$ conditional on $\mathcal{D}_n$, Corollary 1 of~\cite{massart1990tight} concludes
\begin{equation}\label{eq:Average-to-conditional-probability}
\sup_{t\ge 0}\left|\frac{1}{B}\sum_{b=1}^B\mathbbm{1}\{T_b \le t\} - \mathbb{P}(T_b \le t\big|\mathcal{D}_n)\right| \le \sqrt{\frac{\log(2n)}{2B}},
\end{equation}
with probability at least $1 - 1/n$.

Secondly, because $T_b$ is the maximum absolute value of a Gaussian vector with unit variances conditional on $\mathcal{D}_n$, Theorem 2.2 of~\cite{lopes2020central} yields
\begin{equation}\label{eq:Gaussian-comparison-bound}
\sup_{t\ge 0}\,\left|\mathbb{P}(T_b \le t\big|\mathcal{D}_n) - \mathbb{P}\left(\max_{1\le j\le d}|G_j| \le t\right)\right| \le \frac{C\log(d)}{\lambda_{\circ}}\Delta_0\log(1/\Delta_0),
\end{equation}
where
\[
\Delta_0 ~:=~ \max_{1\le j < k\le d}\left|\mbox{corr}(\widehat{\Sigma}_n^{-1}\widehat{V}_n\widehat{\Sigma}_n^{-1})_{jk} - \mbox{corr}(\Sigma_n^{-1}V_n\Sigma_n^{-1})_{jk}\right|.
\]
For notational convenience, let
$A := \widehat{\Sigma}_n^{-1}\widehat{V}_n\widehat{\Sigma}_n^{-1},$ and $B := \Sigma_n^{-1}V_n\Sigma_n^{-1}.$
Also, set $b_j = B^{1/2}e_j, 1\le j\le d$ with $B^{1/2}$. 
Next, we claim that 
\[
\Delta_0 \le \|B^{-1/2}AB^{-1/2} - I_d\|_{\mathrm{op}}\left(2 + \|B^{-1/2}AB^{-1/2}\|_{\mathrm{op}}\right).
\]
Indeed,
\begin{align*}
\Delta_0 
&= \max_{1\le j < k\le d}\left|\frac{e_j^{\top}Ae_k}{\sqrt{(e_j^{\top}Ae_j)(e_k^{\top}Ae_k)}} - \frac{e_j^{\top}Be_k}{\sqrt{(e_j^{\top}Be_j)(e_k^{\top}Be_k)}}\right|\\
&= \max_{1\le j\le k \le d}\left|\frac{b_j^{\top}B^{-1/2}AB^{-1/2}b_k}{\sqrt{(b_j^{\top}B^{-1/2}AB^{-1/2}b_j)(b_k^{\top}B^{-1/2}AB^{-1/2}b_k)}} - \frac{b_j^{\top}b_k}{\|b_j\|_{I_d}\|b_k\|_{I_d}}\right|\\
&\le \max_{1\le j < k\le d}\left|\frac{b_j^{\top}B^{-1/2}AB^{-1/2}b_k}{\sqrt{(b_j^{\top}B^{-1/2}AB^{-1/2}b_j)(b_k^{\top}B^{-1/2}AB^{-1/2}b_k)}} - \frac{b_j^{\top}B^{-1/2}AB^{-1/2}b_k}{\sqrt{(b_j^{\top}b_j)(b_k^{\top}b_k)}}\right|\\
&\qquad+ \max_{1\le j < k\le d}\left|\frac{b_j^{\top}B^{-1/2}AB^{-1/2}b_k}{\sqrt{(b_j^{\top}b_j)(b_k^{\top}b_k)}} - \frac{b_j^{\top}b_k}{\|b_j\|_{I_d}\|b_k\|_{I_d}}\right|\\
&\le \max_{1\le j < k\le d}\left|\frac{b_j^{\top}B^{-1/2}AB^{-1/2}b_k}{\sqrt{(b_j^{\top}B^{-1/2}AB^{-1/2}b_j)(b_k^{\top}B^{-1/2}AB^{-1/2}b_k)}}\right|\\
& \times\left|1 - \sqrt{\frac{(b_j^{\top}B^{-1/2}AB^{-1/2}b_j)(b_k^{\top}B^{-1/2}AB^{-1/2}b_k)}{b_j^{\top}b_jb_k^{\top}b_k}}\right|+ \|B^{-1/2}AB^{-1/2} - I_d\|_{\mathrm{op}}
\end{align*}
Using Cauchy-Schwarz inequality and the inequality $|1 - \sqrt{x}| \le |1 - x|$, valid for all $x \geq 0$, the previous expression is uper bounded by  the last expression is bounded by
\begin{align*}
&  \max_{1\le j \le k\le d}\left|1 - \frac{(b_j^{\top}B^{-1/2}AB^{-1/2}b_j)(b_k^{\top}B^{-1/2}AB^{-1/2}b_k)}{b_j^{\top}b_jb_k^{\top}b_k}\right| + \|B^{-1/2}AB^{-1/2} - I_d\|_{\mathrm{op}}\\
&\qquad~ \le \max_{1\le j\le d}\left|1 - \frac{b_j^{\top}B^{-1/2}AB^{-1/2}b_j}{b_j^{\top}b_j}\right| + \max_{1\le k\le d}\left|1 - \frac{b_k^{\top}B^{-1/2}AB^{-1/2}b_k}{b_k^{\top}b_k}\right|\|B^{-1/2}AB^{-1/2}\|_{\mathrm{op}}\\ &\qquad+ \|B^{-1/2}AB^{-1/2} - I_d\|_{\mathrm{op}}\\
&\qquad~ \le 2\|B^{-1/2}AB^{-1/2} - I_d\|_{\mathrm{op}} + \|B^{-1/2}AB^{-1/2} - I_d\|_{\mathrm{op}}\|B^{-1/2}AB^{-1/2}\|_{\mathrm{op}}\\
&\qquad~ = \|B^{-1/2}AB^{-1/2} - I_d\|_{\mathrm{op}}\left(2 + \|B^{-1/2}AB^{-1/2}\|_{\mathrm{op}}\right),
\end{align*}
where the first inequality follows form follows from the identity $(1 - ab) = (1 - b) - (a-1)b$ along with the triangle inequality.

Lemma~\ref{lem:sandwich-variance-error-control} now yields the rate of convergence of $\Delta_0$. 
Substituting this in~\eqref{eq:Gaussian-comparison-bound}
and combining with inequality~\eqref{eq:Average-to-conditional-probability}
concludes the proof.
\end{proof}


\end{appendices}

\bibliography{paper_CLT}
\bibliographystyle{apalike}
\newpage
\begin{appendices}
\begin{center}
{\Large {\bf Supplementary Material}}
\end{center}

\section{Proof of the Main Results from Section~\ref{section::partial} (Partial Correlations)}
\label{appendix:main.partial}

Recall the functions $\psi_{jk}(\cdot)$ and their estimators $\widehat{\psi}_{jk}(\cdot)$ given in~\eqref{eq:def-psi-jk} and~\eqref{eq:def-widehat-psi-jk}, respectively, and, similarly, the definitions of $a_j(\cdot)$ and $\widehat{a}_j(\cdot)$  in~\eqref{eq:ajx-function} and~\eqref{eq:ajhatx-function}, respectively.

\begin{proof}[Proof of Theorem~\ref{thm:linear-expansion-partial-corr}]
For notational convenience, let
$\widehat{\omega}_{jk} := e_j^{\top}\widehat{\Sigma}^{-1}e_k$ and $\omega_{jk} := e_j^{\top}\Sigma^{-1}e_k.$
Before bounding $\widehat{\theta}_{jk} - \theta_{jk}$, we note a few inequalities related to $\widehat{\omega}_{jk}$ and $\omega_{jk}$ that follow from~\eqref{eq:inv-covariance-error-bound} of Lemma~\ref{lemma:linear-expansion-inv-covariance}:
\begin{equation}\label{eq:inequalities-omegas}
\begin{split}
\max\left\{\left|\sqrt{\frac{\widehat{\omega}_{jj}}{\omega_{jj}}} - 1\right|, \left|\frac{\widehat{\omega}_{jj}}{\omega_{jj}} - 1\right|, \left|\frac{\widehat{\omega}_{jk} - \omega_{jk}}{\sqrt{\omega_{jj}\omega_{kk}}}\right|\right\} ~&\le~ \frac{\mathcal{D}_n^{\Sigma}}{1 - \mathcal{D}_n^{\Sigma}};\\
\max\left\{\left|\sqrt{\frac{\widehat{\omega}_{jj}\widehat{\omega}_{kk}}{\omega_{jj}\omega_{kk}}} - 1\right|, \left|\frac{\widehat{\omega}_{jj}\widehat{\omega}_{kk}}{\omega_{jj}\omega_{kk}} - 1\right|\right\} ~&\le~ \frac{2\mathcal{D}_n^{\Sigma}}{1 - \mathcal{D}_n^{\Sigma}} + \left(\frac{\mathcal{D}_n^{\Sigma}}{1 - \mathcal{D}_n^{\Sigma}}\right)^2;\quad\mbox{and}\\
\frac{1}{1 + \mathcal{D}_n^{\Sigma}} ~\le~ \frac{\widehat{\omega}_{jj}}{\omega_{jj}} ~&\le~ \frac{1}{1 - \mathcal{D}_n^{\Sigma}},\quad\mbox{for all}\quad 1\le j\le d.
\end{split}
\end{equation}
All these inequalities follow from the fact that
\begin{align*}
\|\Sigma^{1/2}(\widehat{\Sigma}^{-1} - \Sigma^{-1})\Sigma^{1/2}\|_{\mathrm{op}} &= \sup_{x, y\in\mathbb{R}^d}\,\frac{x^{\top}\Sigma^{1/2}(\widehat{\Sigma}^{-1} - \Sigma^{-1})\Sigma^{1/2}y}{\|x\|\|y\|}\\ 
&\ge \max_{j,k}\frac{|e_j^{\top}(\widehat{\Sigma}^{-1} - \Sigma^{-1})e_k|}{\sqrt{(e_j^{\top}\Sigma^{-1}e_j)(e_k^{\top}\Sigma^{-1}e_k)}}.
\end{align*}
Observe that
\begin{align*}
\widehat{\theta}_{jk} - \theta_{jk} &= -\frac{\widehat{\omega}_{jk}}{\sqrt{\widehat{\omega}_{jj}\widehat{\omega}_{kk}}} + \frac{\omega_{jk}}{\sqrt{\omega_{jj}\omega_{kk}}}\\
&= -\frac{\widehat{\omega}_{jk} - \omega_{jk}}{\sqrt{\widehat{\omega}_{jj}\widehat{\omega}_{kk}}} - \frac{\omega_{jk}}{\sqrt{\widehat{\omega}_{jj}\widehat{\omega}_{kk}}}\left[1 - \sqrt{\frac{\widehat{\omega}_{jj}\widehat{\omega}_{kk}}{\omega_{jj}\omega_{kk}}}\right].
\end{align*}
The equation $\sqrt{a} - 1 = (a-1)/2 - (\sqrt{a} - 1)^2/2$ for $a > 0$ yields
\begin{align*}
\widehat{\theta}_{jk} - \theta_{jk} &= -\frac{\widehat{\omega}_{jk} - \omega_{jk}}{\sqrt{\widehat{\omega}_{jj}\widehat{\omega}_{kk}}} - \frac{\omega_{jk}}{2\sqrt{\widehat{\omega}_{jj}\widehat{\omega}_{kk}}}\left[1 - \frac{\widehat{\omega}_{jj}\widehat{\omega}_{kk}}{\omega_{jj}\omega_{kk}}\right] - \frac{\omega_{jk}}{2\sqrt{\widehat{\omega}_{jj}\widehat{\omega}_{kk}}}\left(1 - \sqrt{\frac{\widehat{\omega}_{jj}\widehat{\omega}_{kk}}{\omega_{jj}\omega_{kk}}}\right)^2.
\end{align*}
Finally using $a'b' - 1 = (a' - 1) + (b' - 1) + (a' - 1)(b' - 1)$, we obtain
\begin{align*}
&\left|\widehat{\theta}_{jk} - \theta_{jk} + \frac{\widehat{\omega}_{jk} - \omega_{jk}}{\sqrt{\widehat{\omega}_{jj}\widehat{\omega}_{kk}}} - \frac{\omega_{jk}}{2\sqrt{\widehat{\omega}_{jj}\widehat{\omega}_{kk}}}\left[\frac{\widehat{\omega}_{jj}}{\omega_{jj}} + \frac{\widehat{\omega}_{kk}}{\omega_{kk}} - 2\right]\right|\\
&\qquad\le \frac{|\omega_{jk}|}{2\sqrt{\widehat{\omega}_{jj}\widehat{\omega}_{kk}}}\times\left|1 - \frac{\widehat{\omega}_{jj}}{\omega_{jj}}\right|\times\left|1 - \frac{\widehat{\omega}_{kk}}{\omega_{kk}}\right| ~+~ \frac{|\omega_{jk}|}{2\sqrt{\widehat{\omega}_{jj}\widehat{\omega}_{kk}}}\left|1 - \sqrt{\frac{\widehat{\omega}_{jj}\widehat{\omega}_{kk}}{\omega_{jj}\omega_{kk}}}\right|^2
\end{align*}
We now replace $\widehat{\omega}_{jj}, \widehat{\omega}_{kk}$ in the denominator of the left side to get
\begin{align*}
&\left|\widehat{\theta}_{jk} - \theta_{jk} + \frac{\widehat{\omega}_{jk} - \omega_{jk}}{\sqrt{{\omega}_{jj}{\omega}_{kk}}} - \frac{\omega_{jk}}{2\sqrt{{\omega}_{jj}{\omega}_{kk}}}\left[\frac{\widehat{\omega}_{jj}}{\omega_{jj}} + \frac{\widehat{\omega}_{kk}}{\omega_{kk}} - 2\right]\right|\\
&\qquad\le \frac{|\omega_{jk}|}{2\sqrt{\widehat{\omega}_{jj}\widehat{\omega}_{kk}}}\times\left|1 - \frac{\widehat{\omega}_{jj}}{\omega_{jj}}\right|\times\left|1 - \frac{\widehat{\omega}_{kk}}{\omega_{kk}}\right| ~+~ \frac{|\omega_{jk}|}{2\sqrt{\widehat{\omega}_{jj}\widehat{\omega}_{kk}}}\left|1 - \sqrt{\frac{\widehat{\omega}_{jj}\widehat{\omega}_{kk}}{\omega_{jj}\omega_{kk}}}\right|^2\\
&\qquad\qquad+ \frac{|\widehat{\omega}_{jk} - \omega_{jk}|}{\sqrt{\omega_{jj}\omega_{kk}}}\left|\sqrt{\frac{\omega_{jj}\omega_{kk}}{{\widehat{\omega}_{jj}\widehat{\omega}_{kk}}}} - 1\right| + \frac{|\omega_{jk}|}{2\sqrt{\omega_{jj}\omega_{kk}}}\left|\frac{\widehat{\omega}_{jj}}{\omega_{jj}} + \frac{\widehat{\omega}_{kk}}{\omega_{kk}} - 2\right|\times\left|\sqrt{\frac{\omega_{jj}\omega_{kk}}{{\widehat{\omega}_{jj}\widehat{\omega}_{kk}}}} -  1\right|. 
\end{align*}
To bound the right hand side we use inequalities~\eqref{eq:inequalities-omegas}. Note that
\[
\frac{|\omega_{jk}|}{\sqrt{\omega_{jj}\omega_{kk}}} \le 1,\quad\mbox{and}\quad \frac{|\omega_{jk}|}{\sqrt{\widehat{\omega}_{jj}\widehat{\omega}_{kk}}}  \leq \frac{|\omega_{jk}|}{\sqrt{\omega_{jj}\omega_{kk}}} (1 + \mathcal{D}_n^{\Sigma}) 
\leq 1 + \mathcal{D}_n^{\Sigma}.
\]
Further, second and third inequalites of~\eqref{eq:inequalities-omegas} yield
\begin{align*}
\left|\sqrt{\frac{\omega_{jj}\omega_{kk}}{\widehat{\omega}_{jj}\widehat{\omega}_{kk}}} - 1\right| &= \left|\sqrt{\frac{\widehat{\omega}_{jj}\widehat{\omega}_{kk}}{\omega_{jj}\omega_{kk}}} - 1\right|\sqrt{\frac{\omega_{jj}\omega_{kk}}{\widehat{\omega}_{jj}\widehat{\omega}_{kk}}}\\
&\le \left[\frac{2\mathcal{D}_n^{\Sigma}}{1 - \mathcal{D}_n^{\Sigma}} + \left(\frac{\mathcal{D}_n^{\Sigma}}{1 - \mathcal{D}_n^{\Sigma}}\right)^2\right](1 + \mathcal{D}_n^{\Sigma}) \le 9\mathcal{D}_n^{\Sigma},
\end{align*}
under the assumption $\mathcal{D}_n^{\Sigma} \le 1/2$.

Combining these inequalities, we conclude
\begin{equation}\label{eq:prelim-partial-corr-expansion}
\left|\widehat{\theta}_{jk} - \theta_{jk} + \frac{\widehat{\omega}_{jk} - \omega_{jk}}{\sqrt{{\omega}_{jj}{\omega}_{kk}}} - \frac{\omega_{jk}}{2\sqrt{{\omega}_{jj}{\omega}_{kk}}}\left[\frac{\widehat{\omega}_{jj}}{\omega_{jj}} + \frac{\widehat{\omega}_{kk}}{\omega_{kk}} - 2\right]\right| \le C(\mathcal{D}_n^{\Sigma})^2,
\end{equation}
for a universal constant $C \in (0, \infty)$ and for all $1\le j, k\le d$.
Finally~\eqref{eq:final-linear-expansion-inv-covariance} of Lemma~\ref{lemma:linear-expansion-inv-covariance} implies
\[
\max_{1\le j\le k \le d}\left|\frac{\widehat{\omega}_{jk} - \omega_{jk} - e_j^{\top}\Sigma^{-1}(\widetilde{\Sigma} - \Sigma)\Sigma^{-1}e_k}{\sqrt{\omega_{jj}\omega_{kk}}}\right| \le \|\overline{X}_n - \mu_X\|_{\Sigma^{-1}}^2 + \frac{(\mathcal{D}_n^{\Sigma})^2}{1 - \mathcal{D}_n^{\Sigma}}.
\]
Combining this inequality with~\eqref{eq:prelim-partial-corr-expansion} and using $\mathcal{D}_n^{\Sigma} \le 1/2$ concludes
\begin{align*}
&\left|\widehat{\theta}_{jk} - \theta_{jk} + \frac{e_j^{\top}\Sigma^{-1}(\widetilde{\Sigma} - \Sigma)\Sigma^{-1}e_k}{\sqrt{\omega_{jj}\omega_{kk}}} - \frac{\theta_{jk}}{2}\left[\frac{e_j^{\top}\Sigma^{-1}(\widetilde{\Sigma} - \Sigma)\Sigma^{-1}e_j}{\omega_{jj}} + \frac{e_k^{\top}\Sigma^{-1}(\widetilde{\Sigma} - \Sigma)\Sigma^{-1}e_k}{\omega_{kk}}\right]\right|\\
&\qquad\le C(\mathcal{D}_n^{\Sigma})^2 + \|\overline{X}_n - \mu_X\|_{\Sigma^{-1}}^2,
\end{align*}
for a universal constant $C \in (0, \infty)$. This concludes the proof.
\end{proof}

\begin{proof}[Proof of Theorem \ref{thm:Berry-Esseen-bound-partial-corr}]
We will prove the theorem when $C\eta_n \le 1/2$; otherwise the result is trivially true by increasing the constant $C$. For notational convenience, let $\zeta_{jk} := \mathbb{E}[\psi_{jk}^2(X)]$. Theorem~\ref{thm:linear-expansion-partial-corr} implies that
\begin{equation}\label{eq:proper-normalized-partial-correlation}
\max_{1\le j < k\le d}\left|\frac{\widehat{\theta}_{jk} - \theta_{jk}}{\widehat{\zeta}_{jk}} + \frac{1}{n\widehat{\zeta}_{jk}}\sum_{i=1}^n \psi_{jk}(X_i)\right| ~\le~ \frac{C(\mathcal{D}_n^{\Sigma})^2 + \|\overline{X}_n - \mu_X\|_{\Sigma^{-1}}^2}{\min_{1\le j < k\le d}\widehat{\zeta}_{jk}},
\end{equation}
whenever $\mathcal{D}_n^{\Sigma} \le 1/2$. Furthermore,
\begin{equation}\label{eq:average-variance-estimator}
\left|\frac{1}{n\widehat{\zeta}_{jk}}\sum_{i=1}^n \psi_{jk}(X_i) - \frac{1}{n\zeta_{jk}}\sum_{i=1}^n \psi_{jk}(X_i)\right| = \frac{|n^{-1}\sum_{i=1}^n \psi_{jk}(X_i)|}{\widehat{\zeta}_{jk}\zeta_{jk}}\times|\widehat{\zeta}_{jk} - \zeta_{jk}|.
\end{equation}
Define the random variable
\[
\widetilde{\zeta}_{jk}^2 ~:=~ \frac{1}{n}\sum_{i=1}^n \psi_{jk}^2(X_i),
\]
which may be regarded as a pseudo-estimator of sort, since $\mathbb{E}[\widetilde{\zeta}_{jk}^2] = \zeta^2_{jk}$. Of course $\widetilde{\zeta}_{jk}^2$ is not a computable  estimator of $\zeta^2_{jk}$ because it depends on unknown quantities, namely $\Sigma$ and $\mu_X$.

Equations ~\eqref{eq:proper-normalized-partial-correlation} and \eqref{eq:average-variance-estimator} together imply that
\begin{equation}\label{eq:combination-influence-fn-exp-partial-correlation}
\begin{split}
\max_{1\le j < k\le d}\left|\frac{\widehat{\theta}_{jk} - \theta_{jk}}{\widehat{\zeta}_{jk}} + \frac{1}{n\zeta_{jk}}\sum_{i=1}^n \psi_{jk}(X_i)\right| ~&\le~ \frac{C(\mathcal{D}^{\Sigma})^2 + \|\widebar{X}_n - \mu_X\|_{\Sigma^{-1}}^2}{\min_{1\le j < k\le d}\widehat{\zeta}_{jk}}\\
~&\qquad+~ \max_{1\le j < k\le d}\frac{|\widehat{\zeta}_{jk} - \zeta_{jk}|}{\widehat{\zeta}_{jk}\zeta_{jk}}\left|\frac{1}{n}\sum_{i=1}^n \psi_{jk}(X_i)\right|. 
\end{split}
\end{equation}

Clearly,
\begin{equation}\label{eq:bound-hat.sigma-sigma}
\widehat{\zeta}_{jk} = \zeta_{jk}\left(1 + \frac{\widehat{\zeta}_{jk}}{\zeta_{jk}} - 1\right) \ge \zeta_{jk}\left(1 - \left|\frac{\widehat{\zeta}_{jk}}{\zeta_{jk}} - 1\right|\right).
\end{equation}
Using the pseudo-estimator $\widetilde{\zeta}_{jk}^2$, we have with probability $1 - C/n$,
\begin{equation}\label{eq:sigma-hat-to-sigma-tilde}
\max_{1\le j < k\le d}|\widehat{\zeta}_{jk} - \widetilde{\zeta}_{jk}| \le CK_x^6\sqrt{\frac{d + \log n}{n}} + CK_x^5\frac{d + \log n}{n},
\end{equation}
Next, since
\[
\left|\widetilde{\zeta}_{jk} - \zeta_{jk}\right| = \frac{|\widetilde{\zeta}_{jk}^2 - \zeta_{jk}^2|}{|\widetilde{\zeta}_{jk} + \zeta_{jk}|} \le \frac{1}{\zeta_{jk}}\left|\frac{1}{n}\sum_{i=1}^n \psi_{jk}^2(X_i) - \mathbb{E}[\psi_{jk}^2(X_i)]\right|.
\]
we have that
\[
\max_{1\le j < k\le d}\left|\widetilde{\zeta}_{jk} - \zeta_{jk}\right| ~\le~ \frac{1}{\min_{1\le j < k\le d}\zeta_{jk}}\max_{1\le j < k\le d}\left|\frac{1}{n}\sum_{i=1}^n \psi_{jk}^2(X_i) - \mathbb{E}[\psi_{jk}^2(X_i)]\right|.
\]
Then, applying  Lemma~\ref{lemma:Thm3.1.KuchAbhi} with $\alpha = 1/2, q = d^2$ and $t = \log n$,  we obtain that, with probability at least $1 - 3/n$,
\begin{equation}\label{eq:sigma-tilde-to-sigma}
\max_{1\le j < k\le d}\left|\widetilde{\zeta}_{jk} - \zeta_{jk}\right| ~\le~ \frac{CK_x^4}{\zeta_{\min}}\left[\sqrt{\frac{\log(dn)}{n}} + \frac{\log^2(dn)}{n}\right].
\end{equation}

Combining inequalities~\eqref{eq:sigma-hat-to-sigma-tilde} and~\eqref{eq:sigma-tilde-to-sigma} we now conclude that, with probability at least $1 - (C+3)/n$,
\begin{align*}
\max_{1\le j < k\le d}\left|\frac{\widehat{\zeta}_{jk}}{\zeta_{jk}} - 1\right| ~&\le~ \frac{CK_x^6}{\zeta_{\min}}\sqrt{\frac{d + \log n}{n}} + \frac{CK_x^5}{\zeta_{\min}}\frac{d + \log n}{n}\\
~&\qquad+~ \frac{CK_x^4}{\zeta_{\min}^2}\left[\sqrt{\frac{\log(dn)}{n}} + \frac{\log^2(dn)}{n}\right].
\end{align*}
Because $K_x \ge 1$, the previous bound reduces to
\begin{equation}\label{eq:sigma-hat-to-sigma}
\max_{1\le j < k\le d}\left|\frac{\widehat{\zeta}_{jk}}{\zeta_{jk}} - 1\right|  \leq C\left(\frac{K_x^6}{\zeta_{\min}} + \frac{K_x^4}{\zeta_{\min}^2}\right)\left[\sqrt{\frac{d + \log n}{n}} + \frac{d + \log^2n}{n}\right]. 
\end{equation}
 Assuming $n$ large enough so that the quantity on the right hand side is bounded by $1/2$ and using the inequality \eqref{eq:bound-hat.sigma-sigma}, we get that, with probability at least $1 - (C + 3)/n$, $\widehat{\zeta}_{jk} \ge \zeta_{jk}/2$ for all $1\le j < k\le d$ and hence,
\begin{align*}
&\max_{1\le j < k\le d}\left|\frac{\widehat{\theta}_{jk} - \theta_{jk}}{\widehat{\zeta}_{jk}} + \frac{1}{n\zeta_{jk}}\sum_{i=1}^n \psi_{jk}(X_i)\right|\\ 
&\qquad\le~ \frac{2C(\mathcal{D}^{\Sigma})^2 + 2\|\widebar{X}_n - \mu_X\|_{\Sigma^{-1}}^2}{\zeta_{\min}}
+ 2\max_{1\le j < k\le d}\frac{|\widehat{\zeta}_{jk} - \zeta_{jk}|}{\zeta_{jk}^2}\left|\frac{1}{n}\sum_{i=1}^n \psi_{jk}(X_i)\right|. 
\end{align*}

We now proceed to derive a high probability bound for the last display. The term $\max_{1\le j < k\le d}{|\widehat{\zeta}_{jk} - \zeta_{jk}|}/{\zeta_{jk}^2}$ can be bounded as in equation~\eqref{eq:sigma-hat-to-sigma}, with probability at least  $1 - (C+3)/n$. Next, Lemma~\ref{lemma:Thm3.1.KuchAbhi} with $\alpha=1$ gives that, with probability at least $1 - 3/n$,
\begin{equation}\label{eq:bound-on-sum-psi}
\max_{1\le j < k\le d}\left|\frac{1}{n}\sum_{i=1}^n \psi_{jk}(X_i)\right| ~\le~ CK_x^2\sqrt{\frac{\log(dn)}{n}} + CK_x^2\frac{\log(dn)}{n}.
\end{equation}

To bound $(\mathcal{D}_n^{\Sigma})^2 + \|\overline{X}_n - \mu_X\|_{\Sigma^{-1}}^2$, we notice that, by Proposition~\ref{prop:bounding-D-sigma} 
\[
\mathcal{D}_n^{\Sigma} ~\le~ \|\Sigma^{-1/2}\widetilde{\Sigma}\Sigma^{-1/2} - I_d\|_{\mathrm{op}} + \|\overline{X}_n - \mu_X\|_{\Sigma^{-1}}^2.
\]
Next, Lemma~\ref{lem:concentration-of-covariance} yields that
\[
\mathbb{P}\left(\|\Sigma^{-1/2}\widetilde{\Sigma}\Sigma^{-1/2} - I_d\|_{\mathrm{op}} \ge CK_x^2\sqrt{\frac{d + \log(1/\delta)}{n}} + CK_x^2\frac{d + \log(1/\delta)}{n}\right) \le \delta.
\]
and the sub-Gaussianity assumption further implies that
\[
\mathbb{P}\left(\|\overline{X}_n - \mu_X\|_{\Sigma^{-1}} \ge CK_x\sqrt{\frac{d + \log(1/\delta)}{n}}\right) \le \delta.
\]
Therefore, 
\begin{equation}\label{eq:linear-rep-error-partial-corr}
\mathbb{P}\left(n^{1/2}(\mathcal{D}_n^{\Sigma})^2 + n^{1/2}\|\overline{X}_n - \mu_X\|_{\Sigma^{-1}}^2 \ge C(K_x^2 + K_x^4)\frac{d + \log(n)}{\sqrt{n}}\right) \le \frac{1}{n}.
\end{equation}

Combining the bounds \eqref{eq:sigma-hat-to-sigma}, \eqref{eq:bound-on-sum-psi} and \eqref{eq:linear-rep-error-partial-corr}, we  conclude that, with probability at least $1 - C/n$,
\begin{align*}
&\max_{1\le j < k\le d}\left|\frac{n^{1/2}(\widehat{\theta}_{jk} - \theta_{jk})}{\widehat{\zeta}_{jk}} + \frac{1}{\sqrt{n}\zeta_{jk}}\sum_{i=1}^n \psi_{jk}(X_i)\right|\\ 
~&\le~ \frac{CK_x^4}{\zeta_{\min}}\frac{d + \log n}{\sqrt{n}}\\
~&+~ C\left(\frac{K_x^8\sqrt{\log(dn)}}{\zeta_{\min}^2} + \frac{K_x^6\sqrt{\log(dn)}}{\zeta_{\min}^3}\right)\left[\sqrt{\frac{d + \log n}{n}} + \frac{d + \log^2n}{n}\right]\left(1 + \sqrt{\frac{\log(dn)}{n}}\right),
\end{align*}
whenever the right hand side is smaller than $1/2$.
Because $d \le n$, ${d/n} \le \sqrt{d/n}$ and 
\[
\sqrt{\frac{\log n}{n}} + \frac{\log^2n}{n} = \sqrt{\frac{\log n}{n}}\left(1 + \frac{\log^{3/2}n}{n^{1/2}}\right) \le C\sqrt{\frac{\log n}{n}},
\]
we have that
\[
\sqrt{\frac{d + \log n}{n}} + \frac{d + \log^2n}{n} \le C\sqrt{\frac{d + \log n}{n}}.
\] 
Thus we have shown that, with probability at least $1 - C/n$,
\begin{align*}
&\max_{1\le j < k\le d}\left|\frac{n^{1/2}(\widehat{\theta}_{jk} - \theta_{jk})}{\widehat{\zeta}_{jk}} + \frac{1}{\sqrt{n}\zeta_{jk}}\sum_{i=1}^n \psi_{jk}(X_i)\right|\\ 
~&\le~ \frac{CK_x^4}{\zeta_{\min}}\frac{d + \log n}{\sqrt{n}} ~+~ C\left(\frac{K_x^8}{\zeta_{\min}^2} + \frac{K_x^6}{\zeta_{\min}^3}\right)\sqrt{\frac{(d + \log n)\log(dn)}{n}}\\
~& = \eta_n.
\end{align*}

By the same arguments used in the proof of Theorem~\ref{thm:Berry-Esseen-OLS}, 
\begin{equation}\label{eq:penultimate-partial-correlation11}
\begin{split}
&\sup_{t\ge0}\left|\mathbb{P}\left(\max_{1\le j < k\le d}\left|\frac{n^{1/2}(\widehat{\theta}_{jk} - \theta_{jk})}{\widehat{\zeta}_{jk}}\right| \le t\right) - \mathbb{P}\left(\max_{1\le j < k\le d}|G_{jk}| \le t\right)\right|\\ 
&\qquad\le \mathbb{P}\left(t - \eta_n \le \max_{1\le j \le k\le d}|G_{jk}| \le t + \eta_n\right) + \frac{C}{n}\\
&\qquad\qquad+ \sup_{t\ge0}\left|\mathbb{P}\left(\max_{1\le j\le k \le d}|G_{jk}| \le t\right) - \mathbb{P}\left(\max_{1\le j < k\le d}\left|\frac{1}{\sqrt{n}}\sum_{i=1}^n \frac{\psi_{jk}(X_i)}{\zeta_{jk}}\right| \le t\right)\right|.
\end{split}
\end{equation}

By Nazarov's inequality  \citep[see Lemma A.1 in][]{chernozhukov2017detailed},
\[
\mathbb{P}\left(t - \eta_n \le \max_{1\le j \le k\le d}|G_{jk}| \le t + \eta_n\right) \le C  \eta_n\sqrt{\log d},
\]
for a universal constant $C>0$. 
To bound the last term of~\eqref{eq:penultimate-partial-correlation11}, we use Theorem 2.1(a) of~\cite{koike2019notes}. Firstly, note that $a_j(X)$ (in~\eqref{eq:ajx-function}) is sub-Gaussian by assumption and hence $\psi_{jk}(X)$ is sub-exponential satisfying $\|\psi_{jk}(X)\|_{\psi_1} \le CK_x^2$ for some universal constant $C$; this also implies that $B_n = CK_x^2$, for a constant depending on the minimum of $\zeta_{j,k}, 1\le j< k \le d$, in Theorem 2.1(a) of~\cite{koike2019notes}. Thus, Theorem 2.1(a) of~\cite{koike2019notes} yields
\begin{equation}
\begin{split}
&\sup_{t\ge0}\left|\mathbb{P}\left(\max_{1\le j\le k \le d}|G_{jk}| \le t\right) - \mathbb{P}\left(\max_{1\le j < k\le d}\left|\frac{1}{\sqrt{n}}\sum_{i=1}^n \frac{\psi_{jk}(X_i)}{\zeta_{jk}}\right| \le t\right)\right|\\ 
&\quad\le C(\log d)^{2/3}\left[\left(\frac{K_x^4\log^3d}{n}\right)^{1/6} + \left(\frac{K_x^3\log^2d\log^2n}{n}\right)^{1/3}\right]\\
&\quad\le CK_x^{4/3}\left[\frac{(\log d)^{5/6}}{n^{1/6}} + \frac{\log d(\log n)^{2/3}}{n^{1/3}}\right] \le CK_x^{4/3}\frac{(\log(d\vee n))^{5/6}}{n^{1/6}}.
\end{split}
\end{equation}
Substituting this bound in~\eqref{eq:penultimate-partial-correlation11} completes the proof. 

To prove the result when the minimum eigenvalue of the correlation matrix of $(\psi_{jk}(X_i)/\zeta_{jk}: 1\le j < k \le d)$ is bounded away from zero, we apply Corollary 2.1 of~\cite{chernozhukov2020nearly} with $q = 4$; see the last case of the corollary. Under the assumptions of the theorem, assumption (E.3) of~\cite{chernozhukov2020nearly} holds true with $q = 4$ and $B_n = C\log(d)$ for some constant $C$ depending on the minimum of $\zeta_{jk}, 1\le j < k \le d$. Hence, Corollary 2.1 of~\cite{chernozhukov2020nearly} implies that
\begin{align*}
    &\sup_{t\ge0}\left|\mathbb{P}\left(\max_{1\le j\le k \le d}|G_{jk}| \le t\right) - \mathbb{P}\left(\max_{1\le j < k\le d}\left|\frac{1}{\sqrt{n}}\sum_{i=1}^n \frac{\psi_{jk}(X_i)}{\zeta_{jk}}\right| \le t\right)\right|\\
    &\quad\le CK_x^4\frac{\log^4d\log n}{\sqrt{n}\lambda_{\star}^2},
\end{align*}
where $\lambda_{\star}^2$ is the smallest eigenvalue of the correlation matrix of the random vector $(\psi_{jk}(X_i)/\zeta_{jk}: 1\le j < k \le d)$. Substituting this bound in~\eqref{eq:penultimate-partial-correlation11} yields the $\sqrt{n}$ rate of convergence in the central limit theorem.
\end{proof}

\begin{proof}[Proof of Theorem~\ref{eq:multplier-bootstrap-consistency-partial-corr}]
By Corollary 1 of~\cite{massart1990tight}, we obtain
\[
\sup_{t\ge0}\left|\frac{1}{B}\sum_{i=1}^n \mathbbm{1}\{T_b \le t\} - \mathbb{P}\left(T_b \le t\big|X_1,\ldots,X_n\right)\right| \le \sqrt{\frac{\log(2n)}{2B}},
\]
with probability at least $1 - 1/n$. 

By Lemma 3.1 of~\cite{Cher13}, we obtain
\begin{equation}\label{eq:gaussian-comparison-partial-corr}
\begin{split}
&\sup_{t\ge0}\left|\mathbb{P}(T_b \le t|X_1,\ldots,X_n) - \mathbb{P}\left(\max_{1\le j < k\le d}|G_{jk}| \le t\right)\right|\\ 
&\quad\le C\Delta_0^{1/3}(1\vee\log(d^2/\Delta_0))^{2/3},
\end{split}
\end{equation}
where
\[
\Delta_0 := \max_{\substack{1\le j < k\le d,\\1\le j' \le k'\le d}}\left|\widehat{\mbox{corr}}(\widehat{\psi}_{jk}, \widehat{\psi}_{j'k'}) - \mbox{corr}(\psi_{jk}, \psi_{j'k'})\right|,
\]
with $\widehat{\mbox{corr}}$  defined as the sample correlation between $(\widehat{\psi}_{jk}(X_i), 1\le i\le n)$ and $(\widehat{\psi}_{j'k'}(X_i), 1\le i\le n)$.  The bound~\eqref{eq:gaussian-comparison-partial-corr} can be improved if $\lambda_{\star}$, the minimum eigenvalue of the correlation matrix, is bounded away from zero. This improvement follows from Theorem 2 of~\cite{lopes2020central}.
The rest of the proof is devoted to bounding the term $\Delta_0$.
Towards that end, Lemma~\ref{lem:correlations-from-covariances}  yields that
\begin{equation}\label{eq:Delta_zero-Delta_tilde-bound}
\Delta_0 ~\le~ 4\widetilde{\Delta}_0 := 4\max_{\substack{1\le j,k\le d,\\1\le j',k' \le d}}\,\left|\frac{\widehat{\mbox{cov}}(\widehat{\psi}_{jk}, \widehat{\psi}_{j'k'}) - \mbox{cov}(\psi_{jk}\psi_{j'k'})}{\sqrt{\mbox{Var}(\psi_{jk})\mbox{Var}(\psi_{j'k'})}}\right|, 
\end{equation}
whenever $\widetilde{\Delta}_0 \le 1/2$. 
Below we will derive a high-probability bound for $\tilde{\Delta}_0$, which is shown to be vanishing provided that $d \le \sqrt{n}$.

Because $\sum_{i=1}^n \widehat{\psi}_{jk}(X_i) = 0$, the empirical covariance is given by
\[
\widehat{\mbox{cov}}(\widehat{\psi}_{jk}, \widehat{\psi}_{j'k'}) ~=~ \frac{1}{n}\sum_{i=1}^n \widehat{\psi}_{jk}(X_i)\widehat{\psi}_{j'k'}(X_i),
\]
and similarly, $\mbox{cov}(\psi_{jk}, \psi_{j'k'}) = \mathbb{E}[\psi_{jk}(X)\psi_{j'k'}(X)]$. These equalities lead to
\begin{equation}\label{eq:basic-decomposition-partial-correlation}
\begin{split}
\widehat{\mbox{cov}}(\widehat{\psi}_{jk}, \widehat{\psi}_{j'k'}) - \mbox{cov}(\psi_{jk}, \psi_{j'k'}) &= \frac{1}{n}\sum_{i=1}^n \left\{\widehat{\psi}_{jk}(X_i)\widehat{\psi}_{j'k'}(X_i) - \psi_{jk}(X_i)\psi_{j'k'}(X_i)\right\}\\
&\qquad+ \frac{1}{n}\sum_{i=1}^n \Big\{\psi_{jk}(X_i)\psi_{j'k'}(X_i) - \mathbb{E}\left[\psi_{jk}(X)\psi_{j'k'}(X)\right]\Big\}.
\end{split}
\end{equation}
By Lemma~\ref{lemma:Thm3.1.KuchAbhi}, with probability at least $1 - 3/n$,
\[
\left|\frac{1}{n}\sum_{i=1}^n \left\{\psi_{jk}(X_i)\psi_{j'k'}(X_i) - \mathbb{E}\left[\psi_{jk}(X)\psi_{j'k'}(X)\right]\right\}\right| \le CK_x^4\left[\sqrt{\frac{\log(dn)}{n}} + \frac{\log^{2}(dn)}{n}\right].
\]
We now bound the first term in~\eqref{eq:basic-decomposition-partial-correlation} as follows:
\begin{align*}
&\left|\frac{1}{n}\sum_{i=1}^n \left\{\widehat{\psi}_{jk}(X_i)\widehat{\psi}_{j'k'}(X_i) - \psi_{jk}(X_i)\psi_{j'k'}(X_i)\right\}\right|\\ 
&\qquad\le \left|\frac{1}{n}\sum_{i=1}^n \left(\widehat{\psi}_{jk}(X_i) - \psi_{jk}(X_i)\right)\psi_{j'k'}(X_i)\right|\\
&\qquad\qquad+ \left|\frac{1}{n}\sum_{i=1}^n \left(\widehat{\psi}_{j'k'}(X_i) - \psi_{j'k'}(X_i)\right)\psi_{jk}(X_i)\right|\\
&\qquad\qquad+ \left|\frac{1}{n}\sum_{i=1}^n \left(\widehat{\psi}_{j'k'}(X_i) - \psi_{j'k'}(X_i)\right)\left(\widehat{\psi}_{jk}(X_i) - \psi_{jk}(X_i)\right)\right|\\
&\qquad\le 2\max_{1\le j < k\le d}\,\sqrt{\frac{1}{n}\sum_{i=1}^n \left(\widehat{\psi}_{jk}(X_i) - \psi_{jk}(X_i)\right)^2}\max_{1\le j < k\le d}\,\sqrt{\frac{1}{n}\sum_{i=1}^n \psi_{jk}^2(X_i)}\\
&\qquad\qquad+ \max_{1\le j < k\le d}\,{\frac{1}{n}\sum_{i=1}^n \left(\widehat{\psi}_{jk}(X_i) - \psi_{jk}(X_i)\right)^2}.
\end{align*}

Applying Lemma~\ref{lem:application-theorem8} with $W_i = \psi_{jk}^2(X_i)$, which by assumption is sub-Weibull$(1/2)$, we have that for all $t > 0$,
\[
\mathbb{P}\left(\frac{1}{n}\sum_{i=1}^n \psi_{jk}^2(X_i) \ge 2e\mathbb{E}\left[\psi_{jk}^2(X)\right] + \frac{4etK_w\log^2n}{n} + \frac{4et^3K_w}{n}\right) \le 3e^{-t}.
\]
By union bound over $1\le j < k\le d$, i.e., taking $t = \log(d^2)$, this implies that, with probability at least $1 - 3/n$,
\begin{align*}
\max_{1\le j < k\le d}\frac{1}{n}\sum_{i=1}^n \psi_{jk}^2(X_i) &\le 2e\max_{1\le j < k\le d}\mathbb{E}\left[\psi_{jk}^2(X)\right] + \frac{8e\log(d)K_w\log^2n}{n} + \frac{32e\log^3dK_w}{n}\\
&\le 2eK_w + \frac{8eK_w\log^3(nd)}{n} + \frac{32eK_w\log^3(dn)}{n}\\
&= CK_w\left(1 + \frac{\log^3(nd)}{n}\right). 
\end{align*}
Hence with probability at least $1 - 3/n$,
\begin{equation}\label{eq:square-power-psi}
\max_{1\le j < k\le d}\sqrt{\frac{1}{n}\sum_{i=1}^n \psi_{jk}^2(X_i)} \le C\sqrt{K_w}\left(1 + \sqrt{\frac{\log^3(nd)}{n}}\right).
\end{equation}


Hence with probability at least $1 - 3/n,$
\begin{equation}\label{eq:first-term-decomposition-first}
\begin{split}
&\left|\frac{1}{n}\sum_{i=1}^n \left\{\widehat{\psi}_{jk}(X_i)\widehat{\psi}_{j'k'}(X_i) - \psi_{jk}(X_i)\psi_{j'k'}(X_i)\right\}\right|\\ &\qquad\le C\sqrt{K_x}\left(\sqrt{\frac{\log^3(dn)}{n}} + 1\right)\max_{1\le j < k\le d}\sqrt{\frac{1}{n}\sum_{i=1}^n \left(\widehat{\psi}_{jk}(X_i) - \psi_{jk}(X_i)\right)^2}\\
&\qquad\qquad+ \max_{1\le j < k\le d}\,{\frac{1}{n}\sum_{i=1}^n \left(\widehat{\psi}_{jk}(X_i) - \psi_{jk}(X_i)\right)^2}.
\end{split}
\end{equation}
Assuming $d \le \sqrt{n}$, Lemma~\ref{lem:rate-of-convergence-psihat-minus-psi} proves that with probability at least $1 - C/n$,
\[
\max_{1\le j < k\le d}\sqrt{\frac{1}{n}\sum_{i=1}^n \left(\widehat{\psi}_{jk}(X_i) - \psi_{jk}(X_i)\right)^2} \le CK_x^5\sqrt{\frac{d + \log n}{n}} + CK_x^6\frac{d + \log n}{n}.
\]
Substituting this in~\eqref{eq:first-term-decomposition-first} and then in~\eqref{eq:Delta_zero-Delta_tilde-bound} proves that with probability at least $1 - C/n$,
\[
\Delta_0 \le CK_x^5\sqrt{\frac{d + \log n}{n}} + CK_x^6\frac{d + \log n}{n},
\]
whenever  $d \le \sqrt{n}$ (required for Lemma~\ref{lem:rate-of-convergence-psihat-minus-psi}) and the right hand side is less than 1. Hence,
\[
\sup_{t\ge0}\left|\mathbb{P}(T_b \le t|X_1,\ldots,X_n) - \mathbb{P}\left(\max_{1\le j < k\le d}|G_{jk}| \le t\right)\right| \le CK_x^{2}(\log d)^{2/3}\left(\frac{d + \log n}{n}\right)^{1/6}.
\]
\end{proof}

\begin{lemma}\label{lem:correlations-from-covariances}
Suppose $X_1, \ldots, X_n\in\mathbb{R}^d$ are independent and identically distributed random vectors. Set
\[
\widehat{\sigma}_{jk} := \widehat{\mathrm{cov}}(X_{ij}, X_{ik}),
\]
to be the empirical covariance between $(X_{ij}, 1\le i\le n)$ and $(X_{ik}, 1\le i\le n)$. The empirical correlation $\widehat{\sigma}_{jk}/\sqrt{\widehat{\sigma}_{jj}\widehat{\sigma}_{kk}}$ is denoted by $\widehat{\rho}_{jk}$. Let $\sigma_{jk}$ and $\rho_{jk}$ represent the corresponding population covariance and correlations. Then
\[
\max_{1\le j \le k\le D}\left|\widehat{\rho}_{jk} - {\rho}_{jk}\right| ~\le~ 4\max_{1\le j < k\le d}\left|\frac{\widehat{\sigma}_{jk} - \sigma_{jk}}{\sqrt{\sigma_{jj}\sigma_{kk}}}\right|,
\]
whenever
\[
\max_{1\le j < k\le d}\left|\frac{\widehat{\sigma}_{jk} - \sigma_{jk}}{\sqrt{\sigma_{jj}\sigma_{kk}}}\right| ~\le~ \frac{1}{2}.
\]
\end{lemma}
\begin{proof}
Fix $1\le j < k\le d$ and set
\[
\Delta := \max_{1\le j < k\le d}\left|\frac{\widehat{\sigma}_{jk} - \sigma_{jk}}{\sqrt{\sigma_{jj}\sigma_{kk}}}\right|
\]
Then,
\begin{equation}\label{eq:Delta-definition}
 \frac{1}{1 + \Delta} ~\le~ \frac{\sigma_{jj}}{\widehat{\sigma}_{jj}} ~\le~ \frac{1}{1-\Delta}\quad\mbox{for all}\quad 1\le j\le d.
\end{equation}
Observe that
\begin{align*}
\left|\widehat{\rho}_{jk} - \rho_{jk}\right| &\le \frac{|\widehat{\sigma}_{jk} - \sigma_{jk}|}{\sqrt{\widehat{\sigma}_{jj}\widehat{\sigma}_{kk}}} + \frac{\sigma_{jk}}{\sqrt{\sigma_{jj}\sigma_{kk}}}\left|\sqrt{\frac{\sigma_{jj}\sigma_{kk}}{\widehat{\sigma}_{jj}\widehat{\sigma}_{kk}}} - 1\right|\\
&\le \Delta\sqrt{\frac{\sigma_{jj}\sigma_{kk}}{\widehat{\sigma}_{jj}\widehat{\sigma}_{kk}}} + \left|\sqrt{\frac{\sigma_{jj}\sigma_{kk}}{\widehat{\sigma}_{jj}\widehat{\sigma}_{kk}}} - 1\right|.
\end{align*}
To bound the last term, we see from~\eqref{eq:Delta-definition} that, for all $j$ and $k$,
\[
\frac{1}{(1+\Delta)^2} ~\le~ \frac{\sigma_{jj}\sigma_{kk}}{\widehat{\sigma}_{jj}\widehat{\sigma}_{kk}} ~\le~ \frac{1}{(1 - \Delta)^2}, 
\]
which implies that
\[
\left|\sqrt{\frac{\sigma_{jj}\sigma_{kk}}{\widehat{\sigma}_{jj}\widehat{\sigma}_{kk}}} - 1\right| \le \frac{\Delta}{1 - \Delta}.
\]
Therefore, provided that $\Delta \leq 1/2$,
\[
\left|\widehat{\rho}_{jk} - \rho_{jk}\right| \le \frac{2\Delta}{1 - \Delta} \leq 4 \Delta.
\]
\end{proof}
\begin{lemma}\label{lem:psihat-minus-psi-part1}
For functions $\widehat{a}, \widehat{b}, a, b$ and scalars $\hat{\theta}, \theta\in[-1, 1]$, let
\begin{align*}
\widehat{\psi}(x) &:= \widehat{a}(x) - \mathbb{E}_n[\widehat{a}(X)] - \widehat{\theta}\left\{\widehat{b} - \mathbb{E}_n[b(X)]\right\},\\
\psi(x) &:= a(x) - \mathbb{E}[a(X)] - \theta\left\{b(x) - \mathbb{E}[b(X)]\right\}.
\end{align*}
Then
\begin{align*}
\|\widehat{\psi} - \psi\|_{n} &\le \|\widehat{a} - a\|_n + |\mathbb{E}_n[a(X)] - \mathbb{E}[a(X)]| + \|\widehat{b} - b\|_n + |\mathbb{E}_n[b(X)] - \mathbb{E}[b(X)]|\\
&\qquad+ |\theta - \widehat{\theta}|\;\|b - \mathbb{E}[b(X)]\|_n,
\end{align*}
where for any function $f$, $\|f\|_n := \sqrt{n^{-1}\sum_{i=1}^n f^2(X_i)}.$
\end{lemma}
\begin{proof}
The proof is mostly algebraic manipulation. For notational ease, we write $\widehat{\psi}$ for $\widehat{\psi}(x)$ and similarly for other functions. Firstly,
\begin{align*}
\widehat{\psi} - \psi &= \widehat{a} - a - \mathbb{E}_n[\widehat{a}] + \mathbb{E}[a] - \widehat{\theta}(\widehat{b} - \mathbb{E}_n[\widehat{b}]) + \theta(b - \mathbb{E}[b])\\
&= (\widehat{a} - a) - \mathbb{E}_n[\widehat{a} - a] - (\mathbb{E}_n[a] - \mathbb{E}[a]) - \widehat{\theta}(\widehat{b} - b - \mathbb{E}_n[\widehat{b}] + \mathbb{E}[b])\\
&\qquad+ (\theta - \widehat{\theta})(b - \mathbb{E}[b])\\
&= (\widehat{a} - a) - \mathbb{E}_n[\widehat{a} - a] - (\mathbb{E}_n[a] - \mathbb{E}[a]) - \widehat{\theta}(\widehat{b} - b - \mathbb{E}_n[\widehat{b} - b])\\
&\qquad + \widehat{\theta}(\mathbb{E}_n[b] - \mathbb{E}[b]) + (\theta - \widehat{\theta})(b - \mathbb{E}[b]).
\end{align*}
Observe now that
\[
\|(\widehat{a} - a) - \mathbb{E}_n[\widehat{a} - a]\|_n \le \|\widehat{a} - a\|_n.
\]
Using the fact $\widehat{\theta}, \theta\in[-1, 1]$ concludes the proof.
\end{proof}
\begin{lemma}\label{lem:psihat-minus-psi-part2}
In the notation of~\eqref{eq:ajx-function} and~\eqref{eq:ajhatx-function}, for any $1\le j < k\le d$, we have
\begin{align*}
\|\widehat{a}_j\widehat{a}_k - a_ja_k\|_n &\le 2\max_{1\le j\le d}\left(\frac{1}{n}\sum_{i=1}^n a_j^4(X_i)\right)^{1/4}\left(\frac{1}{n}\sum_{i=1}^n (\widehat{a}_j(X_i) - a_j(X_i))^4\right)^{1/4}\\
&\qquad+ \max_{1\le j\le d}\left(\frac{1}{n}\sum_{i=1}^n (\widehat{a}_j(X_i) - a_j(X_i))^4\right)^{1/2}.
\end{align*}
Consequently, there exists a universal constant $C\in(0, \infty)$ such that for all $1\le j < k\le d$,
\begin{align*}
\|\widehat{\psi}_{jk} - \psi_{jk}\|_n &\le C\max_{1\le j\le d}\left(\frac{1}{n}\sum_{i=1}^n a_j^4(X_i)\right)^{1/4}\left(\frac{1}{n}\sum_{i=1}^n (\widehat{a}_j(X_i) - a_j(X_i))^4\right)^{1/4}\\
&\qquad+ C\max_{1\le j\le d}\left(\frac{1}{n}\sum_{i=1}^n (\widehat{a}_j(X_i) - a_j(X_i))^4\right)^{1/2}\\
&\qquad+ C\max_{1\le j\le d}|\mathbb{E}_n[a_j^2(X)] - \mathbb{E}[a_j^2(X)]|\\
&\qquad+ C\max_{1\le j < k\le d}|\theta_{jk} - \widehat{\theta}_{jk}|\max_{1\le j\le d}\|a_j^2 - \mathbb{E}[a_j^2(X)]\|_n.
\end{align*}
\end{lemma}
\begin{proof}
Clearly,
\begin{align*}
|\widehat{a}_j\widehat{a}_k - a_j a_k| &\le |\widehat{a}_j||\widehat{a}_k - a_k| + |a_k||\widehat{a}_j - a_j|\\
&\le |a_j||\widehat{a}_k - a_k| + |a_k||\widehat{a}_j - a_j| + |\widehat{a}_j - a_j||\widehat{a}_k - a_k|.
\end{align*}
Applying $\|\cdot\|_n$ on both sides and using Cauchy-Schwarz inequality concludes the proof of the first inequality. The second part follows from an application of Lemma~\ref{lem:psihat-minus-psi-part1}. 
\end{proof}
\begin{lemma}\label{lem:ajhat-minus-aj}
For any $1\le j\le d$,
\begin{align*}
&\left(\frac{1}{n}\sum_{i=1}^n (\widehat{a}_j(X_i) - a_j(X_i))^4\right)^{1/4}\\ ~&\qquad\le~ \frac{2\mathcal{D}_n^{\Sigma}}{1 - \mathcal{D}_n^{\Sigma}}\sup_{\theta:\|\theta\| \le 1}\left(\frac{1}{n}\sum_{i=1}^n {\{\theta^{\top}\Sigma^{-1/2}(X_i - \mu_X)\}^4}\right)^{1/4} ~+~ \frac{\|\Sigma^{-1/2}(\overline{X}_n - \mu_X)\|}{1 - \mathcal{D}_n^{\Sigma}}.
\end{align*}
\end{lemma}
\begin{proof}
Recall that
\[
\widehat{a}_j(x) := \frac{(x - \overline{X}_n)^{\top}\widehat{\Sigma}^{-1}e_j}{\sqrt{e_j^{\top}\widehat{\Sigma}^{-1}e_j}}\quad\mbox{and}\quad a_j(x) := \frac{(x - \mu_X)^{\top}\Sigma^{-1}e_j}{\sqrt{e_j^{\top}\Sigma^{-1}e_j}}.
\]
We will now bound $\widehat{a}_j(x) - a_j(x)$. Note that
\begin{align*}
\sup_{x}\left|\widehat{a}_j(x) - \frac{(x - \mu_{X})^{\top}\widehat{\Sigma}^{-1}e_j}{\sqrt{e_j^{\top}\widehat{\Sigma}^{-1}e_j}}\right| &= \left|\frac{(\overline{X}_n - \mu_X)^{\top}\widehat{\Sigma}^{-1}e_j}{\sqrt{e_j^{\top}\widehat{\Sigma}^{-1}e_j}}\right|\\ ~&\le~ \|\widehat{\Sigma}^{-1/2}(\overline{X}_n - \mu_X)\| ~\le~ \frac{\|\Sigma^{-1/2}(\overline{X}_n - \mu_X)\|}{1 - \mathcal{D}_n^{\Sigma}}. 
\end{align*}
Furthermore, 
\begin{align*}
\frac{(x - \mu_X)^{\top}\widehat{\Sigma}^{-1}e_j}{\sqrt{e_j^{\top}\widehat{\Sigma}^{-1}e_j}} - a_j(x) &= \frac{(x - \mu_X)^{\top}(\widehat{\Sigma}^{-1} - \Sigma^{-1})e_j}{\sqrt{e_j^{\top}\widehat{\Sigma}^{-1}e_j}} + \frac{(x - \mu_X)^{\top}\Sigma^{-1}e_j}{\sqrt{e_j^{\top}\Sigma^{-1}e_j}}\left[\sqrt{\frac{e_j^{\top}\Sigma^{-1}e_j}{e_j^{\top}\widehat{\Sigma}^{-1}e_j}} - 1\right].
\end{align*}
Combining these two steps yields
\begin{align*}
&\left(\frac{1}{n}\sum_{i=1}^n (\widehat{a}_j(X_i) - a_j(X_i))^4\right)^{1/4}\\ ~&\qquad\le~ \frac{1}{\sqrt{e_j^{\top}\widehat{\Sigma}^{-1}e_j}}\left(\frac{1}{n}\sum_{i=1}^n \left\{(\Sigma^{-1/2}(X_i - \mu_X))^{\top}(\Sigma^{1/2}\widehat{\Sigma}^{-1}\Sigma^{1/2} - I_d)\Sigma^{1/2}e_j\right\}^4\right)^{1/4}\\
~&\qquad\qquad+~ \left|\sqrt{\frac{e_j^{\top}\Sigma^{-1}e_j}{e_j^{\top}\widehat{\Sigma}^{-1}e_j}} - 1\right|\left(\frac{1}{n}\sum_{i=1}^n \frac{\{(X_i - \mu_X)^{\top}\Sigma^{-1}e_j\}^4}{(e_j^{\top}\Sigma^{-1}e_j)^2}\right)^{1/4} + \frac{\|\Sigma^{-1/2}(\overline{X}_n - \mu_X)\|}{1 - \mathcal{D}_n^{\Sigma}}.
\end{align*}
The first term can be further bounded by
\begin{align*}
&\frac{\|(\Sigma^{1/2}\widehat{\Sigma}^{-1}\Sigma^{1/2} - I_d)\Sigma^{1/2}e_j\|}{\sqrt{e_j^{\top}\widehat{\Sigma}^{-1}e_j}}\sup_{\theta:\|\theta\| \le 1}\left(\frac{1}{n}\sum_{i=1}^n \left\{(\Sigma^{-1/2}(X_i - \mu_X))^{\top}\theta\right\}^4\right)^{1/4}\\
&\qquad\le \frac{\mathcal{D}_n^{\Sigma}}{1 - \mathcal{D}_n^{\Sigma}}\sqrt{\frac{e_j^{\top}\Sigma^{-1}e_j}{e_j^{\top}\widehat{\Sigma}^{-1}e_j}}\sup_{\theta:\|\theta\| \le 1}\left(\frac{1}{n}\sum_{i=1}^n \left\{(\Sigma^{-1/2}(X_i - \mu_X))^{\top}\theta\right\}^4\right)^{1/4}.
\end{align*}
Similarly, the second term is bounded by
\[
\left|\sqrt{\frac{e_j^{\top}\Sigma^{-1}e_j}{e_j^{\top}\widehat{\Sigma}^{-1}e_j}} - 1\right|\sup_{\theta:\|\theta\| \le 1}\left(\frac{1}{n}\sum_{i=1}^n {\{(\Sigma^{-1/2}(X_i - \mu_X))^{\top}\theta\}^4}\right)^{1/4}
\]
Also, we use the fact that
\[
\left|\sqrt{\frac{e_j^{\top}\Sigma^{-1}e_j}{e_j^{\top}\widehat{\Sigma}^{-1}e_j}} - 1\right| ~\le~ \left|{\frac{e_j^{\top}\Sigma^{-1}e_j}{e_j^{\top}\widehat{\Sigma}^{-1}e_j}} - 1\right| ~\le~ \mathcal{D}_n^{\Sigma}\quad\Rightarrow\quad \sqrt{\frac{e_j^{\top}\Sigma^{-1}e_j}{e_j^{\top}\widehat{\Sigma}^{-1}e_j}} \le 1 + \mathcal{D}_n^{\Sigma}.
\]
Therefore,
\begin{align*}
&\left(\frac{1}{n}\sum_{i=1}^n (\widehat{a}_j(X_i) - a_j(X_i))^4\right)^{1/4}\\ &\qquad\le \frac{2\mathcal{D}_n^{\Sigma}}{1 - \mathcal{D}_n^{\Sigma}}\sup_{\theta:\|\theta\| \le 1}\left(\frac{1}{n}\sum_{i=1}^n {\{(\Sigma^{-1/2}(X_i - \mu_X))^{\top}\theta\}^4}\right)^{1/4} + \frac{\|\Sigma^{-1/2}(\overline{X}_n - \mu_X)\|}{1 - \mathcal{D}_n^{\Sigma}}.
\end{align*}
\end{proof}
\begin{lemma}\label{lem:rate-of-convergence-psihat-minus-psi}
Under the assumptions of Theorem~\ref{thm:Berry-Esseen-bound-partial-corr}, there exists a universal constant $C\in(0,\infty)$ such that with probability at least $1 - C/n$,
\[
\max_{1\le j < k\le d}\,\left(\frac{1}{n}\sum_{i=1}^n \left\{\widehat{\psi}_{jk}(X_i) - \psi_{jk}(X_i)\right\}^2\right)^{1/2} \le CK_x^6\sqrt{\frac{d + \log n}{n}} + CK_x^5\frac{d + \log n}{n},
\]
whenever the right hand side is less than 1 and $d \le \sqrt{n}$. 
\end{lemma}
\begin{proof}
 Applying inequality (3.9) of~\cite{mendelson2010empirical} with $F = \{x\mapsto \theta^{\top}\Sigma^{-1/2}x:\,\theta\in\mathbb{R}^d, \|\theta\|_{I_d} \le 1\}$ and $|I| = n$,  we conclude that,  with probability at least $1 - 1/n$,
\[
\sup_{\theta:\|\theta\| \le 1}\left(\frac{1}{n}\sum_{i=1}^n {\{(\Sigma^{-1/2}(X_i - \mu_X))^{\top}\theta\}^4}\right)^{1/4} \le CK_x\left(1 + \frac{d^{1/2}}{n^{1/4}}\right) \le CK_x\left(1 + \frac{d}{\sqrt{n}}\right),
\]
for some universal constant $C\in(0, \infty)$. Furthermore, Lemma~\ref{lem:concentration-of-covariance} yields that, with probability at least $1 - 1/n$,
\[
\mathcal{D}_n^{\Sigma} \le CK_x^2\sqrt{\frac{d + \log n}{n}} + CK_x^2\frac{d + \log n}{n}.
\]
Finally,
\[
\|\Sigma^{-1/2}(\overline{X}_n - \mu_X)\| \le 2\sup_{\theta\in\mathcal{N}_{1/2}}\left|\frac{1}{n}\sum_{i=1}^n \theta^{\top}\Sigma^{-1/2}(X_i - \mu_X)\right|,
\]
where $\mathcal{N}_{1/2}$ is the $1/2$-net of $\{\theta\in\mathbb{R}^d:\,\|\theta\| \le 1\}$ with cardinality $|\mathcal{N}_{1/2}| \le 5^d$. Hence with probability at least $1 - 1/n$,
\[
\|\Sigma^{-1/2}(\overline{X}_n - \mu_X)\| \le CK_x\sqrt{\frac{d + \log n}{n}}.
\]
Combining these inequalities with Lemma~\ref{lem:ajhat-minus-aj} concludes that with probability at least $1 - 3/n$,
\begin{align*}
&\max_{1\le j\le d}\,\left(\frac{1}{n}\sum_{i=1}^n (\widehat{a}_j(X_i) - a_j(X_i))^4\right)^{\frac{1}{4}}\\ 
&\quad\le CK_x^3\frac{d + \sqrt{n}}{\sqrt{n}}\left[\sqrt{\frac{d + \log n}{n}} + \frac{d + \log n}{n}\right] + CK_x\sqrt{\frac{d + \log n}{n}}\\
&\quad\le CK_x^3\left(1 + \frac{d}{\sqrt{n}}\right)\sqrt{\frac{d + \log n}{n}} + CK_x\sqrt{\frac{d + \log n}{n}}\\
&\quad\le CK_x^3\left(1 + \frac{d}{\sqrt{n}}\right)\sqrt{\frac{d + \log n}{n}},
\end{align*}
assuming  $d + \log n \le n$. Lemma~\ref{lem:psihat-minus-psi-part2} now yields with probability at least $1 - 3/n$,
\begin{align}
&\max_{1\le j < k\le d}\,\|\widehat{\psi}_{jk} - \psi_{jk}\|_n\\ 
&\quad\le C\max_{1\le j\le d}\left(\frac{1}{n}\sum_{i=1}^n a_j^4(X_i)\right)^{1/4}K_x^3\sqrt{\frac{d + \log n}{n}} + CK_x^6\frac{d + \log n}{n}\nonumber\\
&\qquad+ C\max_{1\le j\le d}\left|\frac{1}{n}\sum_{i=1}^n a_j^2(X_i) - \mathbb{E}[a_j^2(X)]\right|\label{eq:first-part-psihat-minus-psi}\\
&\qquad+ C\max_{1\le j < k\le d}|\theta_{jk} - \widehat{\theta}_{jk}|\max_{1\le j\le d}\left(\frac{1}{n}\sum_{i=1}^n \{a_j^2(X_i) - \mathbb{E}[a_j^2(X)]\}^2\right)^{1/2}.\nonumber
\end{align}
The calculations leading to~\eqref{eq:square-power-psi} now yields with probability at least $1 - 6/n$,
\[
\max_{1\le j\le d}\left(\frac{1}{n}\sum_{i=1}^n \left\{a_j^2(X_i) - \mathbb{E}[a_j^2(X)]\right\}^2\right)^{1/4} \le CK_x\left(1 + \sqrt{\frac{(\log(dn))^9}{n}}\right),
\]
and
\[
\max_{1\le j\le d}\left(\frac{1}{n}\sum_{i=1}^n a_j^4(X_i)\right)^{1/4} \le CK_x\left(1 + \sqrt{\frac{(\log(dn))^9}{n}}\right).
\]
Because $a_j^2(X_i)$ is sub-exponential with parameter $K_x^2$, using Theorem 2.8.1 of~\cite{Vershynin18}, we get with probability at least $1 - 1/n$
\[
\max_{1\le j\le d}\left|\frac{1}{n}\sum_{i=1}^n a_j^2(X_i) - \mathbb{E}[a_j^2(X)]\right| \le CK_x^2\sqrt{\frac{\log(dn)}{n}} + CK_x^2\frac{\log(dn)}{n}.
\]
Substituting these in~\eqref{eq:first-part-psihat-minus-psi} concludes with probability at least $1 - C/n$, 
\begin{equation}
\begin{split}
\max_{1\le j < k\le d}\,\|\widehat{\psi}_{jk} - \psi_{jk}\|_n &\le CK_x^4\sqrt{\frac{d + \log n}{n}} + CK_x^6\frac{d + \log n}{n}\\
&\qquad+ CK_x^2\sqrt{\frac{\log(dn)}{n}} + CK_x^2\frac{\log(dn)}{n}\\ 
&\qquad+ CK_x\max_{1\le j < k\le d}|\theta_{jk} - \widehat{\theta}_{jk}|.
\end{split}
\end{equation}
Using  $d \le \sqrt{n}$ as well as $K_x \ge 1$, we can simplify the terms above and write
\begin{equation}\label{eq:second-part-psihat-minus-psi}
\begin{split}
\max_{1\le j < k\le d}\,\|\widehat{\psi}_{jk} - \psi_{jk}\|_n &\le CK_x^6\sqrt{\frac{d + \log n}{n}} + CK_x^2\sqrt{\frac{\log(dn)}{n}}\\ &\quad+ CK_x\max_{1\le j < k\le d}|\theta_{jk} - \widehat{\theta}_{jk}|.
\end{split}
\end{equation}
The last term can be bounded based on Theorem~\ref{thm:linear-expansion-partial-corr} and~\eqref{eq:linear-rep-error-partial-corr} to get with probability at least $1 - 1/n$,
\[
\max_{1\le j < k\le d}|\theta_{jk} - \widehat{\theta}_{jk}| \le \max_{1\le j < k\le d}\left|\frac{1}{n}\sum_{i=1}^n \psi_{jk}(X_i)\right| + CK_x^4\frac{d + \log(n)}{n}.
\]
Because $\psi_{jk}(X_i), 1\le i\le n$ are sub-exponential with parameter $K_x^2$, using again Theorem 2.8.1 of~\cite{Vershynin18} yields that,  with probability at least $1 - 1/n$ and for some universal constant $C>0$,
\[
\max_{1\le j < k\le d}|\theta_{jk} - \widehat{\theta}_{jk}| ~\le~ CK_x\sqrt{\frac{\log(dn)}{n}} + CK_x^4\frac{d + \log n}{n}. 
\]
Substituting this in~\eqref{eq:second-part-psihat-minus-psi} concludes with probability at least $1 - C/n$,
\begin{align*}
\max_{1\le j < k\le d}\|\widehat{\psi}_{jk} - \psi_{jk}\|_n &\le CK_x^6\sqrt{\frac{d + \log n}{n}} + CK_x^2\sqrt{\frac{\log(dn)}{n}}\\ &\quad+ CK_x^2\sqrt{\frac{\log(dn)}{n}} + CK_x^5\frac{d + \log n}{n}.
\end{align*}
This concludes the proof.
\end{proof}

\section{Proof of the Auxiliary Results from Section~\ref{section::explicit} (Projection Parameters)}
\label{appendix:auxiliary.ols}

In this section, we provide key concentration inequalities for various quantities used in the  proofs of Theorems~\ref{thm:asymptotic-scaling-CLT-explicit},  Lemma~\ref{lem:sandwich-variance-error-control} and Theorem~\ref{thm:Berry-Esseen-bound-partial-corr}. Many of these results only requires weak moment conditions  and appear to be new. Therefore, they may be of independent interest.

\begin{lemma}\label{lem:concentration-of-covariance}
Under assumption~\ref{eq:covariate-subGaussian}, there exists a universal constant $C \in (0, \infty)$ such that
\[
\mathbb{P}\left(\mathcal{D}_n^{\Sigma} \le CK_x^2\sqrt{\frac{d + \log(1/\delta)}{n}} + CK_x^2\frac{d + \log(1/\delta)}{n}\right) \ge 1 - \delta\quad\mbox{for any}\quad\delta\in(0, 1).
\]
\end{lemma}
\begin{proof}
This results is standard: see, e.g., Theorem 4.7.1 of~\cite{Vershynin18} or Theorem 1 of~\cite{koltchinskii2017a}.
\end{proof}

\begin{proposition}\label{prop:D_Sigma-concentration}
Suppose the covariates have $q_x$ moments for some $q_x \ge 4$. Then there exists a constant $C_{q_x}$ depending only on $q_x$ such that 
\begin{itemize}[leftmargin=*]
\item under assumptions~\ref{eq:DGP} and~\ref{assump:finite-moment-covariates},
with probability at least $1-1/n-\delta$,
\begin{equation}\label{eq:D_moment}
\mathcal{D}_n^{\Sigma} \leq C_{q_x}K_x^2 \left\{  \frac{d\delta^{-2/q_x}}{n^{1-2/q_x}}  + \left(\frac{d}{n}\right)^{1-2 / q_x} \log ^{4}\left(\frac{n}{d}\right)+ \left(\frac{d}{n}\right)^{1/2} \right\}.
\end{equation}
With $\delta = (d/n^{1-2/q_x})^{q_x/8}$, the right hand side tends to zero if $d = o(n^{1-2/q_x})$.
\item under assumptions~\ref{eq:DGP} and~\ref{assump:finite-moment-independence-covariates},
with probability at least $1-1/n-\delta$, 
\begin{equation}\label{eq:D_structure}
\mathcal{D}_n^{\Sigma} \leq C_{q_x}K_x^2 \left\{ \frac{ \sqrt{2d\log(n/\delta)} }{n} +   \frac{(d/\delta)^{2/q_x}}{n^{1-2/q_x}}   + \left(\frac{d}{n}\right)^{1-2 / q_x} \log ^{4}\left(\frac{n}{d}\right)+ \left(\frac{d}{n}\right)^{1/2} \right\}.
\end{equation}
With $\delta = \sqrt{d/n}$, the right hand side tends to zero if $d = o(n)$.
\end{itemize}
\end{proposition}
\begin{proof}
See Lemma 3 of~\cite{yang2021finite} for a proof. The result in~\cite{yang2021finite} actually allows for $q_x \ge 2$.
\end{proof}
\begin{proposition}\label{prop:estimation-error-beta}
Under assumptions~\ref{eq:DGP},~\ref{assump:finite-moment-covariates},~\ref{eq:moments-errors} and~\ref{eq:bounded-asymptotic-variance}, there exists a constant $C\in(0,\infty)$ depending only on $q$ and $q_x$ such that for any $\delta\in(0, 1)$, with probability at least $1 - \delta$,
\begin{align*}
&\|\widehat{\beta} - \beta\|_{\Sigma}\\ 
&\;\le \left(1 - 7K_x\sqrt{\frac{d + 2\log(4/\delta)}{n}}\right)_+^{-1}\left[2\sqrt{\frac{d + \log(2/\delta)}{n\underline{\lambda}}} + \frac{CK_xK_yd^{1/2}}{(\delta/n)^{(q+q_x)/(qq_x)}n}\right].
\end{align*}
\end{proposition}
\begin{proof}
The result is similar to Theorem 4.1 of~\cite{oliveira2013lower}.
Note that
\[
\Sigma^{1/2}(\widehat{\beta} - \beta) = \Sigma^{1/2}\widehat{\Sigma}^{-1}\Sigma^{1/2}\frac{1}{n}\sum_{i=1}^n \Sigma^{-1/2}X_i(Y_i - X_i^{\top}\beta).
\]
This implies that
\begin{equation}\label{eq:deterministic-alternative}
\begin{split}
\|\widehat{\beta} - \beta\|_{\Sigma} &= \|V^{-1/2}\Sigma^{1/2}\Sigma^{1/2}(\widehat{\beta} - \beta)\|_2\\ 
&\le \lambda_{\max}(\Sigma^{1/2}\widehat{\Sigma}^{-1}\Sigma^{1/2})\left\|\frac{1}{n}\sum_{i=1}^n \Sigma^{-1/2}X_i(Y_i - X_i^{\top}\beta)\right\|_2.
\end{split}
\end{equation}
Under assumption~\ref{assump:finite-moment-covariates} with $q_x \ge 4$, one can verify that, for each $v \in \mathbb{R}^d$,
\[
\left( \mathbb{E}[| v^\top X_i|^4] \right)^{1/4} \leq K_x \left( \mathbb{E}[| v^\top X_i|^2] \right)^{1/2}
\]
Thus, Theorem 3.1 of~\cite{oliveira2013lower} applies and yields that
\begin{equation}\label{eq:eigenvalue-sigmahat}
\mathbb{P}\left(\lambda_{\max}(\Sigma^{1/2}\widehat{\Sigma}^{-1}\Sigma^{1/2}) \le \left(1 - 7K_x\sqrt{\frac{d + 2\log(4/\delta)}{n}}\right)_+^{-1}\right) \ge 1 - \frac{\delta}{2}.
\end{equation}
We now control $\|n^{-1}\sum_{i=1}^n \Sigma^{-1/2}X_i(Y_i - X_i^{\top}\beta)\|_2$ using Theorem 3.1 of~\cite{einmahl2008characterization}.
One can also use Theorem 3.5.1 of~\cite{yurinsky1985inequalities}. 
Take 
\begin{align*}
W_i &:= \Sigma^{-1/2}X_i(Y_i - X_i^{\top}\beta).
\end{align*}
The definition of $\beta$ implies $\mathbb{E}[W_i] = 0$. Then, $\left\|{n}^{-1}\sum_{i=1}^n \Sigma^{-1/2}X_i(Y_i - X_i^{\top}\beta)\right\|_{2} = \|n^{-1}\sum_{i=1}^n W_i\|_{2} $, so that
\begin{align*}
\mathbb{E}\left[\left\|\frac{1}{n}\sum_{i=1}^n \Sigma^{-1/2}X_i(Y_i - X_i^{\top}\beta)\right\|_2\right] &\le \left(\mathbb{E}\left[\left\|\frac{1}{n}\sum_{i=1}^n W_i\right\|_{2}^2\right]\right)^{1/2}\\ 
&= \frac{1}{\sqrt{n}}\left(\frac{1}{n}\sum_{i=1}^n \mbox{tr}(\mbox{Var}(W_i))\right)^{1/2} \le \frac{1}{\underline{\lambda}^{1/2}}\sqrt{\frac{d}{n}}.
\end{align*}
Theorem 3.1 of~\cite{einmahl2008characterization} with $\eta = \delta = 1$ implies that
\begin{equation}\label{eq:tail-inequality-beta-estimation}
\mathbb{P}\left(\left\|\frac{1}{n}\sum_{i=1}^n W_i\right\|_{2} \ge \frac{2}{\underline{\lambda}^{1/2}}\sqrt{\frac{d}{n}} + t\right) \le \exp\left(-\frac{n\underline{\lambda}t^2}{3}\right) + \frac{C}{(nt)^s}\sum_{i=1}^n \mathbb{E}[\|W_i\|_2^{s}], 
\end{equation}
for any $s > 2$ such that $\mathbb{E}[(W_i^{\top}W_i)^{s/2}]$ is finite. Here $C\in(0, \infty)$ is a constant depending only on $s$. It is clear that $\|W_i\|_2 = \|\Sigma^{-1/2}X_i\|_2|Y_i - X_i^{\top}\beta|$,
and because of assumption~\ref{eq:DGP} and~\ref{assump:finite-moment-covariates}, taking $s = (1/q + 1/q_x)^{-1}$, we obtain
\begin{align*}
\sum_{i=1}^n \mathbb{E}[\|W_i\|_2^{s}] ~&=~ n{d^{s/2}}\mathbb{E}\left[\left(\frac{1}{d}\sum_{j=1}^d (e_j^{\top}\Sigma^{-1/2}X_i)^2(Y_i - X_i^{\top}\beta)^2\right)^{s/2}\right]\\
~&\le~ n{d^{s/2}}\max_{1\le j\le d} \mathbb{E}\left[|e_j^{\top}\Sigma^{-1/2}X_i(Y_i - X_i^{\top}\beta)|^s\right]\\
~&\le~ nd^{s/2}K_x^{s}K_y^s.
\end{align*}
The last inequality follows from H{\"o}lder's inequality:
\begin{align*}
&\mathbb{E}[|e_j^{\top}\Sigma^{-1/2}X_i(Y_i - X_i^{\top}\beta)|^{q_xq/(q_x + q)}]\\ 
&\quad\le \left(\mathbb{E}[|e_j^{\top}\Sigma^{-1/2}X_i|^{q_x}]\right)^{q_xq/(q_x + q)}\left(\mathbb{E}[|Y_i - X_i^{\top}\beta|^{q}]\right)^{q_xq/(q_x + q)}\\
&\quad\le (K_xK_y)^{q_xq/(q_x+q)}
\end{align*}
Further, taking (for a possibly different constant $C\in(0,\infty)$)
\[
t ~=~ \sqrt{\frac{3\log(2/\delta)}{n\underline{\lambda}}} + CK_xK_y\frac{d^{1/2}}{\delta^{(q+q_x)/(qq_x)}n^{1-(q+q_x)/(qq_x)}},
\] 
in~\eqref{eq:tail-inequality-beta-estimation} yields
\[
\mathbb{P}\left(\left\|\frac{1}{n}\sum_{i=1}^n W_i\right\|_2 \ge 2\sqrt{\frac{d + \log(2/\delta)}{n\underline{\lambda}}} + \frac{CK_xK_yd^{1/2}}{\delta^{(q+q_x)/(qq_x)}n^{1-(q+q_x)/(qq_x)}}\right) \le \frac{\delta}{2}.
\]
Combining this with~\eqref{eq:eigenvalue-sigmahat} and~\eqref{eq:deterministic-alternative} completes the proof.
\end{proof}

\subsection{Estimation Error of Sandwich Variance Estimator}\label{appsubsec:sandwich-variance}
In this section we collect various bounds that are used in the Proof of Lemma~\ref{lem:sandwich-variance-error-control} about consistency of the sandwich variance estimator. We begin with a key, deterministic bound, implying that the rate of consistency  of the sandwich estimator depends on the quantities $ \mathbf{I}_1$, $\mathcal{M}_4$, $\|\widehat{\beta} - \beta\|_{\Sigma} $ and $\mathcal{D}_n^{\Sigma}$, which are handled separately in Lemma~\ref{lem:I-1-ideal-Sandwich-error}, Lemma~\ref{lem:M-4-bound}, Proposition~\ref{prop:estimation-error-beta} and Proposition~\ref{prop:D_Sigma-concentration}, respectively.
\begin{lemma}\label{lem:std-err-bound-deterministic}[Deterministic Bound for Sandwich Variance Estimator]
Define
\begin{equation}
    \begin{split}
        \mathcal{M}_4 ~&:=~ \frac{1}{\underline{\lambda}}\times\sup_{\theta\in\mathbb{R}^d,\|\theta\|_{I_d} = 1}\frac{1}{n}\sum_{i=1}^n (\theta^{\top}\Sigma^{-1/2}X_i)^4,\\
        \overline{V} ~&:=~ \frac{1}{n}\sum_{i=1}^n X_iX_i^{\top}(Y_i - X_i^{\top}\beta)^2,\\
        \mathbf{I}_1 ~&:=~ \sup_{\theta\in\mathbb{R}^d}\left|\frac{\theta^{\top}V^{-1/2}\overline{V}V^{-1/2}\theta}{\theta^{\top}\theta} - 1\right|.
    \end{split}
\end{equation}
For any $n\ge d$, if
\[
\mathcal{A} ~:=~ \left\{\mathcal{D}_n^{\Sigma} \le \frac{1}{2}\right\}\cap\left\{\mathbf{I}_1 \le \frac{1}{2}\right\}\cap\left\{\mathcal{M}_4^{1/2}\|\widehat{\beta} - \beta\|_{\Sigma} \le 1\right\}.
\]
holds true, then
\[
\sup_{\theta\in\mathbb{R}^d}\left|\frac{\theta^{\top}\widehat{\Sigma}_n^{-1}\widehat{V}\widehat{\Sigma}_n^{-1}\theta}{\theta^{\top}\Sigma^{-1}V\Sigma^{-1}\theta} - 1\right| \le \mathbf{I}_1 + 3\mathcal{M}_4^{1/2}\|\widehat{\beta} - \beta\|_{\Sigma} + 21\mathcal{D}_n^{\Sigma},
\]
\end{lemma}
\begin{remark}
Because the sandwich estimator is a complicated non-linear function of the estimators $(\widehat{\Sigma}_n, \widehat{V})$, and $\widehat{V}$ is not a sum of independent matrices, this result reduces the problem into basic components which are more easily controllable using results from sum of independent random variables/vectors/matrices.
\end{remark}
\begin{proof}[Proof of Lemma~\ref{lem:std-err-bound-deterministic}]
We begin by writing 
\begin{align*}
&\|(\Sigma^{-1}V\Sigma^{-1})^{-1/2}(\widehat{\Sigma}_n^{-1}\widehat{V}\widehat{\Sigma}_n^{-1})(\Sigma^{-1}V\Sigma^{-1})^{-1/2} - I_d\|_{\mathrm{op}}\\
&\qquad= \|V^{-1/2}\Sigma\widehat{\Sigma}_n^{-1}\widehat{V}\widehat{\Sigma}_n^{-1}\Sigma V^{-1/2} - I_d\|_{\mathrm{op}} ~=~ \|AA^{\top} - I_d\|_{\mathrm{op}},
\end{align*}
where $A = V^{-1/2}\Sigma\widehat{\Sigma}_n^{-1}\widehat{V}^{1/2}$ with $\widehat{V}^{1/2}$ representing the symmetric square root of $\widehat{V}$. Symmetry of $AA^{\top} - I_d$ and the definition of $\|\cdot\|_{\mathrm{op}}$ implies that
\[
\|AA^{\top} - I_d\|_{\mathrm{op}} = \sup_{\theta\in\mathbb{R}^d}\left|\frac{\|A^{\top}\theta\|^2}{\|\theta\|^2} - 1\right|.
\]
Using the fact that $$A^{\top}\theta ~=~ \widehat{V}^{1/2}V^{-1/2}V^{1/2}(\widehat{\Sigma}_n^{-1}\Sigma - I_d)V^{-1/2}\theta ~+~ \widehat{V}^{1/2}V^{-1/2}\theta,$$ we obtain that 
\begin{align*}
\left|\frac{\|A^{\top}\theta\|^2}{\|\theta\|^2} - 1\right| &\le \left|\frac{\|\widehat{V}^{1/2}V^{-1/2}\theta\|^2}{\|\theta\|^2} - 1\right|\\ 
&\qquad+ \frac{\|\widehat{V}^{1/2}V^{-1/2}V^{1/2}(\widehat{\Sigma}_n^{-1}\Sigma - I_d)V^{-1/2}\theta\|^2}{\|\theta\|^2}\\
&\qquad+ 2\times\frac{\theta^{\top}V^{-1/2}\widehat{V}(\widehat{\Sigma}_n^{-1}\Sigma - I_d)V^{-1/2}\theta}{\|\theta\|^2}\\
&\le \left|\frac{\theta^{\top}V^{-1/2}\widehat{V}V^{-1/2}\theta}{\theta^{\top}\theta} - 1\right|\\ 
&\qquad+ \|\widehat{V}^{1/2}V^{-1/2}\|_{\mathrm{op}}^2\frac{\|V^{1/2}(\widehat{\Sigma}_n^{-1}\Sigma - I_d)V^{-1/2}\theta\|^2}{\|\theta\|^2}\\
&\qquad+ 2\times\frac{\|V^{-1/2}\widehat{V}V^{-1/2}\theta\|}{\|\theta\|}\times\frac{\|V^{1/2}(\widehat{\Sigma}_n^{-1}\Sigma - I_d)V^{-1/2}\theta\|}{\|\theta\|}.
\end{align*}
Taking the supremum over $\theta\in\mathbb{R}^d$ yields
\begin{equation}\label{eq:main-inequality-sandwich}
\begin{split}
\|AA^{\top} - I_d\|_{\mathrm{op}} &\le \sup_{\theta\in\mathbb{R}^d}\left|\frac{\|\widehat{V}^{1/2}V^{-1/2}\theta\|^2}{\|\theta\|^2} - 1\right|\\ &\qquad+ \|\widehat{V}^{1/2}V^{-1/2}\|_{\mathrm{op}}^2\left\{\|\widehat{\Sigma}_n^{-1}\Sigma - I_d\|_{\mathrm{op}}^2 + 2\|\widehat{\Sigma}_n^{-1}\Sigma - I_d\|_{\mathrm{op}}\right\}\\
&= \mathbf{I} ~+~ (\mathbf{I} + 1)\left\{\mathbf{II}^2 + 2\mathbf{II}\right\},
\end{split}
\end{equation}
where
\begin{equation}\label{eq:decomposition-sandwich-error}
\mathbf{I} := \sup_{\theta\in\mathbb{R}^d}\left|\frac{\|\widehat{V}^{1/2}V^{-1/2}\theta\|^2}{\|\theta\|^2} - 1\right|\quad\mbox{and}\quad \mathbf{II} := \|\widehat{\Sigma}_n^{-1}\Sigma - I_d\|_{\mathrm{op}}.
\end{equation}
Based on this inequality, it suffices to bound $\mathbf{I}$ and $\mathbf{II}$.
Regarding $\mathbf{II}$, we note that $\mathbf{II} \le \mathcal{D}_n^{\Sigma}/(1 - \mathcal{D}_n^{\Sigma})$ and hence we obtain the following inclusion of events
\begin{equation}\label{eq:First-implication}
\{\mathcal{D}_n^{\Sigma} \le 1/2\} \subset  \{\mathbf{II}\le1 \} \subset \left\{ \|AA^{\top} - I_d\|_{\mathrm{op}} \le \mathbf{I} + 6(\mathbf{I} + 1)\mathcal{D}_n^{\Sigma} \right\}.
\end{equation}
Note that the definition of $\overline{V}$ differs from $\widehat{V}$ in the use of $\beta$ in place of $\widehat{\beta}$ which yields an average of independent random matrices. Observe that
\[
\mathbf{I} ~\le~ \sup_{\theta\in\mathbb{R}^d}\left|\frac{\|\overline{V}^{1/2}V^{-1/2}\theta\|^2}{\|\theta\|^2} - 1\right| ~+~ \sup_{\theta\in\mathbb{R}^d}\left|\frac{\theta^{\top}V^{-1/2}(\widehat{V} - \overline{V})V^{-1/2}\theta}{\theta^{\top}\theta}\right| ~=:~ \mathbf{I}_1 + \mathbf{I}_{2}.
\]
Lemma~\ref{lem:operator-norm-Vhat-Vbar} proves
\[
\mathbf{I}_2 \le \frac{\mathcal{M}_4}{2}\|\widehat{\beta} - \beta\|^2_{\Sigma} + \sqrt{2}\mathcal{M}_4^{1/2}\|\widehat{\beta} - \beta\|_{\Sigma}\sqrt{1 + \mathbf{I}_1}.
\]
Thus, the event $\left\{\mathbf{I}_1 \le \frac{1}{2}\right\}\cap\{\mathcal{M}_4^{1/2}\|\widehat{\beta} - \beta\|_{\Sigma} \le 1\}$ implies that 
\[
\mathbf{I} \le \mathbf{I}_1 + 3\mathcal{M}_4^{1/2}\|\widehat{\beta} - \beta\|_{\Sigma} \le 1/2 + 3 = 7/2.
\]
This combined with~\eqref{eq:First-implication} implies that if $\mathcal{A}$ holds then
\begin{equation}\label{eq:main-decomposition-sandwich}
\|AA^{\top} - I_d\|_{\mathrm{op}} \le \mathbf{I}_1 + 3\mathcal{M}_4^{1/2}\|\widehat{\beta} - \beta\|_{\Sigma} + 21\mathcal{D}_n^{\Sigma}.
\end{equation}
\end{proof}
\begin{lemma}\label{lem:operator-norm-Vhat-Vbar}
Let
\[
\overline{V} := \frac{1}{n}\sum_{i=1}^n X_iX_i^{\top}(Y_i - X_i^{\top}\beta)^2\quad\mbox{and}\quad \widehat{V} := \frac{1}{n}\sum_{i=1}^n X_iX_i^{\top}(Y_i - X_i^{\top}\widehat{\beta})^2.
\]
Then
\begin{align*}
\|V^{-1/2}(\widehat{V} - \overline{V})V^{-1/2}\|_{\mathrm{op}} &\le \frac{1}{2}\mathcal{M}_4\|\widehat{\beta} - \beta\|_{\Sigma}^2\\ &\qquad+ \sqrt{2}\mathcal{M}_4^{1/2}\|\widehat{\beta} - \beta\|_{\Sigma}\|\overline{V}^{1/2}V^{-1/2}\|_{\mathrm{op}}.
\end{align*}
\end{lemma}
\begin{proof}[Proof of Lemma~\ref{lem:operator-norm-Vhat-Vbar}]
The definition of the operator norm and the symmetry of $V^{-1/2}(\widehat{V} - \overline{V})V^{-1/2}$ implies
\[
\|V^{-1/2}(\widehat{V} - \overline{V})V^{-1/2}\|_{\mathrm{op}} = \sup_{\theta\in\mathbb{R}^d:\,\|\theta\|_{I_d} = 1}\,\left|\theta^{\top}V^{-1/2}(\widehat{V} - \overline{V})V^{-1/2}\theta\right|.
\]
Fix $\theta\in\mathbb{R}^d$ such that $\|\theta\|_{I_d} = 1$.
Expanding $(Y_i - X_i^{\top}\widehat{\beta})^2$ in $\widehat{V}$ by adding and subtracting $X_i^{\top}\beta$ to $X_i^{\top}\widehat{\beta}$ yields
\begin{align*}
&\left|{\theta^{\top}V^{-1/2}(\widehat{V} - \overline{V})V^{-1/2}\theta}\right|\\ 
&\quad\le \frac{1}{n}\sum_{i=1}^n {(\theta^{\top}V^{-1/2}X_i)^2}\left[(X_i^{\top}(\widehat{\beta} - \beta))^2 + 2|Y_i - X_i^{\top}\beta|\times|X_i^{\top}(\widehat{\beta} - \beta)|\right]\\
&\quad\le \frac{1}{n}\sum_{i=1}^n {(\theta^{\top}V^{-1/2}X_i)^2}(X_i^{\top}(\widehat{\beta} - \beta))^2\\ 
&\quad\qquad+ \frac{1}{n}\sum_{i=1}^n {(\theta^{\top}V^{-1/2}X_i)^2}\left[\frac{|Y_i - X_i^{\top}\beta|^2}{L} + L|X_i^{\top}(\widehat{\beta} - \beta)|^2\right]\\
&\quad\le \frac{(1 + L)}{n}\sum_{i=1}^n {(\theta^{\top}V^{-1/2}X_i)^2}(X_i^{\top}(\widehat{\beta} - \beta))^2 + \frac{1}{nL}\sum_{i=1}^n {(\theta^{\top}V^{-1/2}X_i)^2}(Y_i - X_i^{\top}\beta)^2\\
&\quad\le \frac{(1 + L)}{n}\sum_{i=1}^n {(\theta^{\top}V^{-1/2}X_i)^2}(X_i^{\top}(\widehat{\beta} - \beta))^2 + \frac{1}{L}\times{\theta^{\top}V^{-1/2}\bar{V}V^{-1/2}\theta}.
\end{align*}
We write $X_i^{\top}(\widehat{\beta} - \beta)$ as $(\Sigma^{-1/2}X_i)^{\top}(\Sigma^{1/2}(\widehat{\beta} - \beta))$ and bound the first term on the right hand side as
\[
\frac{1}{n}\sum_{i=1}^n {(\theta^{\top}V^{-1/2}X_i)^2}(X_i^{\top}(\widehat{\beta} - \beta))^2 \le \frac{1}{n}\sum_{i=1}^n (\theta^{\top}V^{-1/2}X_i)^2(u^{\top}\Sigma^{-1/2}X_i)^2\|\widehat{\beta} - \beta\|_{\Sigma}^2,
\]
where $u = \Sigma^{1/2}(\widehat{\beta} - \beta)/\|\widehat{\beta} - \beta\|_{\Sigma}$ is of unit norm. The right hand side (without the factor $\|\widehat{\beta} - \beta\|_{\Sigma}^2$) can be further bounded by
\begin{align*}
&\sup_{\theta, u\in\mathbb{R}^d}\,\frac{1}{n}\sum_{i=1}^n \frac{(\theta^{\top}V^{-1/2}X_i)^2(u^{\top}\Sigma^{-1/2}X_i)^2}{(\theta^{\top}\theta)(u^{\top}u)}\\
&\quad= \sup_{\gamma, u\in\mathbb{R}^d}\,\frac{1}{n}\sum_{i=1}^n \frac{(\gamma^{\top}\Sigma^{-1/2}X_i)^2(u^{\top}\Sigma^{-1/2}X_i)^2}{(\gamma^{\top}\Sigma^{-1/2}V\Sigma^{-1/2}\gamma)(u^{\top}u)}\\
&\quad\le \sup_{\gamma, u\in\mathbb{R}^d}\,\frac{1}{n}\sum_{i=1}^n \frac{(\gamma^{\top}\Sigma^{-1/2}X_i)^2(u^{\top}\Sigma^{-1/2}X_i)^2}{(\gamma^{\top}\gamma)(u^{\top}u)}\times\frac{(\gamma^{\top}\gamma)}{(\gamma^{\top}\Sigma^{-1/2}V\Sigma^{-1/2}\gamma)}\\
&\quad\le \sup_{\theta\in\mathbb{R}^d,\|\theta\|_{I_d} \le 1}\,\frac{1}{n}\sum_{i=1}^n (\theta^{\top}\Sigma^{-1/2}X_i)^4\times\frac{1}{\underline{\lambda}}.
\end{align*}
Combining these to bounds into~\eqref{eq:First-bound-Operator-norm-Vhat-Vbar} concludes
\begin{equation}\label{eq:First-bound-Operator-norm-Vhat-Vbar}
\|V^{-1/2}(\widehat{V} - \overline{V})V^{-1/2}\|_{\mathrm{op}} \le \frac{(1 + L)}{2}\mathcal{M}_4\|\widehat{\beta} - \beta\|_{\Sigma}^2 + \frac{\|\overline{V}^{1/2}V^{-1/2}\|_{\mathrm{op}}^2}{L},
\end{equation}
Minimizing over $L > 0$ concludes the result.
\end{proof}

\begin{lemma}\label{lem:I-1-ideal-Sandwich-error}
Suppose assumptions~\ref{eq:DGP},~\ref{eq:bounded-asymptotic-variance},
and~\ref{eq:moments-errors} with $q \ge 2$ holds true.
Recall that $q_{xy} = qq_x/(q + q_x)$. Under~\ref{assump:finite-moment-covariates} with $q_x \ge 2$, if $q_{xy} \ge 2$, then we have with probability at least $1 - 1/n - d^{q_{xy}/4}/n^{q_{xy}/4 - 1/2}$,
\begin{equation}\label{eq:Ideal-sandwich-finite-moment}
\mathbf{I}_1 \le C\overline{\lambda}K_x^2K_y^2\left[\frac{d^{1/2}}{n^{1/2 - 1/q_{xy}}} + \left(\frac{d}{n}\right)^{1 - 2/\min\{q_{xy}, 4\}}\log^4\left(\frac{n}{d}\right)\right].
\end{equation}
The right hand side converges to zero if and only if $d = o(n^{1 - 2/q_{xy}})$.
Under~\ref{assump:finite-moment-independence-covariates} with $q_x \ge 2$, if $q_{xy} \ge 2$, then we have with probability at least $1 - 1/n - d/n^{1-1/q} - d^{1/q_x}/n^{1-1/q_{xy}}$,
\begin{equation}\label{eq:Ideal-sandwich-finite-moment-independence}
\begin{split}
\mathbf{I}_1 &\le C\overline{\lambda}K_x^2K_y^2\left[\frac{d^{1-1/q}}{n^{(1-1/q)^2}} + \frac{d^{-1/2}\log n}{n^{(1-1/q)^2}} + \frac{d^{(1-1/q_{xy})/q_x}}{n^{(1-1/q_{xy})^2}}\right]\\
&\qquad+ C\overline{\lambda}K_x^2K_y^2\left(\frac{d}{n}\right)^{1 - 2/\min\{q_{xy}, 4\}}\log^4\left(\frac{n}{d}\right).
\end{split}
\end{equation}
The right hand side converges to zero if and only if $d = o(n^{1-1/q})$ and $d = o(n^{q_x - q_x/q_{xy}})$.
Here in both cases, $C\in(0,\infty)$ is a constant depending only on $q, q_x$.
\end{lemma}
\begin{proof}[Proof of Lemma~\ref{lem:I-1-ideal-Sandwich-error}]
We apply Theorem 1.1 of~\cite{tikhomirov2017sample} on the random vectors $W_i = V^{-1/2}X_i(Y_i - X_i^{\top}\beta)$. Note that $\mathbb{E}[W_i] = 0$ and $\mathbb{E}[W_iW_i^{\top}] = I_d$ by the definition of $V$. Furthermore, for any $a\in\mathbb{R}^d$ such that $\|a\|_{I_d} \le 1$, we have that, under assumptions~\ref{assump:finite-moment-covariates} and~\ref{eq:moments-errors},
\begin{align*}
\mathbb{E}[|a^{\top}W_i|^{qq_x/(q + q_x)}] &\le \left(\mathbb{E}[|a^{\top}V^{-1/2}X_i|^{q_x}]\right)^{q/(q + q_x)}\left(\mathbb{E}[|Y_i - X_i^{\top}\beta|^{q}]\right)^{q_x/(q + q_x)}\\
&\le (\overline{\lambda}^{1/2}K_xK_y)^{qq_x/(q + q_x)},
\end{align*}
where the first bound follows from H\"{o}lder's inequality. Hence, Theorem 1.1 of~\cite{tikhomirov2017sample} applies with $p = qq_x/(q + q_x)$. Therefore, for a constant $C\in(0,\infty)$ depending on $q, q_x$, with probability at least $1 - 1/n$,
\[
\mathbf{I}_1 \le \frac{C}{n}\max_{1\le i\le n}\|W_i\|_2^2 + C\overline{\lambda}K_x^2K_y^2\left(\frac{d}{n}\right)^{1 - 2(q+q_x)/(qq_x)}\log^4\left(\frac{n}{d}\right),
\]
if $qq_x/(q + q_x)\in(2, 4]$. If $qq_x/(q + q_x) > 4$, then with probability at least $1 - 1/n$,
\[
\mathbf{I}_1 \le \frac{C}{n}\max_{1\le i\le n}\|W_i\|_2^2 + C\overline{\lambda}K_x^2K_y^2\sqrt{\frac{d}{n}},
\]
Combining these two bounds, we write with probability with at least $1 - 1/n$,
\begin{equation}\label{eq:Tikhomirov-result}
\mathbf{I}_1 \le \frac{C}{n}\max_{1\le i\le n}\|W_i\|_2^2 + C\overline{\lambda}K_x^2K_y^2\left(\frac{d}{n}\right)^{1 - 2/\min\{q_{xy}, 4\}}\log^4\left(\frac{n}{d}\right).
\end{equation}
 We will now bound $\max_{1\le i\le n}\|W_i\|_2^2$ under~\ref{assump:finite-moment-covariates} and~\ref{assump:finite-moment-independence-covariates}. Under~\ref{assump:finite-moment-covariates}, note that
\begin{align*}
\mathbb{E}\left[\|W_i\|_2^{q_{xy}}\right] &= d^{q_{xy}/2}\mathbb{E}\left[\left(\frac{1}{d}\sum_{j=1}^d |e_j^{\top}W_i|^2\right)^{q_{xy}/2}\right]\\
&= d^{q_{xy}/2}\max_{1\le j\le d}\mathbb{E}\left[|e_j^{\top}W_i|^{q_{xy}}\right]\\
&\le d^{q_{xy}/2}\left(\overline{\lambda}^{1/2}K_xK_y\right)^{q_{xy}}.
\end{align*}
Therefore, by Markov's inequality and the union bound,
\[
\mathbb{P}\left(\max_{1\le i\le n}\|W_i\|_2 \ge d^{1/2}\overline{\lambda}^{1/2}K_xK_y(\delta/n)^{-1/q_{xy}}\right) \le \delta,
\]
Hence, under~\ref{assump:finite-moment-covariates}, with probability at least $1 - \delta - 1/n$,
\[
\mathbf{I}_1 \le \frac{Cd}{n}\overline{\lambda}K_x^2K_y^2(\delta/n)^{-2/q_{xy}} + C\overline{\lambda}K_x^2K_y^2\left(\frac{d}{n}\right)^{1 - 2/\min\{q_{xy}, 4\}}\log^4\left(\frac{n}{d}\right).
\]
Choosing $\delta = d^{q_{xy}/4}/n^{q_{xy}/4 - 1/2}$, the previous bound along with~\eqref{eq:Tikhomirov-result} yields~\eqref{eq:Ideal-sandwich-finite-moment}.

To prove the result under~\ref{assump:finite-moment-independence-covariates}, note that Assumption~\ref{eq:bounded-asymptotic-variance} implies
\[
\|W_i\|_2 = \|V^{-1/2}X_i(Y_i - X_i^{\top}\beta)\|_2 \le \overline{\lambda}^{1/2}\|\Sigma^{-1/2}X_i\|_2|Y_i - X_i^{\top}\beta|.
\]
This implies
\[
\max_{1\le i\le n}\|W_i\|_2 \le \overline{\lambda}^{1/2}\max_{1\le i\le n}\|\Sigma^{-1/2}X_i\|_2\max_{1\le i\le n}|Y_i - X_i^{\top}\beta|.
\]
By assumption~\ref{eq:moments-errors}, 
\begin{equation}\label{eq:max-errors_tail}
\mathbb{P}\left(\max_{1\le i\le n}|Y_i - X_i^{\top}\beta| \ge K_y(2n/\delta)^{1/q}\right) \le \frac{\delta}{2}.
\end{equation}
By~\ref{assump:finite-moment-independence-covariates}, $\Sigma^{-1/2}X_i = Z_i$ has independent coordinates and hence,
\[
\max_{1\le i\le n}\|\Sigma^{-1/2}X_i\|_2^2 \le \max_{1\le i\le n}\left|\sum_{j=1}^d \{Z_{i}^2(j) - \mathbb{E}[Z_i^2(j)]\}\right| + \max_{1\le j\le n}\sum_{j=1}^d \mathbb{E}[Z_i^2(j)].
\]
For any $1\le i\le n$, Eq. (1.9) of~\cite{rio2017constants} yields with probability at least $1 - \delta$,
\[
\left|\sum_{j=1}^d \{Z_i^2(j) - \mathbb{E}[Z_i^2(j)]\}\right| \le K_x^2\sqrt{d\log(1/\delta)} + K_x^2(3 + q_x/3)d^{2/q_x}\delta^{-2/q_x}.
\]
Hence, with probability at least $1 - \delta/2$,
\begin{equation}\label{eq:independence-maximum}
\max_{1\le i\le n}\|\Sigma^{-1/2}X_i\|_2^2 \le K_x^2d + K_x^2\sqrt{d\log(n/\delta)} + K_x^2(3 + q_x/3)(2nd/\delta)^{2/q_x}.
\end{equation}
Combining inequalities~\eqref{eq:max-errors_tail} and~\eqref{eq:independence-maximum}, we obtain with probability at least $1 - \delta$,
\[
\max_{1\le i\le n}\|W_i\|_2 \le \overline{\lambda}^{1/2}K_xK_y\left(\frac{2n}{\delta}\right)^{1/q}\left[2d + \log(n/\delta) + (3 + q_x/3)\left(\frac{2nd}{\delta}\right)^{1/q_x}\right].
\]
Taking $\delta = \max\{d/n^{1-1/q}, d^{1/q_x}/n^{1-1/q_{xy}}\}$, this with~\eqref{eq:Tikhomirov-result} yields~\eqref{eq:Ideal-sandwich-finite-moment-independence}.
\end{proof}

\begin{lemma}\label{lem:M-4-bound}
Under assumptions~\ref{eq:DGP} and~\ref{assump:finite-moment-covariates} with $q_x \ge 4$, we have
\begin{equation}\label{eq:finite-moment-M_4}
\mathbb{E}[\mathcal{M}_4] \le \frac{K_x^4}{\underline{\lambda}}\left[1 + C\frac{d^2\log(n)}{n^{1-4/q_x}} + C\left(\frac{d^2\log(n)}{n^{1-4/q_x}}\right)^{1/2}\right].
\end{equation}
The right hand side converges to zero only if $d = o(n^{1/2 - 2/q_x})$, ignoring the log terms. This requirement reduces to $d = o(n^{1/2})$ if $q_x \ge \log n$.\\
Under assumption~\ref{eq:DGP} and~\ref{assump:finite-moment-independence-covariates} with $q_x \ge 4$, we have
\begin{equation}\label{eq:finite-moment-independent-M_4}
\begin{split}
\mathbb{E}[\mathcal{M}_4] &\le \frac{K_x^4}{\underline{\lambda}}\left[1 + C\frac{\log n}{n}(d^2 + \log^2(n) + (\log n)^{2 - 4/q_x}(nd)^{4/q_x})\right]\\
&\quad+ C\frac{K_x^4}{\underline{\lambda}}\sqrt{\frac{\log(n)}{n}}\left[d + \log(n) + (\log n)^{1 - 2/q_x}(nd)^{2/q_x}\right].
\end{split}
\end{equation}
The right hand side converges to zero only if $d = o(n^{1/2}\wedge n^{q_x/4 - 1})$ (ignoring logarithmic terms). This requirement becomes $d = o(n^{1/2})$ if $q_x = 6$.
\end{lemma}
\begin{proof}[Proof of Lemma~\ref{lem:M-4-bound}]
Define
\[
\widebar{\mathcal{M}}_4 ~:=~ \sup_{\|\theta\|_{I_d} \le 1}\,\left|\frac{1}{n}\sum_{i=1}^n \left\{(\theta^{\top}\Sigma^{-1/2}X_i)^4 - \mathbb{E}[(\theta^{\top}\Sigma^{-1/2}X_i)^4]\right\}\right|.
\]
The set $\{\theta:\,\|\theta\|_{I_d} \le 1\}$ is a symmetric convex body of radius $1$ and has a modulus of convexity of power type $2$. Thus Theorem 3 of~\cite{guedon2007lp} applies and yields (for a universal constant $C$),
\begin{equation}\label{eq:Guedon-Rudelson-Bound}
    \begin{split}
        \mathbb{E}\left[\widebar{\mathcal{M}}_4\right] &\le  C\frac{\log(n)}{n}\mathbb{E}\left[\max_{1\le i\le n}\|X_i\|_{\Sigma^{-1}}^4\right]\\ &\quad+ C\sqrt{\frac{\log(n)}{n}}\left(\mathbb{E}\left[\max_{1\le i\le n}\|X_i\|_{\Sigma^{-1}}^4\right]\right)^{1/2} \sup_{\|\theta\|_{I_d}\le1}\left(\mathbb{E}[|\theta^{\top}\Sigma^{-1/2}X|^4]\right)^{1/2}\\
        &\le C\frac{\log(n)}{n}\left(\mathbb{E}\left[\max_{1\le i\le n}\|X_i\|_{\Sigma^{-1}}^{q_x}\right]\right)^{4/q_x}\\
        &\quad+ CK_x^2\sqrt{\frac{\log(n)}{n}}\left(\mathbb{E}\left[\max_{1\le i\le n}\|X_i\|_{\Sigma^{-1}}^{q_x}\right]\right)^{2/q_x}.        
    \end{split}
\end{equation}
Note that assumption~\ref{assump:finite-moment-covariates} implies
\begin{equation}\label{eq:moment-bound-M_4}
    \begin{split}
        \mathbb{E}\left[\max_{1\le i\le n}\|X_i\|_{\Sigma^{-1}}^{q_x}\right] &\le n\mathbb{E}\left[\|X_i\|_{\Sigma^{-1}}^{q_x}\right]\\ &\le nd^{q_x/2}\mathbb{E}\left[\left( \frac{1}{d}\sum_{j=1}^d (e_j^{\top}\Sigma^{-1/2}X_i)^2\right)^{q_x/2}\right]\\
        &\le nd^{q_x/2}\left(\frac{1}{d}\sum_{j=1}^{d} \mathbb{E}[|e_j^{\top}\Sigma^{-1/2}X_i|^{q_x}]\right) \le nd^{q_x/2}K_x^{q_x}.        
    \end{split}    
\end{equation}
Combining inequalities~\eqref{eq:Guedon-Rudelson-Bound} and~\eqref{eq:moment-bound-M_4}, we obtain
\[
\mathbb{E}\left[\widebar{\mathcal{M}}_4\right] \le CK_x^4\frac{d^2\log(n)}{n^{1-4/q_x}} + CK_x^4\left(\frac{d^2\log(n)}{n^{1-4/q_x}}\right)^{1/2}.
\]
Therefore, 
\[
\mathbb{E}[\mathcal{M}_4] \le \frac{2K_x^4}{\underline{\lambda}}\left[1 + C\frac{d^2\log(n)}{n^{1-4/q_x}} + C\left(\frac{d^2\log(n)}{n^{1-4/q_x}}\right)^{1/2}\right].
\]
This proves~\eqref{eq:finite-moment-M_4}. To prove the bound under assumption~\ref{assump:finite-moment-independence-covariates}, we note that
\begin{align*}
\max_{1\le i\le n}\|X_i\|_{\Sigma^{-1}}^{4} &= \max_{1\le i\le n}\left(\|X_i\|_{\Sigma^{-1}}^2 - \mathbb{E}[\|X_i\|_{\Sigma^{-1}}^2] + \mathbb{E}[\|X_i\|_{\Sigma^{-1}}^2]\right)^2\\
&\le 8\max_{1\le i\le n}\left|\sum_{j=1}^d \{Z_i^2(j) - \mathbb{E}[Z_i^2(j)]\}\right|^{2}\\ 
&\quad
+ 8d^{2}\max_{1\le i\le n}\left(\frac{1}{d}\sum_{j=1}^d \mathbb{E}[Z_i^2(j)]\right)^{2}.
\end{align*}
For each $1\le i\le n$, the random variables $Z_i(j), 1\le j\le d$ are independent and hence using the fact $\mathbb{E}[|Z_i(j)|^{q_x}] \le K_x^{q_x}$ (which follows from assumption~\ref{assump:finite-moment-independence-covariates}), we get from Proposition 3.1 of~\cite{gine2000exponential}
\begin{align*}
&\mathbb{E}\left[\max_{1\le i\le n}\left|\sum_{j=1}^d \{Z_i^2(j) - \mathbb{E}[Z_i^2(j)]\}\right|^{2}\right]\\ &\quad\le C\left(\mathbb{E}\left[\max_{1\le i\le n}\left|\sum_{j=1}^d \{Z_i^2(j) - \mathbb{E}[Z_i^2(j)]\}\right|\right]\right)^{2}\\
&\qquad+ C\max_{1\le i\le n}\sum_{j=1}^d \mathbb{E}[Z_i^4(j)] + C\mathbb{E}\left[\max_{1\le i\le n}\max_{1\le j\le d}|Z_i(j)|^{4}\right]\\
&\quad\le C\left(\mathbb{E}\left[\max_{1\le i\le n}\left|\sum_{j=1}^d \{Z_i^2(j) - \mathbb{E}[Z_i^2(j)]\}\right|\right]\right)^{2}\\
&\qquad+ CK_x^{4}d + C\left(\mathbb{E}\left[\sum_{i=1}^n \sum_{j=1}^d |Z_i(j)|^{q_x}\right]\right)^{4/q_x}\\
&\quad\le C\left(\mathbb{E}\left[\max_{1\le i\le n}\left|\sum_{j=1}^d \{Z_i^2(j) - \mathbb{E}[Z_i^2(j)]\}\right|\right]\right)^{2} + CK_x^{4}d + CK_x^4(nd)^{4/q_x}.
\end{align*}
Furthermore, Proposition B.1 of~\cite{kuchibhotla2019least} yields
\begin{align*}
&\mathbb{E}\left[\max_{1\le i\le n}\left|\sum_{j=1}^d \{Z_i^2(j) - \mathbb{E}[Z_i^2(j)]\}\right|\right]\\ 
&\quad\le CK_x^2d^{1/2}\sqrt{\log(n)} + CK_x^2(\log n)^{1-2/q_x}(nd)^{2/q_x}. 
\end{align*}
Therefore, 
\[
\mathbb{E}\left[\max_{1\le i\le n}\|X_i\|_{\Sigma^{-1}}^4\right] \le 8d^2K_x^4 + CK_x^4d\log(n) + CK_x^4(\log n)^{2 - 4/q_x}(nd)^{4/q_x}.
\]
Using this inequality in the first inequality of~\eqref{eq:Guedon-Rudelson-Bound}, we obtain
\begin{align*}
\mathbb{E}\left[\widebar{\mathcal{M}}_4\right] &\le CK_x^4\frac{\log n}{n}\left[d^2 + d\log(n) + (\log n)^{2 - 4/q_x}(nd)^{4/q_x}\right]\\
&\quad+ CK_x^4\sqrt{\frac{\log(n)}{n}}\left[d^2 + d\log(n) + (\log n)^{2 - 4/q_x}(nd)^{4/q_x}\right]^{1/2}.
\end{align*}
This completes the proof of~\eqref{eq:finite-moment-independent-M_4}.
\end{proof}

\section{Proofs of Auxiliary Results for Section~\ref{section::partial} (Partial Correlations)}
\label{appendix:auxiliary.partial}
We begin by bounding $\mathcal{D}_n^{\Sigma}$ in terms of the intermediate Gram matrix $\widetilde{\Sigma}$. These bounds and associated derivations are be used repeatedly in the proofs of the results from Section~\ref{section::partial}.

\begin{proposition}\label{prop:bounding-D-sigma}
For every $n\ge1$,
\[
\mathcal{D}_n^{\Sigma} ~\le~ \|\Sigma^{-1/2}\widetilde{\Sigma}\Sigma^{-1/2} - I_d\|_{\mathrm{op}} + \|\overline{X}_n - \mu_X\|_{\Sigma^{-1}}^2.
\]
\end{proposition}

\begin{proof}
The triangle inequality implies that
\[
\mathcal{D}_n^{\Sigma} ~\le~ \|\Sigma^{-1/2}\widetilde{\Sigma}\Sigma^{-1/2} - I_d\|_{\mathrm{op}} + \|\Sigma^{-1/2}(\widehat{\Sigma}_n - \widetilde{\Sigma})\Sigma^{-1/2}\|_{\mathrm{op}}. 
\]
The definition of $\widetilde{\Sigma}$ yields.
\[
\|\Sigma^{-1/2}(\widehat{\Sigma}_n - \widetilde{\Sigma})\Sigma^{-1/2}\|_{\mathrm{op}} ~=~ \|\Sigma^{-1/2}(\overline{X}_n - \mu_X)\|_{I_d}^2.
\] 
This concludes the proof.
\end{proof}
\begin{lemma}\label{lemma:linear-expansion-inv-covariance}
Under the assumption that $\widehat{\Sigma}_n$ is invertible and $\mathcal{D}_n^{\Sigma} <1$,
\begin{equation}\label{eq:inv-covariance-error-bound}
\|\Sigma^{1/2}(\widehat{\Sigma}_n^{-1} - \Sigma^{-1})\Sigma^{1/2}\|_{\mathrm{op}} ~\le~ \frac{\mathcal{D}_n^{\Sigma}}{1 - \mathcal{D}_n^{\Sigma}}. 
\end{equation}
and
\begin{equation}\label{eq:final-linear-expansion-inv-covariance}
\left\|\Sigma^{1/2}\left\{\widehat{\Sigma}_n^{-1} - \Sigma^{-1} + \Sigma^{-1}(\widetilde{\Sigma} - \Sigma)\Sigma^{-1}\right\}\Sigma^{1/2}\right\|_{\mathrm{op}} \le \|\overline{X}_n - \mu_X\|^2_{\Sigma^{-1}} + \frac{(\mathcal{D}_n^{\Sigma})^2}{1 - \mathcal{D}_n^{\Sigma}}.
\end{equation}
\end{lemma}
\begin{proof}
We start with the following equality:
\begin{align*}
\widehat{\Sigma}_n^{-1} - \Sigma^{-1} ~&=~ \widehat{\Sigma}_n^{-1}(\Sigma - \widehat{\Sigma}_n)\Sigma^{-1}\\
~&=~ \Sigma^{-1}(\Sigma - \widehat{\Sigma}_n)\Sigma^{-1} + (\widehat{\Sigma}_n^{-1} - \Sigma^{-1})(\Sigma - \widehat{\Sigma}_n)\Sigma^{-1}\\
~&=~ \Sigma^{-1}(\Sigma - \widehat{\Sigma}_n)\Sigma^{-1}\\ &\qquad+ \widehat{\Sigma}_n^{-1}\Sigma^{1/2}(I_d - \Sigma^{-1/2}\widehat{\Sigma}_n\Sigma^{-1/2})(I_d - \Sigma^{-1/2}\widehat{\Sigma}_n\Sigma^{-1/2})\Sigma^{-1/2}.
\end{align*}
The first equality implies
\[
\|\Sigma^{1/2}(\widehat{\Sigma}_n^{-1} - \Sigma^{-1})\Sigma^{1/2}\|_{\mathrm{op}} \le \|\Sigma^{1/2}\widehat{\Sigma}_n^{-1}\Sigma^{1/2}\|_{\mathrm{op}}\|\Sigma^{-1/2}(\Sigma - \widehat{\Sigma}_n)\Sigma^{-1/2}\|_{\mathrm{op}},
\]
which proves~\eqref{eq:inv-covariance-error-bound}. The last equality above implies
\begin{equation}\label{eq:linear-expansion-partial-corr}
\begin{split}
&\left\|\Sigma^{1/2}\left\{\widehat{\Sigma}_n^{-1} - \Sigma^{-1} + \Sigma^{-1}(\widehat{\Sigma}_n - \Sigma)\Sigma^{-1}\right\}\Sigma^{1/2}\right\|_{\mathrm{op}}\\ ~&\le~ \|\Sigma^{1/2}\widehat{\Sigma}_n^{-1}\Sigma^{1/2}\|_{\mathrm{op}}\|\Sigma^{-1/2}\widehat{\Sigma}_n\Sigma^{-1/2} - I_d\|_{\mathrm{op}}^2\\
~&\le~ \frac{\|\Sigma^{-1/2}\widehat{\Sigma}_n\Sigma^{-1/2} - I_d\|_{\mathrm{op}}^2}{1 - \|\Sigma^{-1/2}\widehat{\Sigma}_n\Sigma^{-1/2} - I_d\|_{\mathrm{op}}}.
\end{split}
\end{equation}
assuming $\|\Sigma^{-1/2}\widehat{\Sigma}_n\Sigma^{-1/2} - I_d\|_{\mathrm{op}} < 1$ and using the fact that $\|A^{-1}\|_{\mathrm{op}} \le (1 - \|I - A\|_{\mathrm{op}})^{-1}$ whenever $\|I - A\|_{\mathrm{op}} < 1$.
This inequality almost proves a linear representation of $\widehat{\Sigma}_n^{-1} - \Sigma^{-1}$ except that $\widehat{\Sigma}_n - \Sigma$ is \emph{not} an average of independent random matrices. Using $\widetilde{\Sigma}$, we get
\[
\|\Sigma^{-1/2}(\widehat{\Sigma}_n - \widetilde{\Sigma})\Sigma^{-1/2}\|_{\mathrm{op}} ~=~ \|\Sigma^{-1/2}(\overline{X}_n - \mu_X)\|_{I_d}^2.
\] 
Combining this equality with~\eqref{eq:linear-expansion-partial-corr} concludes the proof.
\end{proof}

\section{Auxiliary Results}
\label{appendix:auxiliary}

The following result is an application of Theorem 3.1 in \cite{KuchAbhi17}.

\begin{lemma}\label{lemma:Thm3.1.KuchAbhi}
Suppose $W_1, \ldots, W_n\in\mathbb{R}^q$ are independent mean zero random vectors such that each of their coordinate is sub-Weibull$(\alpha)$, that is, $\|W_i(j)\|_{\alpha} \le K_w$ for $1\le j\le q$ and for 
some $\alpha\in (0, 1]$, , then for all $t\ge0$,
\[
\mathbb{P}\left(\max_{1\le j\le q}\left|\frac{1}{n}\sum_{i=1}^n W_i(j)\right| \ge CK_w\left\{\sqrt{\frac{t + \log q}{n}} + \frac{(t + \log q)^{1/\alpha}}{n}\right\}\right) \le 3e^{-t}.
\]
\end{lemma}
\begin{proof}
Theorem 3.1 and Proposition A.3 of~\cite{KuchAbhi17} jointly give that
\[
\mathbb{P}\left(\left|\frac{1}{n}\sum_{i=1}^n W_i(j)\right| \ge \frac{CK_w}{\sqrt{n}}\left(\sqrt{t} + \frac{t^{1/\alpha}}{\sqrt{n}}\right)\right) \le 3e^{-t},
\]
for all $t > 0$ and for some universal constant $C\in(0, \infty)$. The result now follows from a union bound. 
\end{proof}

The next bound is an application of Theorem 8 of~\cite{Bouch05}.

\begin{lemma}\label{lem:application-theorem8}
Suppose $W_1, \ldots, W_n$ are non-negative sub-Weibull$(1/2)$ random variables, that is, $\|W_i\|_{\psi_{1/2}} \le K_w < \infty$, then for all $t \ge 0$,
\begin{equation}\label{eq:to-be-proved-theorem-8}
\mathbb{P}\left(\frac{1}{n}\sum_{i=1}^n W_i \ge \frac{2}{n}\sum_{i=1}^n \mathbb{E}[W_i] + K_w\frac{t^3(\log n)^2}{n}\right) \le ee^{-t}.
\end{equation}
\end{lemma}
\begin{proof}
Theorem 8 of~\cite{Bouch05} implies
\begin{equation}\label{eq:main-implication-thereom8}
\left\|\frac{1}{n}\sum_{i=1}^n W_i\right\|_q \le \frac{2}{n}\sum_{i=1}^n \mathbb{E}[W_i] + \frac{2q}{n}\left\|\max_{1\le i\le n} W_i\right\|_q.
\end{equation}
(This follows by taking, following the notation of that paper, $\theta = 1$ and noting that $\kappa/2 < 1$). Because the $W_i$'s are sub-Weibull$(1/2)$, that is, $\mathbb{E}\left[\exp(\sqrt{|W_i|/K_w})\right] \le 2$,
\[
\mathbb{P}\left(W_i \ge K_wt^2\right) \le 2\exp\left(-t\right)\quad\mbox{for all}\quad t > 0.
\]
Hence by a union bound
\[
\mathbb{P}\left(\max_{1\le i\le n}W_i \ge K_w(t + \log n)^2\right) \le 2\exp(-t)\quad\mbox{for all}\quad t > 0.
\]
This yields
\[
\mathbb{P}\left((\max_{1\le i\le n} W_i - 2K_w\log^2n)_+ \ge 2K_wt^2\right) \le 2\exp(-t),
\]
or in other words, $(\max_{1\le i\le n}W_i - 2K_w\log^2n)_+$ is sub-Weibull$(1/2)$ with parameter $2K_w$. Therefore, for all $q \ge 1$,
\[
\left\|\max_{1\le i\le n}W_i\right\|_{q} \le 2K_w\log^2n + 2K_wq^2.
\]
Substituting this inequality in~\eqref{eq:main-implication-thereom8} yields that, for all $q \ge 1$,
\[
\left\|\frac{1}{n}\sum_{i=1}^n W_i\right\|_q \le \frac{2}{n}\sum_{i=1}^n \mathbb{E}[W_i] + \frac{4qK_w\log^2n}{n} + \frac{4q^3K_w}{n}.
\]
For any random variable $R$, Markov's inequality implies
\[
\mathbb{P}\left(R \ge e\|R\|_t\right) \le \frac{\|R\|_t^t}{e^t\|R\|_t^t} = e^{-t},\quad\mbox{for}\quad t \ge 1. 
\]
Therefore, for all $t\ge 1$,
\[
\mathbb{P}\left(\frac{1}{n}\sum_{i=1}^n W_i \le \frac{2e}{n}\sum_{i=1}^n \mathbb{E}[W_i] + \frac{4etK_w\log^2n}{n} + \frac{4et^3K_w}{n}\right) \le e^{-t}.
\]
To make this valid over all $t > 0$, we use the fact that probabilities are bounded by 1 and multiply the right hand side by $e$ so that for $t < 1$, $ee^{-t} > 1$. This completes the proof of~\eqref{eq:to-be-proved-theorem-8}.
\end{proof}

\end{appendices}

\end{document}